\RequirePackage{fix-cm}
\documentclass[smallextended]{svjour3}       
\smartqed  

\usepackage{braket,amsfonts}

\usepackage{array}

\usepackage[caption=false]{subfig}
\usepackage{pgfplots}

\usepackage{booktabs}
\usepackage{amsmath}
\usepackage{fancyhdr}
\usepackage{extramarks}
\usepackage{amsfonts}
\usepackage{tikz}
\usepackage{algorithm}
\usepackage[noend]{algpseudocode}
\usepackage{amssymb}
\usepackage{stmaryrd}
\usepackage{color}
\usepackage{verbatim}
\usepackage{hyperref}
\hypersetup{hidelinks}
\usepackage{caption}
\usepackage{float}
\usepackage{cases,mathdots,colordvi,appendix,url,multirow,lscape}
\usepackage{geometry}
\usepackage{setspace}
\usepackage{enumitem}

\usepackage{graphicx,epstopdf}

\usepackage{mathrsfs}

\usepackage{booktabs}
\usepackage{siunitx}
\makeatletter
\@addtoreset{equation}{section}
\makeatother

\newcommand{\betterbar}[1]{\overset{\rule[-0.1ex]{0.55em}{0.4pt}}{#1}}

\newcommand{\calJbar}{\makebox[0pt][l]{\rule[2ex]{0.85em}{0.4pt}}{\mathcal{J}}}

\DeclareMathOperator*{\argmin}{argmin}

\usepackage[normalem]{ulem}
\usepackage{xcolor}

\begin{document}
	\title{On some perturbation properties of nonsmooth optimization on Riemannian manifolds with applications}
	\author{Yuexin Zhou         \and
		Chao Ding      \and
		Yangjing Zhang 
	}
	\institute{Y.X. Zhou \at
		Engineering Systems and Design, Singapore University of Technology and Design, Singapore, Singapore.\\
        Part of this work was completed while the author was with the Institute of Applied Mathematics, Academy of Mathematics and Systems Science, Chinese Academy of Sciences, Beijing, P.R. China, School of Mathematical Sciences, University of Chinese Academy of Science, Beijing, P.R. China.\\ 
		\email{yuexin\_zhou@sutd.edu.sg}   
		\and
		C. Ding \at
		Institute of Applied Mathematics, Academy of Mathematics and Systems Science, Chinese Academy of Sciences, Beijing,  P.R. China.\\ 
		\email{dingchao@amss.ac.cn}
		\and
		Y.J. Zhang      \at
		Institute of Applied Mathematics, Academy of Mathematics and Systems Science, Chinese Academy of Sciences, Beijing,  P.R. China.\\ 
		\email{yangjing.zhang@amss.ac.cn}
	}
	
	\date{This version: Sep 30, 2025}

	\maketitle

\begin{abstract}

This paper presents a perturbation analysis framework for nonsmooth optimization on connected Riemannian manifolds to bridge the gap between the rapid development of algorithmic approaches and a robust theoretical foundation. Using tangent-space local models, we transport core notions from Euclidean variational analysis, such as strong regularity, the Aubin property, and isolated calmness of the Karush–Kuhn–Tucker (KKT) solution mapping, to the manifold setting. Furthermore, we introduce the manifold (strong) variational sufficiency and show that its strong version is intrinsic, i.e., independent of the chosen retraction, and for polyhedral, second-order cone, and semidefinite programs, it coincides with the manifold strong second-order sufficient condition. These insights yield concrete algorithmic consequences. We show that the Riemannian sequential quadratic programming achieves local superlinear and, under mild additional assumptions, quadratic convergence without strict complementarity, while the Riemannian Augmented Lagrangian Method attains R-linear convergence even when Lagrange multipliers are nonunique. Moreover, the proposed condition guarantees positive definiteness of the generalized Hessians associated with the augmented Lagrangian, enabling superlinear semismooth Newton steps in inner solves. Numerical experiments on robust matrix completion and compressed modes validate the theoretical predictions.

\end{abstract}

%
%

\section{Introduction}
Consider the following nonsmooth optimization problems on Riemannian manifolds 
\begin{equation}\label{eq:problem-mani}
\begin{array}{ll}
\min & f(x)+\theta(g(x)) \\
\text { s.t. }
& x\in \mathcal{M},
\end{array}
\end{equation}
where $\mathcal{M}$ is a connected Riemannian manifold, $f:\mathcal{M} \rightarrow \mathbb{R}$ and $g:\mathcal{M} \rightarrow \mathbb{Y}$ are continuously differentiable functions, $\mathbb{Y}$ is a Euclidean space equipped with a scalar
product $\langle\cdot, \cdot\rangle$ and its induced
norm $\|\cdot\|$,  $\theta : \mathbb{Y}\rightarrow (-\infty, \infty]$ is a proper closed convex function. If $\theta$ is an indicator function of a closed convex set, then (\ref{eq:problem-mani}) is a constrained manifold optimization problem.  Applications of (\ref{eq:problem-mani}) arise in various scenarios such as principal component analysis problems \cite{ZHT06}, low-rank 
matrix completion problems \cite{CA16}, orthogonal dictionary learning problems \cite{DH14,SWW12}, and compressed modes problems \cite{OLCO13}. We refer readers to \cite{AH19,HLWY20} for further details.

Various algorithms have been developed to solve the manifold nonsmooth optimization problem \eqref{eq:problem-mani}, including subgradient methods \cite{FO98,GH16}, proximal gradient methods \cite{CMMZ20,HW22,HW19}, alternating direction methods of multipliers 
\cite{KGB16,LO14}, and proximal point algorithm (PPA) \cite{CDMS21,FO02}. More recently, the augmented Lagrangian method {(ALM)} has been extended to nonsmooth manifold optimization in \cite{DP19,ZBDZ21}, further enriching the algorithmic landscape. In contrast with these advances on the computational side, the underlying theory remains relatively underdeveloped. In particular, fundamental aspects such as perturbation analysis have yet to be systematically explored for nonsmooth manifold problems, even though they are essential for understanding solution stability and for proving rigorous convergence guarantees. 

Perturbation analysis is a cornerstone of optimization theory, as it offers critical insights into the stability of optimal solutions under small changes in the problem data. In the Euclidean setting, this topic has been systematically developed within an even broader framework of variational analysis. Among these, a significant line of research in the literature focuses on the regularity properties of Karush-Kuhn-Tucker (KKT) solution mappings, which describe how the set of stationary points changes under perturbations. The comprehensive treatise by Dontchev and Rockafellar \cite{DR09} unifies several such properties: metric regularity, Aubin (pseudo-Lipschitz) property, calmness, isolated calmness, and strong regularity, each of which guarantees a Lipschitz-type response of the solution mapping to small data perturbations.

The regularity properties of KKT solution mappings have been characterized in various optimization settings. For nonlinear programming, strong regularity is characterized by the {strong second-order sufficient condition (SSOSC)}, together with the linear independence constraint qualification (LICQ), as shown in \cite{BS13,BS95,DR96}. In the context of conic optimization problems, strong regularity has similarly been characterized: for nonlinear second-order cone programs (NLSOC) and nonlinear semidefinite programming (NLSDP), \cite{BR05} and \cite{S06} established
that strong regularity is equivalent to the SSOSC combined with constraint nondegeneracy. Furthermore, for robust isolated calmness, \cite{DSZ17} showed 
that this property holds if and only if the second-order sufficient condition (SOSC) and the strict Robinson constraint qualification (SRCQ) are satisfied for $\mathcal{C}^2$-cone reducible cases. These regularity properties not only deepen our theoretical understanding but also play a crucial role in applications, such as establishing error bounds for optimization methods. Moreover, they are instrumental in analyzing the convergence rates of numerical algorithms, particularly in proving the local superlinear convergence of Newton-type methods and {ALM}; see, for instance, \cite{B94,KS19,LZ08,SSZ08}.

In the context of manifold optimization, perturbation analysis becomes more complicated due to the inherent geometric structure of the feasible set. Nevertheless, several recent works have begun to extend Euclidean theory to this setting. Hosseini and Pouryayevali \cite{HP13} provided 
a sufficient condition for the weak metric regularity of real-valued functions on complete Riemannian manifolds under strict differentiability, and, as a by-product, derived new necessary optimality conditions. More recently, Bao, Zhou, and Ding \cite{ZBD22} investigated the Riemannian robust isolated calmness of the KKT system and established its equivalence to the simultaneous satisfaction of the manifold strong Robinson constraint qualification (M-SRCQ) and the manifold second-order sufficient condition (M-SOSC). This result forms a crucial bridge between constraint qualifications and solution stability, closely mirroring the classical theory developed in Euclidean optimization. In this work, we study the perturbation properties of \eqref{eq:problem-mani}. We locally transform (\ref{eq:problem-mani}) into an optimization problem on the tangent space $T_{\bar{x}}\mathcal{M}$ at a point $\bar{x}$: 
	 \begin{equation}\label{eq:problem-tangent} 
	 	\begin{array}{ll}
	 		 \min & f\circ R_{\bar{x}}(\xi)+\theta(g\circ R_{\bar{x}}(\xi))\\
	 		 \text { s.t. } & \xi\in T_{\bar{x}}\mathcal{M}. 
	 		 \end{array} 
	 	\end{equation}
	  This problem is locally equivalent to (\ref{eq:problem-mani}) via the transformation $x=R_{\bar{x}}(\xi)$. Compared with (\ref{eq:problem-mani}), problem (\ref{eq:problem-tangent}) is more approachable because $T_{\bar{x}}\mathcal{M}$ is a Euclidean space. 
      This reformulation allows us to study the stability of the KKT system of the original Riemannian problem \eqref{eq:problem-mani} by examining the corresponding stability  of the KKT system 
      of \eqref{eq:problem-tangent}.

Beyond the stability of the KKT solution mapping, variational sufficiency (and its strong form) provides a complementary lens on perturbation behavior. It is especially valuable for analyzing ALM when multipliers may be non-unique or the problem is non-convex. The recent notion proposed by Rockafellar \cite{R22} of variational (strong) convexity lies at the heart of this approach: a function is variationally (strongly) convex around a reference point if, in that neighborhood, both its values and subdifferentials coincide with those of some (strongly) convex function. In the Euclidean setting, the (strong) variational sufficient condition demands that the perturbed augmented objective exhibit this property at a first-order stationary point. Rockafellar showed that strong variational sufficiency guarantees local (strong) convexity of the augmented Lagrangian and thereby underpins fast local convergence of ALM. Furthermore,
it is demonstrated that strong variational sufficiency is equivalent to the SSOSC if the function $\theta$ in \eqref{eq:problem-mani} (when the manifold $\mathcal{M}$ is taken as a Euclidean space) is a polyhedral function \cite{R22} or an indicator function of a second-order cone or semidefinite cone \cite{WDZZ22}. This condition also coincides with augmented tilt stability, which characterizes Lipschitz continuity of the solution set under linear perturbations.

Motivated by these developments, we aim to extend the (strong) variational sufficient condition to nonsmooth optimization problems on Riemannian manifolds. One challenge is formulating convexity on manifolds. Unlike in Euclidean space, a `line segment' between two points on a manifold cannot be expressed as a convex combination, which complicates directly extending convexity. To avoid defining convexity on the manifold itself, we examine the variational convexity of the tangent-space problem \eqref{eq:problem-tangent} at a stationary point $\bar{x}$, which is locally equivalent to the manifold problem \eqref{eq:problem-mani}. Since the tangent space $T_{\bar{x}}\mathcal{M}$ is a Euclidean space, this approach allows us to utilize tools from classical convex analysis. We define the manifold (strong) variational sufficient condition by requiring that the locally equivalent tangent problem satisfies the (strong) variational sufficient condition. Under appropriate assumptions, we show this condition is equivalent to the manifold strong second-order sufficient condition (M-SSOSC). Notably, this equivalence indicates that the concept is independent of the specific choice of retraction. Additionally, we demonstrate that the manifold strong variational sufficient condition guarantees local strong convexity of the augmented Lagrangian for the locally equivalent problem and implies augmented tilt stability. From the perspective of local convexity, variational sufficiency can be viewed as a specialized form of convexity for optimization problems. Since it is equivalent to augmented tilt stability, it maintains a close link to perturbation analysis. However, while the regularity properties of the KKT system discussed earlier involve both primal and dual variables, variational sufficiency concentrates solely on the behavior of the primal variables. Furthermore, this local convexity enables us to construct a local dual problem for manifold optimization, providing a new perspective on the convergence properties of manifold nonsmooth optimization algorithms.

After examining the perturbation characteristics of the KKT system and the manifold variational sufficient condition, we turn to their implications for 
the convergence analysis of manifold optimization algorithms. Specifically, we examine two recent methods: 
Riemannian sequential quadratic programming (RSQP) \cite{OOT22} and Riemannian Augmented Lagrangian Method (RALM) \cite{DP19,ZBDZ21}, investigating how the preceding theoretical results underpin their local convergence behavior. 
Assuming the robust isolated calmness of the KKT system, we demonstrate that RSQP achieves local superlinear (quadratic) convergence. For RALM, under the manifold strong variational sufficient condition, we establish R-linear convergence of the outer iterations and superlinear convergence of the semismooth Newton method used for the inner subproblems. These results not only extend classical Euclidean convergence rate analyses to the manifold setting but also highlight the pivotal role of our newly developed perturbation results in supporting algorithmic efficiency.

The rest of the paper is organized as follows. Section \ref{sec:preliminary} reviews the requisite background on Riemannian geometry and nonsmooth analysis. Section \ref{sec:regular} formulates the tangent-space surrogate problem of \eqref{eq:problem-mani} and establishes perturbation properties of its KKT system. In Section \ref{sec: variational-sufficiency}, we define and study the manifold variational sufficient condition, including its relations to convexity, to M-SSOSC, and to perturbation properties. Sections \ref{sec:rsqp} and \ref{sec:ralm} leverage these perturbation results for algorithmic analysis: we prove local superlinear convergence of RSQP under isolated calmness, and R-linear (outer) plus superlinear (inner) convergence of RALM under manifold strong variational sufficiency. Section \ref{sec:numerical} presents numerical experiments validating the theory, while Section \ref{sec:conclu} concludes and outlines directions for future research.

\section{Preliminaries and notations}\label{sec:preliminary}
This section recalls several notions of manifolds that will be used in our discussion. Most of the properties referred to below can be found in the literature, specifically in \cite{AMS09,L13}.
	
	Let $\mathcal{M}$ be an $n$-dimensional smooth manifold and $x \in \mathcal{M}$.  $\mathfrak{F}_{x}(\mathcal{M})$ is defined as the set of all smooth real-valued functions on a neighborhood of $x$. The mapping  $\xi_{x}$ from $\mathfrak{F}_{x}(\mathcal{M})$ to $\mathbb{R}$ such that there exists a curve $\gamma$ on $\mathcal{M}$ with $\gamma(0)=x$ satisfying $\xi_{x} f:=\dot{\gamma}(0)f:=\left.\dfrac{d(f(\gamma(t)))}{d t}\right|_{t=0}$ for all $f \in \mathfrak{F}_{x}(\mathcal{M})$ is called a tangent vector, and the tangent space  $T_x\mathcal{M}$ is the set of all tangent vectors to $\mathcal{M}$ at $x$. If $\mathcal{M}$ is embedded in a Euclidean space $\mathbb{X}$, the normal space $N_x\mathcal{M}$ is defined as the orthogonal complement of $T_x\mathcal{M}$ in $\mathbb{X}$. The tangent bundle is defined as $T\mathcal{M}:=\bigcup_{x} T_{x} \mathcal{M}$, which is the set of all tangent vectors to $\mathcal{M}$. A map $V : \mathcal{M} \rightarrow T \mathcal{M} $ is called a vector field on $\mathcal{ M}$ if $V(x) \in  T_{x}\mathcal{M}$ for all $x \in\mathcal{M}$.
	
	Let $F: \mathcal{M} \rightarrow \mathbb{X}$ be a smooth mapping between a manifold $\mathcal{M}$ and a Euclidean space $\mathbb{X}$. The mapping $DF(x): T_{x} \mathcal{M} \rightarrow T_{F(x)} \mathbb{X}$  which is defined by $\left(D F(x)\xi_{x}\right) f:=\xi_{x}(f \circ F)$ for $\xi_{x} \in T_{x} \mathcal{M}$ and $f \in \mathfrak{F}_{F(x)}(\mathbb{X})$, is a linear mapping called the differential of $F$ at $x$. If $\mathcal{M}$ is embedded in a Euclidean space, then $D F(x)$ reduces to the classical directional derivative, i.e., $\displaystyle D F(x)\xi_{x}=\lim _{t \rightarrow 0} \frac{F\left(x+t \xi_{x}\right)-F(x)}{t}$.	To distinguish it from the Euclidean differential, we use $h'(x)\xi$ to represent the traditional directional derivative in the direction $\xi$ and $\nabla h(x)$ to be the Euclidean gradient of $h$.
	
	A Riemannian metric $\langle\cdot, \cdot\rangle_{x}$ on $\mathcal{M}$ is a smoothly varying inner product with respect to $x$ defined for tangent vectors. A differentiable manifold whose tangent spaces are endowed with Riemannian metrics is called a Riemannian manifold. Whenever the point $x$ is clear from the context,  we suppress the subscript and simply write $\langle\cdot, \cdot\rangle$. The induced norm of this inner product is denoted by $\|\cdot\|$ with the subscript being omitted. Given $f \in \mathfrak{F}_{x}(\mathcal{M})$, the gradient of $f$ at $x$, denoted by $\operatorname{grad} f(x)$, is defined as the unique tangent vector that satisfying $	\langle\operatorname{grad} f(x), \xi\rangle:=\xi_{x} f$ for all $\xi \in T_{x} \mathcal{M}$.

	The length of a curve $\gamma: [a, b] \rightarrow \mathcal{M}$ on a Riemannian manifold $\mathcal{M}$ is defined by $
	L(\gamma) = \int_a^b \sqrt{\langle \dot{\gamma}(t), \dot{\gamma}(t) \rangle} \, \mathrm{d}t$, where $\langle \cdot, \cdot \rangle$ denotes the Riemannian metric on $\mathcal{M}$. The Riemannian distance between two points $y, z \in \mathcal{M}$ is given by $
	d(y, z) := \inf_{\gamma \in \Gamma} L(\gamma)
	$, where $\Gamma$ denotes the set of all smooth curves in $\mathcal{M}$ joining $y$ and $z$. The set $\{y \in \mathcal{M} \mid d(x, y) < \delta\}$ defines a neighborhood of $x$ with radius $\delta > 0$. We use $\operatorname{dist}(x,\mathcal{D})$ to denote the distance from a point $x\in\mathcal{M}$ to a set $\mathcal{D}\subseteq\mathcal{M}$. A geodesic is a curve on $\mathcal{M}$ that locally minimizes arc length. For every tangent vector $\xi \in T_x \mathcal{M}$, there exists an interval $\mathcal{I} \subseteq \mathbb{R}$ containing $0$ and a unique geodesic $\gamma(\cdot; x, \xi): \mathcal{I} \to \mathcal{M}$ satisfying $\gamma(0) = x$ and $\dot{\gamma}(0) = \xi$. The exponential mapping at $x \in \mathcal{M}$ is defined by $\operatorname{Exp}_x: T_x \mathcal{M} \to \mathcal{M}, \quad \xi \mapsto \operatorname{Exp}_x(\xi) := \gamma(1; x, \xi)
	$. A vector field $X$ along a smooth curve $\gamma$ is said to be parallel if $\nabla_{\dot{\gamma}} X = 0$, where $\nabla$ denotes the Riemannian connection on $\mathcal{M}$. Given a smooth curve $\gamma$ and a vector $\eta \in T_{\gamma(0)} \mathcal{M}$, there exists a unique parallel vector field $X_\eta$ along $\gamma$ such that $X_\eta(0) = \eta$. The parallel transport of $\eta$ along $\gamma$ from $0$ to $t$ is then defined as $
	P_\gamma^{0 \to t} \eta := X_\eta(t)$. When the geodesic from $p$ to $q$ is unique, denoted by $\gamma_{pq}$, we define	$P_{pq} := P_{\gamma_{pq}}^{0 \to 1}$.

Let $\mathrm{id}_{T_x\mathcal{M}}$ be the identity mapping on $T_x\mathcal{M}$. A retraction on a manifold $\mathcal{M}$ is a smooth mapping $R$ from the tangent bundle $T \mathcal{M}$ onto $\mathcal{M}$ satisfying $R_x\left(0_x\right)=x$ and $
D R_x\left(0_x\right)=\mathrm{id}_{T_x \mathcal{M}}$, where $R_x$ denotes the restriction of $R$ to $T_x \mathcal{M}$. The coordinate expression of $\operatorname{grad}  f(x)$ is given by
$
	D R_{x}(\xi)\operatorname{grad}  f(x)= \nabla  (f\circ R_{x})(\xi)$. The Riemannian Hessian of $f \in \mathfrak{F}_x(\mathcal{M})$ at a point $x$ in $\mathcal{M}$ is defined as the (symmetric) linear mapping $\operatorname{Hess}f(x)$ from $T_x \mathcal{M}$ into itself satisfying $\operatorname{Hess}f(x) \xi=\nabla_{\xi} \operatorname{grad} f(x)$ for all $\xi \in T_x \mathcal{M}$. By \cite[Proposition 5.5.6]{AMS09}, if $ \operatorname{grad}f(x)=0$, then
\begin{equation}\label{eq:hess-tang-mani}
	\operatorname{Hess} f(x)=\operatorname{Hess}\left(f \circ R_x\right)\left(0_x\right).
\end{equation}

Consider a subset $\mathcal{A} \subseteq \mathcal{M}$. The epigraph of a function $f: \mathcal{A} \rightarrow (-\infty,\infty]$ is defined as $ \operatorname{epi} f:=\{(x, \alpha) \in \mathcal{A} \times \mathbb{R} \mid f(x) \leq \alpha\} $. A proper function $f: \mathcal{A} \rightarrow (-\infty,\infty]$ is called lower semicontinuous (lsc) if $\operatorname{epi} f$ is closed. When considering the manifold $\mathcal{M}$ as a topological space, we can use the classical definition of lower semicontinuity of a function $f:\mathcal{M}\to (-\infty,\infty]$, that for any $x\in \mathcal{M}$, it holds that $\liminf _{y \rightarrow x} f(y) \geq f(x)$. It is not difficult to verify that this definition is equivalent to the epigraph definition if we consider $\mathcal{M}$ itself as a closed set of the topological space. The following definition of the Fr\'{e}chet subdifferential is taken from \cite[Corollary 4.5]{AFL05}.
	
	\begin{definition}\label{def:subdiff-frechet}
	 Let $f: \mathcal{M} \rightarrow (-\infty,\infty]$ be a function defined on a Riemannian manifold, and  $(\mathcal{U},\varphi)$ is a chart of $\mathcal{M}$. The Fr\'{e}chet subdifferential $\partial f(x)$ of $f$ at a point $x \in \operatorname{dom} f=\{x \in \mathcal{M}\mid f(x)<\infty\}$ is defined as
		$$
		\begin{aligned}
		\partial f(x) & =\big\{D \varphi(x)\zeta \mid \zeta \in \mathbb{R}^n, \liminf _{v \rightarrow 0} \frac{f \circ \varphi^{-1}(\varphi(x)+v)-f(x)-\langle\zeta, v\rangle}{\|v\|} \geq 0\big\} \\
		& =\big\{D \varphi(x)\zeta \mid \zeta \in \partial\left(f \circ \varphi^{-1}\right)(\varphi(x))\big\}.
		\end{aligned}
		$$
	\end{definition}
From this definition, one may deduce that $\partial f(x)=\partial\left(f \circ R _x\right)\left(0_x\right)$ for any given retraction $R_x$. 

With the Riemannian distance $d(\cdot,\cdot)$ defined above, the Lipschitz continuous property of functions can be extended to manifolds. A function $f:\mathcal{M}\to\mathbb{R}$ is Lipschitz continuous in a set $\mathcal{U}$ if for any $y, z \in \mathcal{U}$, $|f(y)-f(z)| \leq Ld(y, z)$. If there exists a neighborhood $\mathcal{U}$ of $x\in\mathcal{ M}$ such that $f$ is Lipschitz continuous on $\mathcal{U}$, we say that $f$ is locally Lipschitz continuous at $x$. The generalized directional derivative of a locally Lipschitz continuous function $f$ at $x \in \mathcal{M}$ in the direction $v \in T_{x} \mathcal{M}$ is defined in \cite{HP11} as
$$
	f^{\circ}(x ; v):=\limsup _{y \rightarrow x, t \downarrow 0} \dfrac{f \circ \varphi^{-1}(\varphi(y)+t D \varphi(x)v)-f \circ \varphi^{-1}(\varphi(y))}{t},
$$
	where $(\mathcal{U}, \varphi)$ is a chart containing $x$. The definition of the generalized directional derivative implies that $f^{\circ}(x,v)=\left(f\circ R_x\right)^{\circ}(0_x,v)$ for any retraction.
	The Clarke subdifferential of a locally Lipschitz continuous function $f$ at $x \in \mathcal{M}$, denoted by $\partial_{C} f(x)$, is defined as
$$
	\partial_{C} f(x)=\left\{\xi \in T_{x} \mathcal{M}\mid\langle\xi, v\rangle \leq f^{\circ}(x ; v) \text { for all } v \in T_{x} \mathcal{M}\right\}.
$$
	According to \cite[Proposition 2.5]{HP11}, an equivalence can be established between the Clarke subdifferential of $f$ and the  Clarke subdifferential of $f \circ R$ for any retraction $R$, as described below.
	
	\begin{proposition}\label{pro:subdiff-mani-tangent}
Let $\mathcal{ M }$ be a Riemannian manifold and $x \in \mathcal{ M }$. Suppose that $f: \mathcal{ M }\to \mathbb{R}$ is Lipschitz continuous near $x$ and $(\mathcal{U}, \varphi)$ is a chart at $x$. Then
$$
\partial_{C} f(x)=D \varphi(x)\left[\partial_{C}\left(f \circ \varphi^{-1}\right)(\varphi(x))\right].
$$
Therefore, we have $\partial_{C} f(x)=\partial_{C}\left(f \circ R _x\right)\left(0_x\right)$ for any retraction $R_x$.
	\end{proposition}

For a given function $\theta: \mathbb{Y} \rightarrow(-\infty,+\infty]$, the lower and upper directional epiderivatives (cf. e.g., \cite[(2.68) and (2.69)]{BS13}) of $\theta$ at $y \in \operatorname{dom} \theta$ in the direction $h \in \mathbb{Y}$ are defined as
$$
\theta_{-}^{\downarrow}(y ; h) :=\liminf _{t \downarrow 0,\, h^{\prime} \to h} \frac{\theta\left(y+t h^{\prime}\right)-\theta(y)}{t}, \quad 
\theta_{+}^{\downarrow}(y ; h) :=\sup _{\left\{t_{n}\right\} \in \Sigma}\left(\liminf _{n \to \infty,\, h^{\prime} \to h} \frac{\theta\left(y+t_{n} h^{\prime}\right)-\theta(y)}{t_{n}}\right),
$$
respectively, where $\Sigma$ is the set of all positive sequences $\left\{t_{n}\right\}$ converging to zero. Since $\theta$ is a proper closed convex function, it follows that $\theta_{-}^{\downarrow}(y ; \cdot)=\theta_{+}^{\downarrow}(y ; \cdot)$. If $\theta_{-}^{\downarrow}(y; h)$ is finite for $x \in \operatorname{dom} \theta$ and $h \in \mathbb{Y}$, the lower second order epiderivatives \cite[(2.76)]{BS13} for $w\in \mathbb{Y}$ is defined as:
$$
\theta_{-}^{\downarrow \downarrow}(y ; h, w):=\liminf _{t \downarrow 0 ,\, w^{\prime} \to w} \frac{\theta\left(y+t h+\frac{1}{2} t^{2} w^{\prime}\right)-\theta(y)-t \theta_{-}^{\downarrow}(y ; h)}{\frac{1}{2} t^{2}}.
$$

\section{Regularity properties of the KKT solution mapping}\label{sec:regular}
In this section, we study the regularity properties of the KKT solution mapping associated with problem~(\ref{eq:problem-mani}), to formulate intrinsic criteria for stability on Riemannian nonsmooth optimization problems. In the following two subsections, we first reformulate the manifold problem on the tangent space via a suitable retraction and develop the associated first- and second-order calculus, then employ this framework to characterize the desired regularity properties and error bounds for the KKT mapping. These results will lay the theoretical foundations required for the convergence results established in later sections.

\subsection{Perturbed problem and variational properties}\label{subsect:per-problem}

To analyze the regularity of the KKT solution mapping, we first establish a perturbation scheme for problem \eqref{eq:problem-mani} based on a retraction. Unlike the chart-based construction in \cite{ZBD22}, our retraction-centered framework is aligned directly with the operations used by manifold algorithms and thus offers both theoretical insights and algorithmic foundations.

For a given $x\in \mathcal{M} $, the inverse function theorem \cite[Theorem 4.5]{L13} guarantees that any retraction $R_x$  acts as a diffeomorphism in a neighborhood of the origin $0_x\in T_x\mathcal{M}$.  The following definition given in \cite{HU17} defines the injectivity radius of a Riemannian manifold with respect to retraction, which aligns with the definition of the classical injectivity radius \cite[Definition 10.19]{B20}, when the retraction is chosen as the exponential mapping.
	\begin{definition} \label{def:inject-radius}
		The injectivity radius of the manifold $\mathcal{M}$ at a point $x$ with respect to a retraction $R_x$, denoted by $r_{R_x}$, is defined as the supremum over all radii $r>0$ for which $R_x$ is well defined and acts as a diffeomorphism on the open ball $B_x(r)=\left\{v \in T_x \mathcal{M}\mid\|v\|<r\right\}$.
	\end{definition}
    By the inverse function theorem, we have $r_{R_x}>0$ for each given $R_x$. Additionally, as outlined in \cite[Proposition 10.22]{B20}, we have $d(x, y)=\|v\|$ within the ball $B_x(r_{\operatorname{Exp}_x})$ if $y=\operatorname{Exp}_x(v)$.  
	
To build up our local perturbation framework of (\ref{eq:problem-mani}), we first introduce the following function defined on the tangent space at a given point $x\in\mathcal{M}$. Let $x\in\mathcal{M}$ and $r_{R_x}$ be the injectivity radius of $\mathcal{M}$ at $x$ with respect to $R_x$. For a given function $f:\mathcal{M}\to(-\infty,\infty]$, we denote $f_{R_x}:T_x\mathcal{M}\to(-\infty,\infty]$ by
		$$
			f_{R_x}(\xi)=\left\{\begin{array}{ll}
				f\left(R_x (\xi)\right), & \xi \in B_x(r_{R_x}), \\
				+\infty, & \xi \notin B_x(r_{R_x}).
			\end{array}\right.
	$$
This function is the restriction of the pullback of $f$ at a given point $x$, as defined in \cite[Section~4]{AMS09}, where the general pullback function refers to $f\circ R$ at each point on the manifold. It is worth noting that $f_{R_x}$ inherit the lower semicontinuity of $f$ within the ball $B_x(r_{R_x})$, although such continuity may fail on its boundary. Fortunately, this limitation does not affect our analysis, as we are concerned only with the behavior of $f_{R_x}$ inside $B_x(r_{R_x})$.
	
We say that a point $\bar{x} \in \mathcal{M}$ is a KKT point of problem~\eqref{eq:problem-mani} if there exists a multiplier $\bar{y} \in \mathbb{Y}$ such that the pair $(\bar{x}, \bar{y})$ satisfies the following KKT conditions:
\begin{equation}\label{eq:man-kkt-origin}
		\operatorname{grad}_x L(x,y) =  \operatorname{grad}  f(x)+Dg(x)^*y =0, \quad
		y \in \partial \theta(g(x)),
\end{equation}
where 
\begin{equation}\label{eq:lag1}
    L(x,y)=f(x)+\langle g(x),y\rangle.
\end{equation}
 In what follows, we analyze the problem locally around the KKT point $\bar{x}$.  Let $R_{\bar{x}}$ be a retraction at $\bar{x}$ acting as a diffeomorphism on some neighborhood $\mathcal{U}\subseteq T_{\bar{x}}\mathcal{M}$. We define the perturbed  problem on $T_{\bar{x}}\mathcal{M}$ in a neighborhood of $\bar{x}$ by 
	\begin{equation}\label{eq:problem-rn-perturb}
		\begin{array}{ll}
			\min &  f_{R_{\bar{x}}}(\xi)+\theta(g_{R_{\bar{x}}}(\xi)+b)-\langle \hat{a}, \xi\rangle \\
			\text { s.t. } &\xi\in \mathcal{U},
		\end{array}
	\end{equation}
	where $(\hat{a},b)\in \mathbb{R}^n\times \mathbb{Y}$ is the perturbation parameter. Since $R_{\bar{x}}$ is a diffeomorphism, \eqref{eq:problem-rn-perturb} is locally equivalent to the following perturbed problem of (\ref{eq:problem-mani}) on $\mathcal{M}$: 
	\begin{equation}\label{eq:problem-mani-perturb}
		\begin{array}{ll}
			\min & f(x)+\theta(g(x)+b)-\langle \hat{a},R_{\bar{x}}^{-1}(x)\rangle\\
			\text { s.t. } 
			& x\in R_{\bar{x}}(\mathcal{U}).
		\end{array}
	\end{equation}
For any given $(\hat{a},b)\in T_{\bar{x}}\mathcal{M}\times\mathbb{Y}$, we define the KKT solution mapping of \eqref{eq:problem-rn-perturb} as
\begin{equation*}
    \widehat{\mathcal{S}}_{\mathrm{KKT}}(\hat{a}, b) =\left\{  (\xi,y)
    \mid \hat{a}=\nabla  f_{R_{\bar{x}}}(\xi)+g'_{R_{\bar{x}}}(\xi)^*y,\ y\in \partial \theta(g_{R_{\bar{x}}}(\xi)+b) \right\},
\end{equation*}
and we define the KKT solution mapping of \eqref{eq:problem-mani-perturb} as
\begin{equation}\label{eq:man-kkt-perturb2}
	 \mathcal{S}_{\mathrm{KKT}}(\hat{a}, b) =\left\{ 	(x,y)
     \mid	(DR_{\bar{x}}^{-1}(x))^*\hat{a}=\operatorname{grad}  f(x)+Dg(x)^*y,  \
			y\in \partial \theta(g(x)+b) \right\}.
\end{equation}	
Let $x=R_{\bar{x}}(\xi)$. By the chain rule, we have $ \nabla  f_{R_{\bar{x}}}(\xi)+g'_{R_{\bar{x}}}(\xi)^*y=DR_{\bar{x}}(\xi)^*(\operatorname{grad}  f(x)+Dg(x)^*y)$, and $(DR_{\bar{x}}(\xi)^*)^{-1}\hat{a}=(DR_{\bar{x}}^{-1}(x))^*\hat{a}$. 
It then follows that the KKT solution mappings  satisfy
    \begin{equation}\label{eq:eqv-skkt}
			\widehat{\mathcal{S}}_{\mathrm{KKT}}(\hat{a}, b)=\mu_{\bar{x}}(\mathcal{S}_{\mathrm{KKT}}(\hat{a}, b)\cap
			(R_{\bar{x}}(\mathcal{U})\times\mathbb{Y})) ,
	\end{equation}
where $\mu_{\bar{x}}:=(R_{\bar{x}}^{-1},\operatorname{id}_{\mathbb{Y}})$, and $\operatorname{id}_{\mathbb{Y}}$ denotes the identity mapping on $\mathbb{Y}$. The set of Lagrangian multipliers associated with \eqref{eq:problem-rn-perturb} at $(\xi,\hat{a},b)$ and with \eqref{eq:problem-mani-perturb} at $(x,\hat{a},b)$ are defined by
    	$$
	\widehat{M}(\xi,\hat{a}, b):=\big\{y \in\mathbb{Y}\mid(\xi, y) \in \widehat{\mathcal{S}}_{\mathrm{KKT}}(\hat{a}, b)\big\}, 
	\quad
    M(x, \hat{a}, b):=\big\{y\in \mathbb{Y} \mid(x,y) \in \mathcal{S}_{\mathrm{KKT}}(\hat{a}, b)\big\},
	$$
respectively. It follows from \eqref{eq:eqv-skkt} that $\widehat{M}(\xi,\hat{a}, b)=M(x, \hat{a}, b)$.

Let $(\hat{a},b)=(0,0)$. For problem \eqref{eq:problem-rn-perturb}
we denote 
\begin{equation}\label{eq:lag2}
L_{R_{\bar{x}}}(\xi, y) = f_{R_{\bar{x}}}(\xi) + \langle g_{R_{\bar{x}}}(\xi),y\rangle.
\end{equation}
The manifold Robinson constraint qualification (M-RCQ) is said to hold at a feasible solution $x$ of  problem (\ref{eq:problem-mani-perturb}) if
\begin{equation}\label{eq:man-rcq-perturb}
		Dg(x)T_{x}\mathcal{ M}+	\mathcal{T}_{\operatorname{dom}\theta}(g(x))=	\mathbb{Y},
\end{equation}
where $\mathcal{T}_{\operatorname{dom}\theta}(g(x))$ is the tangent cone for $g(x)$ and $\operatorname{dom}(\theta)$.  Define the critical cone of functions $\theta$ and $g$ at a KKT point $x$ and its corresponding multiplier $y\in M(x,0,0)$ by 
$$
	\mathcal{C}_{\theta,g}(x,y):=\left\{d \in \mathbb{Y} \mid \theta^{\downarrow}(g(x) ; d)=\langle d, y\rangle \right\}.
$$
The M-SRCQ is said to hold at a KKT point $x$ with respect to  $y\in M(x,0,0)$ if
\begin{equation}\label{eq:man-srcq-perturb}
	Dg(x)T_{x}\mathcal{ M}+	\mathcal{C}_{\theta,g}(x,y)=	\mathbb{Y}.
\end{equation}
The manifold constraint nondegeneracy is said to hold at $x$ if 
\begin{equation}\label{eq:man-cn-perturb}
	Dg(x)T_{x}\mathcal{ M}+	\operatorname{lin}(\mathcal{T}_{\operatorname{dom}\theta}(g(x)))=	\mathbb{Y} ,
\end{equation}
where $\operatorname{lin}(\mathcal{T}_{\operatorname{dom}\theta}(g(x)))$ is the lineality space of $\mathcal{T}_{\operatorname{dom}\theta}(g(x))$. If $\mathcal{M}$ is a Euclidean space $\mathbb{E}$, then the tangent space $T_{x} \mathcal{M}$ coincides with $\mathbb{E}$ itself (see \cite[Section 3.5.2]{AMS09}). In this case, the M-RCQ, M-SRCQ, and manifold constraint nondegeneracy conditions reduce to the standard constraint qualifications defined in Euclidean spaces (cf. e.g., \cite{BS13}).

The critical cone at a stationary point $x$ of problem (\ref{eq:problem-mani-perturb})  is
$$
	\mathcal{C}(x)=\left\{\xi \in T_{x}\mathcal{M} \mid  Df(x)\xi+(\theta\circ g)^{\circ}(x ; \xi) = 0\right\}.
$$
By the $\mathcal{C}^{2}$-cone reducibility of $\theta$, 
the M-SOSC at $x$  is said to hold if for any $\xi \in \mathcal{C}(x) \backslash\{0\}$,
\begin{equation}\label{eq:man-sosc-perturb}
	\sup _{y \in M(x,0, 0)}\left\{ \left\langle \xi, \operatorname{Hess}_x L(x,y)\xi\right\rangle-\psi_{(g(x),Dg(x)\xi)}^{*}(y)\right\} >0,
\end{equation}
where $\psi_{(g(x),Dg(x)\xi)}^{*}(\cdot)$ is the conjugate function of $\psi_{(g(x),Dg(x)\xi)}(\cdot)=\theta_{-}^{\downarrow\downarrow}(g(x) ; Dg(x)\xi, \cdot)$ for any $g(x) \in \operatorname{dom} \theta$ and any $\xi \in T_x\mathcal{M}$. The M-SSOSC at $(x,y)$ is defined as
$$
	\left\langle \xi, \operatorname{Hess}_x L(x,y)\xi\right\rangle-\psi_{(g(x),Dg(x)\xi)}^{*}(y) >0 \quad \forall\,  Dg(x)\xi \in \operatorname{aff}\mathcal{C}_{\theta,g}\left(x,y\right) \backslash\{0\},
$$
where  $\operatorname{aff}\mathcal{C}_{\theta, g}(x,y)$ is the affine hull of the critical cone $\mathcal{C}_{\theta, g}(x,y)$. Moreover, if $\theta$ is the indicator function of a $\mathcal{C}^2$-cone reducible set $\mathcal{K}$, then the M-SSOSC can be equivalently written as
	\begin{equation}\label{eq:ssosc-mani}
	\left\langle\xi, \operatorname{Hess}_x L(x,y ) \xi\right\rangle-\sigma\left(\bar{y}, \mathcal{T}_{\mathcal{K}}^2\left(g\left(x\right), D g\left(x\right) \xi\right)\right)>0 \quad \forall\,  Dg(x)\xi \in \operatorname{aff}\mathcal{C}_{\theta,g}\left(x,y\right) \backslash\{0\},
\end{equation}
where $\sigma(y, \mathcal{D})$ is the support function of the set $\mathcal{D}$ at $y$ and  $\mathcal{T}_{\mathcal{K}}^2(y,h)$ is the second-order tangent set of $\mathcal{K}$ at $y$ in direction $h$.

We now establish a precise correspondence between the Riemannian constraint qualifications and second-order optimality conditions of the original problem (\ref{eq:problem-mani-perturb}) and those of the locally equivalent problem (\ref{eq:problem-rn-perturb}) in the tangent space. The following proposition shows that these conditions are in fact equivalent in the two settings.

\begin{proposition}\label{pro:equiv-cq}
	Let $(\hat{a},b)=(0,0)$. The M-RCQ (respectively, M-SRCQ, manifold constraint nondegeneracy, M-SOSC, M-SSOSC) holds at $\bar{x}$ (with respect to $\bar{y}$) for problem (\ref{eq:problem-mani-perturb}) if and only if the RCQ (respectively, SRCQ, constraint nondegeneracy, SOSC, SSOSC)  holds at $0_{\bar{x}}$ (with respect to $\bar{y}$) for problem (\ref{eq:problem-rn-perturb}). 
\end{proposition}
\begin{proof}
The RCQ holds at $0_{\bar{x}}$ for problem (\ref{eq:problem-rn-perturb}) if
$$
(g\circ R_{\bar{x}})^{\prime}(0_{\bar{x}})T_{\bar{x}}\mathcal{M}+	\mathcal{T}_{\operatorname{dom}\theta}(g\circ R_{\bar{x}}(0_{\bar{x}}))=	\mathbb{Y} ,
$$
where for any $\xi\in T_{\bar{x}}\mathcal{M}$, 
$
(g\circ R_{\bar{x}})^{\prime}(0_{\bar{x}})\xi 
= DR_{\bar{x}}(0_{\bar{x}})[Dg(\bar{x})\xi] 
= Dg(\bar{x})\xi,
$
and $g\circ R_{\bar{x}}(0_{\bar{x}}) = g(\bar{x})$.
Therefore, M-RCQ and manifold constraint nondegeneracy hold at $\bar{x}$ for (\ref{eq:problem-mani-perturb}) if and only if RCQ and constraint nondegeneracy hold at $0_{\bar{x}}$ for (\ref{eq:problem-rn-perturb}). For the M-SRCQ, it suffices to show that $	\mathcal{C}_{\theta,g}(\bar{x},\bar{y})=	\mathcal{C}_{\theta,g\circ R_{\bar{x}}}(0_{\bar{x}},\bar{y})$,   
which follows immediately from the definition. 
By (\ref{eq:hess-tang-mani}), we have $
\operatorname{Hess}_x L(\bar{x},\bar{y})=\operatorname{Hess}_{\xi} L_{R_{\bar{x}}}(0_{\bar{x}},\bar{y})
$
at the stationary point $\bar{x}$, which shows the equivalence between SOSC (resp. SSOSC) for problem  (\ref{eq:problem-rn-perturb}) and M-SOSC (resp. M-SSOSC) for problem  (\ref{eq:problem-mani-perturb}). The proof is completed.\qed
\end{proof}

\begin{remark}
The definitions of M-RCQ, M-SRCQ, manifold constraint nondegeneracy, M-SOSC, and M-SSOSC for problem \eqref{eq:problem-mani-perturb} are independent of the choice of retraction at $\bar{x}$. Although the formulation of the locally equivalent problem \eqref{eq:problem-rn-perturb} in the tangent space involves on a specific  retraction $R_{\bar{x}}$, Proposition \ref{pro:equiv-cq} shows that the constraint qualifications and second-order optimality conditions  admit the 
Riemannian counterparts regardless of the retraction used. This confirms that these properties are intrinsically independent of $R_{\bar{x}}$.
\end{remark}

\subsection{Characterizations of the regularity properties}

In this subsection, we shall investigate the characterizations of regularity properties for set-valued mappings on Riemannian manifolds. Specifically, we extend the Euclidean concepts of the Aubin property, calmness, isolated calmness, and strong regularity, which are central tools for assessing solution sensitivity in optimization \cite[Chapter 3]{DR09}, to mappings whose domain and range lie on smooth manifolds. These properties are essential for understanding the stability of solutions. We first extend the notion of excess, which measures the distance between two sets in Euclidean space and is defined in \cite[Section 3A]{DR09}, to the Riemannian settings.

\begin{definition}\label{def:excess}
	For two given sets $C$ and $D$ in $\mathcal{M}$, the excess of $C$ beyond $D$ is defined by
	$$
	e(C, D)=\sup _{x \in C} \operatorname{dist}(x, D).
	$$

\end{definition}


Building on the definition of distance between two sets on a Riemannian manifold, we are able to extend the notions of regularity for set-valued mappings, including the Aubin property, calmness, isolated calmness, and strong regularity, within the context of Riemannian geometry.

\begin{definition}\label{def:aubin}
	Let $\mathbb{E}$ be a Euclidean space. A set-valued mapping $\Psi: \mathbb{E} \rightrightarrows \mathcal{M}$ is said to have the Aubin property at $\bar{p}$ for $\bar{x}$ if there exist a constant $\kappa>0$ and open neighborhoods $\mathcal{U}$ of $\bar{p}$ and $\mathcal{V}$ of $\bar{x}$ such that
	$$
		e(\Psi(p^{\prime})\cap \mathcal{V}, \Psi(p)) \leq \kappa\|p^{\prime}-p\|\quad
		\forall\,  \,p,p^{\prime}\in \mathcal{U}.
	$$
\end{definition}

\begin{definition}\label{def:calm}
	Let $\mathbb{E}$ be a Euclidean space. A set-valued mapping $\Psi: \mathbb{E} \rightrightarrows \mathcal{M}$ is said to be calm at $\bar{p}$ for $\bar{x}$ if there exist a constant $\kappa>0$ and open neighborhoods $\mathcal{U}$ of $\bar{p}$ and $\mathcal{V}$ of $\bar{x}$ such that
	$$
		e(\Psi(p)\cap \mathcal{V}, \Psi(\bar{p})) \leq \kappa\|p-\bar{p}\| \quad \forall\,  \,p\in \mathcal{U}.
	$$
\end{definition}

\begin{definition}\label{def:isolated-calm}
	Let $\mathbb{E}$ be a Euclidean space. A set-valued mapping $\Psi: \mathbb{E} \rightrightarrows \mathcal{M}$ is said to be isolated calm at $\bar{p}$ for $\bar{x}$ if there exist a constant $\kappa>0$ and open neighborhoods $\mathcal{U}$ of $\bar{p}$ and $\mathcal{V}$ of $\bar{x}$ such that
	\begin{equation}\label{eq:iso-calm}
		d(x,\bar{x})\leq \kappa\|p-\bar{p}\| \quad\forall\, \,x\in \Psi(p)\cap\mathcal{V} \quad \mbox{and} \quad p\in \mathcal{U}.
	\end{equation}
Moreover, $\Psi$ is said to be robustly isolated calm at $\bar{p}$ for $\bar{x}$ if (\ref{eq:iso-calm}) holds and, for each
$p \in \mathcal{U}$, $\Psi(p) \cap \mathcal{V} \neq \emptyset$.
\end{definition}

\begin{definition}\label{def:strong-regular}
		Let $\mathbb{E}$ be a Euclidean space. A set-valued mapping $\Psi: \mathbb{E}\rightrightarrows \mathcal{M}$ is said to be strongly regular at $\bar{p}$ for $\bar{x}$ if there exist open neighborhoods $\mathcal{U}$ of $\bar{p}$ and $\mathcal{V}$ of $\bar{x}$ such that         for any $p\in \mathcal{U}$, the set  $\Psi(p)\cap\mathcal{V}$ is  a singleton, and the mapping $\Psi(\cdot)\cap\mathcal{V}$ is Lipschitz continuous on $\mathcal{U}$. That is, there exists $L>0$ such that for any $p_1,p_2\in \mathcal{U}$,  
	 one has $d(x_1,x_2)\leq L\|p_1-p_2\|$ for $x_1 := \Psi(p_1)\cap\mathcal{V}$ and $x_2 := \Psi(p_2)\cap\mathcal{V}$.
\end{definition}



Notably, when the underlying manifold is structured as a Euclidean space, these definitions naturally reduce to their classical counterparts within the established framework of variational analysis \cite{DR09}, thus maintaining consistency with established regularity theory in the Euclidean setting.
Let $\bar{x}\in\mathcal{M}$ be a given point on the manifold. For any retraction $R_{\bar{x}}$ around $\bar{x}$, there exists an open subset $\mathcal{U}\subseteq T_{\bar{x}}\mathcal{M}$ on which $R_{\bar{x}}$ is a diffeomorphism.
Given a set-valued mapping $\Psi:\mathbb{E}\rightrightarrows \mathcal{M}$ from a Euclidean space $\mathbb{E}$ to the manifold $\mathcal{M}$, we define its  Euclidean counterpart $\widetilde{\Psi}(\cdot):\mathbb{E}\rightrightarrows T_{\bar{x}}\mathcal{M}$ under $R_{\bar{x}}$ by $\widetilde{\Psi}(\cdot):=R_{\bar{x}}^{-1}(\Psi(\cdot))\cap\mathcal{U}$. The next lemma establishes the relationship between the manifold 
mapping $\Psi$ and its Euclidean counterpart  $\widetilde{\Psi}$.
\begin{lemma}\label{lem:regular-eu-mani}
	Let $\Psi:\mathbb{E}\rightrightarrows \mathcal{M}$ be a set-valued mapping, and let $R_{\bar{x}}$ be a retraction around $\bar{x}\in \Psi(\bar{p})$. Assume that $R_{\bar{x}}$  is a diffeomorphism on $\mathcal{U}\subseteq T_{\bar{x}}\mathcal{M}$.  
	If the Euclidean counterpart $\widetilde{\Psi}:\mathbb{E}\rightrightarrows T_{\bar{x}}\mathcal{M}$ defined by $\widetilde{\Psi}(\cdot):=R_{\bar{x}}^{-1}(\Psi(\cdot))\cap\mathcal{U}$ is Aubin (calm, isolated calm, strongly regular) at $\bar{p}$ with respect to $0_{\bar{x}}$, then $\Psi$ is Aubin (calm, isolated calm, strongly regular) at $\bar{p}$ with respect to $\bar{x}$.
\end{lemma}
\begin{proof}
If $\widetilde{\Psi}$ is Aubin at $\bar{p}$ with respect to $0_{\bar{x}}$, then  there exist a constant $\kappa>0$ and open neighborhoods $\mathcal{P}$ of $\bar{p}$ and $\mathcal{V}$ of $\bar{x}$ such that for any $p,p^{\prime}\in \mathcal{P}$, 
	$
	e(\widetilde{\Psi}(p^{\prime})\cap \mathcal{V},\widetilde{\Psi}(p))\leq \kappa\|p^{\prime}-p\|,
	$
	where
    	$$
		e(\widetilde{\Psi}(p^{\prime})\cap \mathcal{V},\widetilde{\Psi}(p)) = \sup_{\xi^{\prime}\in \widetilde{\Psi}(p^{\prime})\cap \mathcal{V}} \ \inf_{\xi\in \widetilde{\Psi}(p)}\|\xi^{\prime}-\xi\|
        = \sup_{\xi^{\prime}\in R_{\bar{x}}^{-1}(\Psi(p^{\prime})\cap R_{\bar{x}}(\mathcal{V}))} \ \inf_{\xi\in R_{\bar{x}}^{-1}(\Psi(p))\cap\mathcal{U}}\|\xi^{\prime}-\xi\|.
	$$
Shrinking $\mathcal{U}$ if necessary, it follows from \cite[Corollary 5.7.11]{P06} that the Riemannian distance function $d:\mathcal{M}\times \mathcal{M}\to \mathbb{R}$ is smooth on $R_{\bar{x}}(\mathcal{U})$.  
Hence, there exists $\alpha>0$ such that for any $x,x^{\prime}\in R_{\bar{x}}(\mathcal{U})$, 
	$ d(x^{\prime},x)\leq \alpha\|R_{\bar{x}}^{-1}(x^{\prime})-R_{\bar{x}}^{-1}(x)\|. $
Therefore,
	$$
	\begin{aligned}
		e(\Psi(p^{\prime})\cap R_{\bar{x}}(\mathcal{V}),\Psi(p))&= \sup_{x^{\prime}\in \Psi(p^{\prime})\cap R_{\bar{x}}(\mathcal{V})} \ \inf_{x\in \Psi(p)}d(x',x)\\
		&\leq \sup_{x^{\prime}\in \Psi(p^{\prime})\cap R_{\bar{x}}(\mathcal{V})} \ \inf_{x\in \Psi(p)\cap R_{\bar{x}}(\mathcal{U})}d(x',x)\\
		& \leq \sup_{\xi^{\prime}\in R_{\bar{x}}^{-1}(\Psi(p^{\prime})\cap R_{\bar{x}}(\mathcal{V}))} \ \inf_{\xi\in R_{\bar{x}}^{-1}(\Psi(p))\cap\mathcal{U}}\alpha\|\xi^{\prime}-\xi\|\\
		&\leq \alpha \kappa \|p^{\prime}-p\|.
	\end{aligned}	
	$$
	This shows that $\Psi$ is Aubin at $\bar{p}$ with respect to $\bar{x}$. Similarly, if $\widetilde{\Psi}$ is calm or isolated calm at $\bar{p}$ with respect to $0_{\bar{x}}$, we can prove that $\Psi$ is  calm or isolated calm at $\bar{p}$ with respect to $\bar{x}$.
	
Finally, if  $\widetilde{\Psi}$ is strongly regular at $\bar{p}$ with respect to $0_{\bar{x}}$, then there exists a neighborhood $\mathcal{V}$ of $0_{\bar{x}}$ such that the Aubin property of $\widetilde{\Psi}$ about $0_{\bar{x}}$ holds, which means that $\Psi$ has Aubin property at $\bar{p}$ with $\bar{x}$. Since $\widetilde{\Psi}(p)\cap \mathcal{V}$ is a singleton for any $p\in\mathcal{P}$, where $\mathcal{P}$ is a neighborhood of $\bar{p}$, and $R_{\bar{x}}$ is a diffeomorphism over $\mathcal{U}$, it can be proved that $\Psi(p)\cap R_{\bar{x}}(\mathcal{V})$  is also a singleton. This means that $\Psi$ is strongly regular at $\bar{p}$ with respect to $\bar{x}$.\qed
\end{proof}


We now present the main results of this section, namely, the characterization of the regularity properties of the KKT solution mapping \eqref{eq:man-kkt-perturb2}. Most of the proofs are omitted, as many of the results follow directly from Proposition \ref{pro:equiv-cq}, Lemma \ref{lem:regular-eu-mani}, and the corresponding conclusions established in the Euclidean setting.
%
Define the natural mapping $ F: \mathcal{M} \times \mathbb{Y} \to T\mathcal{M} \times \mathbb{Y} $ for problem (\ref{eq:problem-mani-perturb}) by
\[
F(x, y) := \begin{bmatrix}
	\operatorname{grad} f(x) + Dg(x)^* y \\
	g(x) - \operatorname{prox}_\theta\left(g(x) + y\right)
\end{bmatrix}, \quad (x, y) \in \mathcal{M} \times \mathbb{Y},
\]
where 
$\operatorname{prox}_\theta(u) := \operatorname{argmin}_{y \in \mathbb{Y}} \left\{ \theta(y) + \frac{1}{2}\|u - y\|^2 \right\}$ denotes the proximal mapping of $\theta$.  
It is straightforward to verify that $(0, 0, \bar{x}, \bar{y}) \in \operatorname{gph} \mathcal{S}_{\mathrm{KKT}}$ if and only if $(0, 0, \bar{x}, \bar{y}) \in \operatorname{gph} F^{-1}$.  Let  $\partial_C F(\bar{x}, \bar{y})$ denote Clarke's generalized Jacobian of $ F $ at $(\bar{x}, \bar{y})$.
With this setup, we are now ready to establish the following proposition  on the nonsingularity of $\partial_C F(\bar{x}, \bar{y})$ for nonlinear programming (NLP), NLSOC, and NLSDP.

\begin{proposition}\label{pro:man-licq+ssosc-perturb}
Suppose that $\operatorname{epi} \theta$ is  $\mathcal{C}^2$-cone reducible in the sense of \cite[Definition 3.135]{BS13}.
Let $(\bar{x},\bar{y}) \in \mathcal{M}\times \mathbb{Y}$ be a KKT pair of problem~\eqref{eq:problem-mani} satisfying \eqref{eq:man-kkt-origin}. 
Assume that the M-RCQ~\eqref{eq:man-rcq-perturb} holds at $\bar{x}$. 
Then the following statement holds:
\begin{itemize}[leftmargin=2em]
	\item[\rm (i)] The M-SOSC \eqref{eq:man-sosc-perturb} holds at $\bar{x}$ and the M-SRCQ \eqref{eq:man-srcq-perturb} holds at $\bar{x}$ with respect to $\bar{y}$ if and only if the KKT solution mapping $\mathcal{S}_{\mathrm{KKT}}$ in \eqref{eq:man-kkt-perturb2} is robustly isolated calm at the origin for $(\bar{x}, \bar{y})$.
\end{itemize}
 Furthermore, if $\operatorname{epi} \theta$ is a convex polyhedral set, a second-order cone, or a semidefinite cone,
then the following statements are equivalent:
\begin{itemize}[leftmargin=2em]
	\item[\rm (ii)] The M-SSOSC holds at $(\bar{x}, \bar{y})$ and the manifold constraint nondegeneracy holds at $\bar{x}$;
	\item[\rm (iii)] The KKT solution mapping $\mathcal{S}_{\mathrm{KKT}}$ is strongly regular at the origin for $(\bar{x}, \bar{y})$;
	\item[\rm (iv)] All elements in  $\partial_C F(\bar{x}, \bar{y})$ are nonsingular.
\end{itemize}

	\end{proposition}

\begin{proof}
Statement (i) has already been established in \cite[Proposition 5]{ZBD22}. The equivalence between statements (ii) and (iii) follows from Proposition~\ref{pro:equiv-cq}, Lemma~\ref{lem:regular-eu-mani}, and the results in \cite{BR05,BS13,S06}. Hence, it remains to prove the equivalence between (iii) and (iv).
Define the natural mapping $\widehat{F}: T_{\bar{x}}\mathcal{M} \times \mathbb{Y} \to T_{\bar{x}}\mathcal{M} \times \mathbb{Y}$ for problem (\ref{eq:problem-rn-perturb}) by
	\[
	\widehat{F}(\xi, y) := \begin{bmatrix}
		\nabla f_{R_{\bar{x}}}(\xi) + g_{R_{\bar{x}}}^{\prime}(\xi)^* y \\
		g_{R_{\bar{x}}}(\xi) - \operatorname{prox}_\theta\left(g_{R_{\bar{x}}}(\xi) + y\right)
	\end{bmatrix}, 
    \quad (\xi, y) \in T_{\bar{x}}\mathcal{M} \times \mathbb{Y}.
	\]
	For the cases of NLP, NLSOC, and NLSDP, it follows from \cite[Proposition 5.38]{BS13}, \cite[Theorem 30]{BR05}, and \cite[Theorem 4.1]{S06} that every element in $\partial_C \widehat{F}(0_{\bar{x}}, \bar{y})$ is nonsingular if and only if (iii) holds. It remains to show $\partial_C F(\bar{x}, \bar{y}) = \partial_C \widehat{F}(0_{\bar{x}}, \bar{y})$.
    For $x=R_{\bar{x}}(\xi)$, by the chain rule, we have
		\[\widehat{F}(\xi, y) =
    \begin{bmatrix}
        DR_{\bar{x}}(\xi)^*(\operatorname{grad}  f(x)+Dg(x)^*y) \\
        g_{R_{\bar{x}}}(\xi) - \operatorname{prox}_\theta\left(g_{R_{\bar{x}}}(\xi) + y\right)
    \end{bmatrix}.
	\]
The first components of $F$ and $\widehat{F}$ are equal by applying \cite[Proposition 5.5.6]{AMS09} at the KKT pair $(\bar{x},\bar{y})$.
Define $h(x,y) := g(x) - \operatorname{prox}_\theta\left(g(x)+y\right)$ and $\widehat{h}(\xi,y) := g_{R_{\bar{x}}}(\xi) - \operatorname{prox}_\theta\left(g_{R_{\bar{x}}}(\xi)+y\right)$.
Clearly, $h(x,y) = \widehat{h}(R_{\bar{x}}^{-1}(x), y)$. It follows from \cite[Proposition 1]{ZBD22} that
	\[
	\partial_C h(x, y) \subseteq \{(DR_{\bar{x}}^{-1}(x)^*[u], w) \mid (u, w) \in \partial_C \widehat{h}(R_{\bar{x}}^{-1}(x), y)\}.
	\]
	Moreover, since $DR_{\bar{x}}^{-1}(x)^*[u] = u$ at $\bar{x}$ for any $u$, an extension of \cite[Theorem 2.3.10]{C90} to the Riemannian setting gives
	$
	\partial_C h(\bar{x}, \bar{y}) = \partial_C \widehat{h}(0_{\bar{x}}, \bar{y}).
	$
	Combining the above, we obtain $\partial_C F(\bar{x}, \bar{y}) = \partial_C \widehat{F}(0_{\bar{x}}, \bar{y})$.  
	This completes the proof.
	\qed
\end{proof}

Recently, new results concerning the equivalence between the Aubin property and strong regularity for the KKT solution mappings of NLSOC and NLSDP have been established in \cite{CCSZ241,CCSZ24}. By combining these findings with the classical equivalence results for convex polyhedral programming from \cite{DR96}, we obtain the following proposition.
	
	\begin{proposition}\label{pro:man-strong-reg-aubin}
Suppose that $\operatorname{epi} \theta$ is  a convex polyhedral set, a second-order cone, or a semidefinite cone.
Let $(\bar{x},\bar{y}) \in \mathcal{M}\times \mathbb{Y}$ be a KKT pair of problem~\eqref{eq:problem-mani} satisfying \eqref{eq:man-kkt-origin}.     
Assume that the M-RCQ (\ref{eq:man-rcq-perturb}) holds at $\bar{x}$. 
Then the KKT solution mapping $\mathcal{S}_{\mathrm{KKT}}$ in \eqref{eq:man-kkt-perturb2} has the Aubin property at the origin for $(\bar{x},\bar{y})$ if and only if it is strongly regular at the origin for $(\bar{x},\bar{y})$.
	\end{proposition}

Beyond the above equivalence results, the robust isolated calmness or strong regularity of the KKT solution mapping guarantees a two-sided error bound in a neighborhood of the KKT pair $(\bar{x}, \bar{y})$.  This two-sided error bound plays a crucial role in establishing the local linear convergence of the Riemannian augmented Lagrangian method (RALM), as we will show in Section \ref{sec:ralm}. To formalize this, define the KKT residual mapping $E: \mathcal{M} \times \mathbb{Y} \to \mathbb{R}$ by
\begin{equation}\label{kktresidual}
E(x, y) := \left\|\operatorname{grad} f(x) + Dg(x)^* y 
\right\| + \left\| g(x) - \operatorname{prox}_\theta\left( g(x) + y \right) \right\|, \quad (x, y) \in \mathcal{M} \times \mathbb{Y}.
\end{equation}

\begin{proposition}\label{pro:errorbound}
     Let $(\bar{x},\bar{y}) \in \mathcal{M}\times \mathbb{Y}$ be a KKT pair of problem~\eqref{eq:problem-mani} satisfying \eqref{eq:man-kkt-origin}.
     Suppose that the KKT solution mapping $\mathcal{S}_{\mathrm{KKT}}$ in \eqref{eq:man-kkt-perturb2} is robustly isolated calm (or strongly regular) at $(\bar{x}, \bar{y})$. Then $M(\bar{x}, 0, 0) = \{\bar{y}\}$. Furthermore, there exist constants $c_1, c_2 > 0$ such that, for any $(x, y) \in \mathcal{M} \times \mathbb{Y}$ sufficiently close to $(\bar{x}, \bar{y})$ with $E(x, y)$ sufficiently small, it holds that
	$$
		 c_{1} E(x, y) \leq d(x, \bar{x}) + \|y - \bar{y}\| \leq c_2 E(x, y).
		$$
	 \end{proposition}

 \begin{proof}
	 The conclusion follows by combining Proposition~\ref{pro:man-licq+ssosc-perturb}, \cite[Proposition 5]{ZBD22}, \cite[Theorem 7]{ZBD22}, and the fact that strong regularity implies isolated calmness. 
     \qed
	 \end{proof}
	
\begin{remark}
In \cite{ZBD22}, the results in Proposition~\ref{pro:errorbound}, as well as the equivalence between M-SOSC with M-SRCQ and robust isolated calmness, were established using the normal coordinate chart around the KKT point.In contrast, this paper derives these results via the retraction mapping at the KKT  point $\bar{x}$.
These two approaches are essentially equivalent, since there exists a local diffeomorphism between the normal coordinate chart and the exponential map around $\bar{x}$, and the retraction serves as a first-order approximation of the exponential map.
\end{remark}

\section{Manifold variational sufficiency}\label{sec: variational-sufficiency}

In the previous section, we study the stability of the KKT system of Riemannian nonsmooth optimization problem \eqref{eq:problem-mani}, which involves both the primal variable and its multipliers. These properties offer insights into local convergence, in particular, the local convergence of Newton-type methods (e.g., RSQP), and will be analyzed in detail in subsequent sections. In this section, we focus on another concept related to perturbation analysis, (strong) variational sufficiency, which is recently introduced in \cite{R22}. In Euclidean setting, the strong variational sufficiency is equivalent to the augmented tilt stability \cite{R22}, and serves as a key tool for analyzing the local convergence rate of the ALM in situations where standard assumptions like local convexity or uniqueness of Lagrange multipliers do not hold \cite{R22-ext}. In this section, we shall extend this concept to the Riemannian setting and examine how manifold (strong) variational sufficiency relates to optimality conditions and perturbation analysis.  Its significance in the convergence analysis of RALM will be demonstrated in a later section.

\subsection{Augmented problem and its first-order optimality conditions}
As in Section \ref{subsect:per-problem}, given a KKT point $\bar{x}\in \mathcal{M}$, let $R_{\bar{x}}$ be a retraction at $\bar{x}$ and a diffeomorphism in $\mathcal{U}\subseteq T_{\bar{x}}\mathcal{M}$. Let $\hat{a}=0$ in this section. Consider the following augmented problem for 
    \eqref{eq:problem-mani-perturb}:
	\begin{equation}\label{eq:problem-mani-perturb-augment}
		\begin{array}{ll}
			\min & f(x)+\theta(g(x)+b)+ \frac{\rho}{2}\|b\|^2\\
			\text { s.t. }
			& x\in R_{\bar{x}}(\mathcal{U}).
		\end{array}
	\end{equation}
Locally,  \eqref{eq:problem-mani-perturb-augment} is equivalent to the following augmented problem for \eqref{eq:problem-rn-perturb}:
\begin{equation}\label{eq:problem-rn-perturb-augment}
		\begin{array}{ll}
			\min &  f_{R_{\bar{x}}}(\xi)+\theta(g_{R_{\bar{x}}}(\xi)+b)+ \frac{\rho}{2}\|b\|^2\\
			\text { s.t. } &\xi\in \mathcal{U}.
		\end{array}
\end{equation}
Denote the objective functions by 
     \begin{equation}\label{def:orgobj}
     \varphi(x,b)= f(x)+\theta(g(x)+b),\quad \varphi_{R_{\bar{x}}}(\xi,b)= f_{R_{\bar{x}}}(\xi)+\theta(g_{R_{\bar{x}}}(\xi)+b). 
     \end{equation}
Then the Lagrangian functions for problems \eqref{eq:problem-mani-perturb} and \eqref{eq:problem-rn-perturb} are given, respectively, by 
     \begin{align*}
         l(x, y) & =\inf _b\{\varphi(x, b)-\left\langle y ,b\right\rangle \}=L(x , y)-\theta^*(y), \\
         l_{R_{\bar{x}}}(\xi, y) & =\inf _b\{\varphi_{R_{\bar{x}}}(\xi, b)-\left\langle y ,b\right\rangle \}=L_{R_{\bar{x}}}(\xi , y)-\theta^*(y),
     \end{align*}
where $L$ and $L_{R_{\bar{x}}}$ are defined in \eqref{eq:lag1} and \eqref{eq:lag2}, respectively. 
For $\rho > 0$, the augmented Lagrangian function of problems (\ref{eq:problem-mani-perturb}) and (\ref{eq:problem-rn-perturb}) are given, respectively, by
    \begin{equation}\label{def:auglag}
        l^{\rho}(x, y) =\inf_b\big\{\varphi(x, b)-\left\langle y ,b\right\rangle +\frac{\rho}{2}\|b\|^2\big\}, \quad
        l_{R_{\bar{x}}}^{\rho}(\xi, y) =\inf_b\big\{\varphi_{R_{\bar{x}}}(\xi, b)-\left\langle y ,b\right\rangle +\frac{\rho}{2}\|b\|^2\big\}. 
    \end{equation}
The augmented Lagrangian functions can be regarded as the Lagrangian functions of the augmented objective functions 
    \begin{equation}\label{def:augobj}
    \varphi^{\rho}(x, b)=\varphi(x, b)+\frac{\rho}{2}\|b\|^2,\quad \varphi_{R_{\bar{x}}}^{\rho}(\xi, b)=\varphi_{R_{\bar{x}}}(\xi, b)+\frac{\rho}{2}\|b\|^2, 
    \end{equation}
    respectively. Thus, by definition, $-l(x,\cdot)$ and $-l_{R_{\bar{x}}}(\xi,\cdot)$ are the conjugate functions of $\varphi(x, \cdot)$ and $\varphi_{R_{\bar{x}}}(\xi, \cdot)$, respectively, and similarly, $-l^{\rho}(x,\cdot)$ and $-l_{R_{\bar{x}}}^{\rho}(\xi,\cdot)$ are the conjugate functions of $\varphi^{\rho}(x, \cdot)$ and $\varphi_{R_{\bar{x}}}^{\rho}(\xi, \cdot)$, respectively. 
    Since $\theta$ is closed proper convex, by \cite[Theorem~12.2]{R70}, 
	\begin{equation}\label{eq:conjugate}
	\begin{array}{lcc}
		&\displaystyle\varphi(x, b)=\sup _y\{l(x, y)+\left\langle y ,b\right\rangle\}, \quad&\displaystyle\varphi^{\rho}(x, b)=\sup _y\left\{l^{\rho}(x, y)+\left\langle y ,b\right\rangle\right\},\\
         &\displaystyle \quad\varphi_{R_{\bar{x}}}(\xi, b)=\sup _y\{l_{R_{\bar{x}}}(\xi, y)+\left\langle y ,b\right\rangle\}, &\displaystyle\quad\varphi_{R_{\bar{x}}}^{\rho}(\xi, b)=\sup _y\{l_{R_{\bar{x}}}^{\rho}(\xi, y)+\left\langle y ,b\right\rangle\}.
	\end{array}
\end{equation}


    We now examine the relationships between the first-order  conditions of problems \eqref{eq:problem-mani-perturb} and \eqref{eq:problem-rn-perturb}, together with those of their augmented counterparts \eqref{eq:problem-mani-perturb-augment} and \eqref{eq:problem-rn-perturb-augment}. 
    In fact, the first-order condition in Proposition~\ref{pro:first-order-cond} is also equivalent to the KKT conditions \eqref{eq:man-kkt-origin}.
    

\begin{proposition}\label{pro:first-order-cond}
Let $\rho>0$, $(\bar{x},\bar{y})$ be given, and let $R_{\bar{x}}$ be a retraction at $\bar{x}$.   The following statements are equivalent:
\begin{itemize}[leftmargin=2em]
			\item [\rm (i)] $(\bar{x},\bar{y})$ satisfies the first-order condition  of problem \eqref{eq:problem-mani-perturb}, i.e., 
            $(0, \bar{y}) \in \partial \varphi(\bar{x}, 0)$;
            
			\vspace{1mm}
			\item [\rm (ii)] $(0_{\bar{x}},\bar{y})$ satisfies the first-order condition of problem \eqref{eq:problem-rn-perturb}, i.e., 
            $(0, \bar{y}) \in \partial \varphi_{R_{\bar{x}}}(0_{\bar{x}}, 0)$;
            
            \vspace{1mm}
            \item [\rm (iii)] $(\bar{x},\bar{y})$ satisfies  the first-order condition  of problem \eqref{eq:problem-mani-perturb-augment}, i.e., 
            $(0, \bar{y}) \in \partial \varphi^{\rho}(\bar{x}, 0)$;

            \vspace{1mm}
			\item [\rm (iv)] $(0_{\bar{x}},\bar{y})$ satisfies the first-order condition of problem \eqref{eq:problem-rn-perturb-augment}, i.e., 
            $(0_{\bar{x}},\bar{y})\in \partial \varphi_{R_{\bar{x}}}^{\rho}(0_{\bar{x}}, 0)$;
            
            \vspace{1mm}
			\item [\rm (v)] $\operatorname{grad}_x l(\bar{x},\bar{y})=0$, $0\in\partial_y [-l](\bar{x},\bar{y})$, or $\operatorname{grad}_xL(\bar{x},\bar{y})=0$, $\bar{y}\in\partial \theta(g(\bar{x}))$;

            \vspace{1mm}
			\item [\rm (vi)] $\nabla_{\xi}l_{R_{\bar{x}}}(0_{\bar{x}},\bar{y})=0$, $0\in\partial_y [-l_{R_{\bar{x}}}](0_{\bar{x}},\bar{y})$, or $\nabla_{\xi}L_{R_{\bar{x}}}(0_{\bar{x}},\bar{y})=0$, $\bar{y}\in\partial \theta(g_{R_{\bar{x}}}(0_{\bar{x}}))$;

            \vspace{1mm}
			\item [\rm (vii)] $\operatorname{grad}_xl^{\rho}(\bar{x},\bar{y})=0$, $0\in\nabla_y l^{\rho}(\bar{x},\bar{y})$, or $\operatorname{grad}_xL(\bar{x},\bar{y})=0$, $\nabla \operatorname{env}_{\rho}\theta(g(\bar{x})+\rho^{-1}\bar{y})=\bar{y}$, where $\operatorname{env}_{\rho}\theta$ is the Moreau envelope of $\theta$ defined by 
			$\operatorname{env}_{\rho}\theta(p):=\min_{y \in \mathbb{Y}} \theta(y)+\dfrac{\rho}{2}\|p-y\|^2$; 

            \vspace{1mm}
			\item [\rm (viii)] $\nabla_{\xi}l_{R_{\bar{x}}}^{\rho}(0_{\bar{x}},\bar{y})=0$, $0\in\nabla_y  l_{R_{\bar{x}}}^{\rho}(0_{\bar{x}},\bar{y})$, or $\nabla_{\xi}L_{R_{\bar{x}}}(0_{\bar{x}},\bar{y})=0$, $\nabla \operatorname{env}_{\rho}\theta(g_{R_{\bar{x}}}(0_{\bar{x}})+\rho^{-1}\bar{y})=\bar{y}$.
		\end{itemize}
	\end{proposition}
	\begin{proof}
		It follows from Proposition~\ref{pro:subdiff-mani-tangent} that  $(0, \bar{y}) \in \partial \varphi(\bar{x}, 0) \Longleftrightarrow (0, \bar{y}) \in \partial \varphi_{R_{\bar{x}}}(0_{\bar{x}}, 0)$. Moreover, when $b=0$, it is obvious that $\partial \varphi(\bar{x}, 0)=\partial \varphi^{\rho}(\bar{x}, 0)$ and $\partial \varphi_{R_{\bar{x}}}(0_{\bar{x}}, 0) = \partial \varphi_{R_{\bar{x}}}^{\rho}(0_{\bar{x}}, 0)$, which implies the equivalence of (i), (ii), (iii), and (iv). 
        By \eqref{eq:conjugate} and the chain rules of the subdifferential of manifold functions \cite[Proposition 1]{ZBD22}, we have
        $$(0_{\bar{x}}, \bar{y}) \in \partial \varphi(\bar{x}, 0) \Longleftrightarrow \bar{y} \in \partial \theta(g(\bar{x})), \; 0_{\bar{x}}=\operatorname{grad}_x L(\bar{x}, \bar{y}) 
			 \Longleftrightarrow 0_{\bar{x}} \in \partial_x l(\bar{x}, \bar{y}), \; 0\in \partial_y[-l](\bar{x}, \bar{y}).$$
		Moreover, Proposition \ref{pro:subdiff-mani-tangent} yields
		$$
		\begin{aligned}
			(0_{\bar{x}}, \bar{y}) \in \partial \varphi(\bar{x}, 0) \Longleftrightarrow(0_{\bar{x}}, \bar{y}) \in \partial \varphi_{R_{\bar{x}}}(0_{\bar{x}}, 0)  & \Longleftrightarrow \bar{y} \in \partial \theta(g_{R_{\bar{x}}}(0_{\bar{x}})), \; 0_{\bar{x}}=\nabla_{\xi} l_{R_{\bar{x}}}(0_{\bar{x}}, \bar{y}) \\
			& \Longleftrightarrow 0_{\bar{x}} \in \partial_\xi l_{R_{\bar{x}}}(0_{\bar{x}}, \bar{y}), \; 0 \in \partial_y[-l_{R_{\bar{x}}}](0_{\bar{x}}, \bar{y}).
		\end{aligned}
		$$
		Therefore, statements (i), (ii), (v), and (vi) are equivalent. Similarly, we can prove the equivalence between (i), (ii), (vii) and (viii).\qed
	\end{proof}

\subsection{Characterizations of manifold variational sufficiency}
We now turn to the notion of manifold variational sufficiency.  
Before that, let us first recall the concept of variational convexity in the Euclidean case, as  introduced in \cite{R22}. 
     	Let $f: \mathbb{R}^n \rightarrow(-\infty, \infty]$ be a lsc function. We say that $f$ is  variationally  convex with respect to  $(\bar{w}, \bar{z}) \in \operatorname{gph}\partial f$  if there exist convex open  neighborhoods $\mathcal{W}$ of $\bar{w}$ and $\mathcal{Z}$ of $\bar{z}$, together with a proper lsc convex function $h$ satisfying $h \leq f$ on $\mathcal{W}$, such that
		$$
        [\mathcal{W} \times \mathcal{Z}] \cap \operatorname{gph} \partial h=[\mathcal{W} \times \mathcal{Z}] \cap \operatorname{gph} \partial f,
        $$
		and $h(w)=f(w)$ holds for all $(w, z)$ in this common set. If, in addition, $h$ is strongly convex on $\mathcal{W}$, then we say that $f$ is variationally strongly convex with respect to $(\bar{w},\bar{z})$.   
In the Euclidean setting, the (strong) variational sufficient condition for 
problem $\min\{ \varphi(x,b) \mid b=0\}$ is said to hold with respect to $(\bar{x},\bar{y})$ satisfying the first-order condition $(0, \bar{y}) \in \partial \varphi(\bar{x}, 0)$  at level $\rho>0$  if 
the augmented objective $\varphi^{\rho}$ is variationally (strongly) convex with respect to $((\bar{x},0),(0,\bar{y})) \in \operatorname{gph} \partial \varphi^{\rho}$. Here, $\varphi$ and $\varphi^{\rho}$ are defined in \eqref{def:orgobj} and \eqref{def:augobj}, respectively.

We define the manifold (strong) variational sufficiency for problem~\eqref{eq:problem-mani} through the (strong) variational sufficiency for its locally equivalent formulation \eqref{eq:problem-tangent} on the tangent space, which is problem~\eqref{eq:problem-rn-perturb-augment} with $b=0$.  
This definition is then based on the variational (strong) convexity of the objective $\varphi^{\rho}_{R_{\bar{x}}}$ in problem~\eqref{eq:problem-rn-perturb-augment}, and therefore depends on the choice of the retraction $R_{\bar{x}}$.  
We will later establish, however, that the notion of manifold (strong) variational sufficiency is in fact independent of the particular retraction selected.
 
	\begin{definition}\label{def:vari-suff-mani}
    Let $(\bar{x},\bar{y}) \in \mathcal{M}\times \mathbb{Y}$ be a KKT pair of problem~\eqref{eq:problem-mani} satisfying \eqref{eq:man-kkt-origin}. 
    With respect to $(\bar{x},\bar{y})$ at level $\rho>0$, the manifold (strong) variational sufficient condition for  problem (\ref{eq:problem-mani}) under the retraction $R_{\bar{x}}$ is said to hold  if the (strong) variational sufficient condition holds for problem (\ref{eq:problem-tangent}).
	\end{definition}

Motivated by the recent developments in \cite{R22,R22-ext}, we next discuss the local augmented dual property for problem \eqref{eq:problem-mani}. In particular, the following proposition establishes the equivalence between the manifold (strong) variational sufficiency and the (strong) convexity of the augmented Lagrangian function, which follows directly from \cite[Theorem 1]{R22} and Proposition \ref{pro:subdiff-mani-tangent}.
\begin{proposition}\label{pro:vari-suff-al-conv}
 Let $(\bar{x},\bar{y}) \in \mathcal{M}\times \mathbb{Y}$ be a KKT pair of problem~\eqref{eq:problem-mani} satisfying \eqref{eq:man-kkt-origin}. 
 The manifold (strong) variational sufficient condition for problem \eqref{eq:problem-mani} under the retraction $R_{\bar{x}}$ holds  with respect to $(\bar{x}, \bar{y})$  at level $\rho >0$ if and only if there exists a closed convex neighborhood $\mathcal{W} \times \mathcal{Y}$ of $(0_{\bar{x}}, \bar{y})$ such that
\begin{itemize}[leftmargin=2em]
	\item[\rm (i)] $l^{\rho}_{R_{\bar{x}}}(\xi, y)$ is locally (strongly) convex at $0_{\bar{x}}$ for every $y \in \mathcal{Y}$;

    \vspace{1mm}
	\item[\rm (ii)] $l^{\rho}_{R_{\bar{x}}}(\xi, y)$ is concave in $y$ for every $\xi \in \mathcal{W} \subseteq B_{\bar{x}}(r_{R_{\bar{x}}})$.
\end{itemize}
Under these conditions, $(\bar{x}, \bar{y})$ is a saddle point of $l^{\rho}(x, y)$ with respect to minimization over $x \in R_{\bar{x}}(\mathcal{W})$ and maximization over $y \in \mathcal{Y}$. Moreover, these properties hold for any $\rho' \geq \rho$.  
\end{proposition}
\begin{proof}
	By applying \cite[Theorem 1]{R22} and Proposition \ref{pro:subdiff-mani-tangent} to (\ref{eq:problem-tangent}) at $(0_{\bar{x}},\bar{y})$, we obtain that $l_{R_{\bar{x}}}^{\rho}(\xi, y)$ is convex in $\xi \in \mathcal{W}$ for every $y \in \mathcal{Y}$, and concave in $y \in \mathcal{Y}$ for every $\xi \in \mathcal{W}\subseteq B_{\bar{x}}(r_{R_{\bar{x}}})$. 
	It remains to show that $(\bar{x}, \bar{y})$ is a saddle point of $l^{\rho}(x, y)$ in $R_{\bar{x}}(\mathcal{W})\times\mathcal{Y}$, or that $l^{\rho}(x, y)$ attains its minimum at $\bar{x}$ in $R_{\bar{x}}(\mathcal{W})$. This is shown by
	$
	l^{\rho}\left( \bar{x},\bar{y}\right) =l_{R_{\bar{x}}}^{\rho}\left( 0_{\bar{x}},\bar{y}\right) \leq  l_{R_{\bar{x}}}^{\rho}\left( R_{\bar{x}}^{-1}(x),\bar{y}\right) =l^{\rho}(x,\bar{y})$ for any $x\in R_{\bar{x}}(\mathcal{W}).
	$
	Hence we complete the proof.\qed
\end{proof}

In the Euclidean setting, a key feature of the strong variational sufficiency is its equivalence to the SSOSC. In our work, we further reveal the intrinsic relationship between the manifold strong variational sufficiency and the M-SSOSC; see Theorem~\ref{thm:strong-vari-suff-ssosc}.

 For a differentiable function $ f:\mathcal{M}\to \mathbb{R} $ with a locally Lipschitz continuous gradient and a given retraction $ R_{\bar{x}} $, the Hessian bundle of $ f_{R_{\bar{x}}} $ at $ 0_{\bar{x}} $ is defined as
$$
 \overline{\nabla}^2 f_{R_{\bar{x}}}(0_{\bar{x}}) := \left\{H \;\middle\vert\; \exists\, \xi_k \to 0_{\bar{x}} \text{ with } \nabla^2 f_{R_{\bar{x}}}(\xi_k) \to H, \; \xi_k \in T_{\bar{x}}\mathcal{M} \right\},
 $$
 where $\nabla^2 f_{R_{\bar{x}}}(\cdot)$ denotes the Euclidean Hessian of $ f_{R_{\bar{x}}}(\cdot)$.
 For the augmented Lagrangian function $ l_{R_{\bar{x}}}^{\rho}(\xi,y) $ in \eqref{def:auglag}, consider any matrix $ H \in \overline{\nabla}^2 l_{R_{\bar{x}}}^{\rho}(\xi,y) $. we  partition $ H $ into four blocks 
 $
H = 
\begin{bmatrix}
    H_{\xi\xi} & H_{\xi y}\\
    H_{y\xi}   & H_{yy}
\end{bmatrix},
$
 and define 
$ \displaystyle
\overline{\nabla}_{\xi\xi}^2 l_{R_{\bar{x}}}^{\rho}(0_{\bar{x}},y) := \left\{ H_{\xi\xi} \mid H \in \overline{\nabla}^2 l_{R_{\bar{x}}}^{\rho}(\xi,y) \right\}.
 $
Moreover,  $ \mathcal{C}_{\theta, g}(x, y) = \mathcal{C}_{\theta, g_{R_x}}(0_x, y) $ holds if $ \mathcal{C}_{\theta, g_{R_x}}(0_x, y) = \left\{d \in \mathbb{Y} \mid \theta^{\downarrow}(g_{R_x}(0_x); d) = \langle d, y \rangle\right\}$. By leveraging \cite[Theorem 3]{R22}, we can now establish the connection between the manifold strong variational sufficiency and the M-SSOSC for problem \eqref{eq:problem-mani}, as stated in the following theorem.

\begin{theorem}\label{thm:strong-vari-suff-ssosc}
     Let $(\bar{x},\bar{y}) \in \mathcal{M}\times \mathbb{Y}$ be a KKT pair of problem~\eqref{eq:problem-mani} satisfying \eqref{eq:man-kkt-origin}. 
     Then, the manifold strong variational sufficient condition for problem \eqref{eq:problem-mani} under the retraction $R_{\bar{x}}$ holds  with respect to $(\bar{x}, \bar{y})$  at level $\rho > 0$ if and only if every matrix in $\overline{\nabla}_{\xi\xi}^2 l_{R_{\bar{x}}}^{\rho}(0_{\bar{x}}, \bar{y})$ is positive definite. Moreover, any $H_{\xi\xi} \in \overline{\nabla}_{\xi\xi}^2 l_{R_{\bar{x}}}^{\rho}(0_{\bar{x}}, \bar{y})$ can be written in the form
	\begin{equation*}
		H_{\xi\xi}=\operatorname{Hess}_x L(\bar{x}, \bar{y}) + Dg(\bar{x})^* G\, \operatorname{grad} g(\bar{x}) \quad \text{for some} \quad G \in \overline{\nabla}^2 \operatorname{env}_{\rho}\theta\left(g(\bar{x}) + \rho^{-1} \bar{y}\right).
	\end{equation*}
	If $\theta$ is a polyhedral convex function, then the manifold strong variational sufficient condition is equivalent to the following  M-SSOSC at $(\bar{x}, \bar{y})$:
	\begin{equation*}
		\left\langle \xi, \operatorname{Hess}_x L(\bar{x}, \bar{y}) \xi \right\rangle > 0 \quad \forall\, Dg(\bar{x})\xi \in \operatorname{aff}\mathcal{C}_{\theta, g}(\bar{x}, \bar{y}) \setminus \{0\}.
	\end{equation*}
	Furthermore, if $\theta$ is the indicator function of a second-order cone or a positive semidefinite cone, then the manifold strong variational sufficient condition is equivalent to the M-SSOSC at $(\bar{x}, \bar{y})$ given in \eqref{eq:ssosc-mani}, where $\mathcal{K}$ denotes the respective cone.
\end{theorem}

\begin{proof}
	By \cite[Theorem 3]{R22}, every $H_{\xi\xi}\in \overline{\nabla}_{\xi\xi}^2  l_{R_{\bar{x}}}^{\rho}(0_{\bar{x}},\bar{y})$ is positive definite and  of the form
	$$
	H_{\xi\xi}=\nabla_{\xi\xi}^2  l_{R_{\bar{x}}}(0_{\bar{x}}, \bar{y})+ \nabla g_{R_{\bar{x}}}(0_{\bar{x}})^* G \nabla g_{R_{\bar{x}}}(0_{\bar{x}})\;\text {for some}\; G \in \overline{\nabla}^2 \operatorname{env}_{\rho}\theta\left(g_{R_{\bar{x}}}(0_{\bar{x}})+\rho^{-1} \bar{y}\right).
	$$
    Since $\operatorname{grad}_x L(\bar{x},\bar{y})=0$, the definition of the retraction together with \eqref{eq:hess-tang-mani} shows that this representation of $H_{\xi\xi}$ coincides with the expression stated in the theorem.
    
    If $\theta$ is polyhedral convex, \cite[Theorem 4]{R22} shows that strong variational sufficiency holds if and only if
	$$
	\left\langle\xi, \nabla_{\xi\xi}^2   l_{R_{\bar{x}}}\left(0_{\bar{x}}, \bar{y}\right) \xi\right\rangle>0 \quad \forall\,  \,g^{\prime}_{R_{\bar{x}}}(0_{\bar{x}})\xi \in \operatorname{aff}\mathcal{C}_{\theta,g_{R_{\bar{x}}}}\left(0_{\bar{x}},\bar{y}\right) \backslash\{0\}.
	$$
    This condition is equivalent to the M-SSOSC stated in the theorem, since  
    \[
        \nabla_{\xi\xi}^2 l_{R_{\bar{x}}}(0_{\bar{x}},\bar{y}) = \operatorname{Hess}_x L(\bar{x}, \bar{y}), 
        \quad g^{\prime}_{R_{\bar{x}}}(0_{\bar{x}})\xi = Dg(\bar{x})\xi, 
        \quad \mathcal{C}_{\theta,g_{R_{\bar{x}}}}(0_{\bar{x}},\bar{y})=\mathcal{C}_{\theta,g}(\bar{x},\bar{y}).
    \]
    
     Finally, if $\theta$ is the indicator function of a second-order cone or a positive semidefinite cone, 
	the equivalence follows from  \cite{WDZZ22} and the above analysis. Therefore the proof is completed.\qed
\end{proof}

\begin{remark}\label{rm:suff-independent}
	Theorem \ref{thm:strong-vari-suff-ssosc} shows that the matrices in the Hessian bundle of $l_{R_{\bar{x}}}^{\rho}(0_{\bar{x}},\bar{y})$ are independent of the particular choice of retraction $R_{\bar{x}}$. This fact can also be established directly. Given two retractions $R_1$ and $R_2$ at $\bar{x}$, it suffices to prove that the Hessian bundles $\overline{\nabla}_{\xi\xi}^2 l_{R_1}^{\rho}(0_{\bar{x}}, \bar{y})$ and $\overline{\nabla}_{\xi\xi}^2 l_{R_2}^{\rho}(0_{\bar{x}}, \bar{y})$ are identical. Specifically, for any $ H \in \overline{\nabla}_{\xi\xi}^2 l_{R_2}^{\rho}(0_{\bar{x}}, \bar{y})$, there exists a sequence $ \{\xi_k\} $ converging to $ 0_{\bar{x}}$ such that $ \nabla_{\xi\xi}^2 l_{R_2}^{\rho}(\xi_k, \bar{y})\to H$. Let $W_{\bar{x}}:=R_1^{-1}R_2$. Then we have 
	$$
 \nabla_{\xi\xi}^2 l_{R_2}^{\rho}(\xi_k, \bar{y})= W^{\prime}_{\bar{x}}(\xi_k)^* \nabla_{\xi\xi}^2 l_{R_1}^{\rho}(W_{\bar{x}}(\xi_k), \bar{y}) \nabla W_{\bar{x}}(\xi_k) + \nabla (W^{\prime}_{\bar{x}}(\xi_k)^*) \nabla_{\xi\xi}^2 l_{R_1}^{\rho}(W_{\bar{x}}(\xi_k), \bar{y}).
	$$
	It follows from the definition of $W_{\bar{x}}$ that when $\xi_k\to 0_{\bar{x}}$,  $\nabla W_{\bar{x}}(\xi_k)\to \operatorname{id}$ and $\nabla(W^{\prime}_{\bar{x}}(\xi_k)^*)\to 0_{\bar{x}}$. Therefore, if $ \nabla_{\xi\xi}^2 l_{R_2}^{\rho}(\xi_k, \bar{y})\to H$, then there exists a sequence $\{W_{\bar{x}}(\xi_k)\}$ converge to $ 0_{\bar{x}}$ such that $ \nabla_{\xi\xi}^2 l_{R_1}^{\rho}(W_{\bar{x}}(\xi_k), \bar{y}) \to H$. By symmetry, exchanging $R_1$ and $R_2$ yields $\overline{\nabla}_{\xi\xi}^2 l_{R_1}^{\rho}(0_{\bar{x}}, \bar{y})=\overline{\nabla}_{\xi\xi}^2 l_{R_2}^{\rho}(0_{\bar{x}}, \bar{y})$. Therefore, the Hessian bundle does not depend on the choice of retraction. Consequently,  the manifold strong variational sufficient condition is an intrinsic property of manifold optimization problems, inherently independent of the choice of retraction. 
\end{remark}

\begin{remark}\label{rm:convexity}
	Building on the above construction, we can introduce a pointwise retraction‑based local convexity on manifolds. A function $f:{\cal M}\to (-\infty,\infty]$ is  said to be (strongly) locally retraction‑convex at $x\in{\cal M}$ with respect to a retraction $R_x$ if $f_{R_x}$ is locally (strongly) convex on the tangent space $T_x\mathcal{M}$. Likewise, for a pair $(x, z) \in \operatorname{gph}\partial f$, we say that $f$ is variationally (strongly) convex at $(x, z)$ under the retraction $R_x$ when $f_{R_x}$ enjoys  Euclidean variational (strong) convexity at $(0_{x},z)\in \operatorname{gph}\partial f_{R_{x}}$ on $T_x\mathcal{M}\times T_x\mathcal{M}$. This concept is closely related to the retraction-convexity introduced in \cite[Definition 3.2]{HW22}, yet differs in the following respect: \cite{HW22} imposes the property on a neighborhood of $x$, whereas we investigate it at the point $x$ itself. Moreover, we  establish that if $f$ is continuously differentiable and $f_{R_x}$ is prox‑regular (\cite[13.27]{RW98}), then the local strong retraction‑convexity is invariant under the choice of retraction. Extending this invariance to weaker smoothness assumptions or to functions that are convex but not strongly convex remains an open question for future study.
\end{remark}


Another important perturbation property that plays a crucial role in the convergence analysis of the ALM is augmented tilt stability. Tilt stability characterizes the sensitivity of the optimal solution set to small linear perturbations of the objective or the constraints, and thus serves as a fundamental measure of the robustness of an optimization problem. Recently, Rockafellar extended this concept to augmented tilt stability and employed it to derive local convergence rates of ALM for non‑convex and nonsmooth problems in the Euclidean setting \cite{R22-ext}. Inspired by this, we investigate augmented tilt stability for optimization on Riemannian manifolds, and in Section~\ref{sec:ralm}, we show how it underpins the convergence analysis of manifold algorithms. We now formalize the notion of augmented tilt stability on a manifold with respect to a chosen retraction and subsequently characterize it in terms of manifold strong variational sufficiency.


\begin{definition}\label{def:augmented-tilt-stab-mani}
Within the framework of  Proposition~\ref{pro:vari-suff-al-conv}, which establishes the local convexity-concavity of $l^{\rho}(x, y)$ on the  neighborhoods $\mathcal{W} \times \mathcal{Y}$, the manifold augmented tilt stability is said to hold at $\bar{x} \in \mathcal{M}$ with respect to the retraction $R_{\bar{x}}$ if there exists a neighborhood $\mathcal{V}$ of $0_{\bar{x}}$ such that the following mapping is single-valued and Lipschitz continuous:
	\[
	(v, y) \mapsto \underset{x \in R_{\bar{x}}(\mathcal{W})}{\operatorname{argmin}}\left\{ l^{\rho}(x, y) - \left\langle v, R_{\bar{x}}^{-1}(x) \right\rangle \right\} \quad \text{for} \quad (v, y) \in \mathcal{V} \times \mathcal{Y}.
	\]
\end{definition}



\begin{proposition}\label{pro:vari-strong-suff-al-strong-conv}
	The manifold strong variational sufficient condition holds at $(\bar{x},\bar{y})$ with respect to $R_{\bar{x}}$ for local optimality of problem~\eqref{eq:problem-mani} at level ${\rho}$ if and only if the manifold augmented tilt stability holds at $\bar{x}$ with respect to $R_{\bar{x}}$.
\end{proposition}

\begin{proof}
For $(v,y)\in\mathcal{V}\times\mathcal{Y}$, denote 
	$$
    h(v,y):=\underset{x \in R_{\bar{x}}\left( \mathcal{W}\right) }{\operatorname{argmin}}\left\{l^{\rho}(x, y)-\left\langle v, R_{\bar{x}}^{-1}(x)\right\rangle \right\}\quad \mbox{and} 
    \quad 
    h_{R_{\bar{x}}}(v, y) :=\underset{\xi \in  \mathcal{W} }{\operatorname{argmin}}\left\{l_{R_{\bar{x}}}^{\rho}(\xi, y)-\left\langle v, \xi\right\rangle \right\}.
    $$ 
	If $x\in h(v,y)$, then $R_{\bar{x}}^{-1}(x)$ is a minimizer of $l_{R_{\bar{x}}}^{\rho}(\xi, y)-\left\langle v, \xi\right\rangle$, implying  $h(v,y)\subseteq R_{\bar{x}}(h_{R_{\bar{x}}}(v, y))$. 
    The converse inclusion holds as well, and thus $h(v,y)=R_{\bar{x}} (h_{R_{\bar{x}}}(v, y))$. Therefore, $h(v,y)$ is single-valued and Lipschitz continuous if and only if  $h_{R_{\bar{x}}}(v,y)$ is single-valued and Lipschitz continuous. Thus, the augmented tilt stability holds at both $\bar{x}$ and $0_{\bar{x}}$. 
	Finally,  since $l^{\rho}_{R_{\bar{x}}}(\xi,\cdot)$ is concave in $\mathcal{Y}$ for each $\xi\in \mathcal{W}$, the desired result follows directly from  \cite[Theorem 2]{R22}. \qed
\end{proof}

\section{Application I: Riemannian sequential quadratic programming }\label{sec:rsqp}

Riemannian sequential quadratic programming (RSQP) extends the classical Euclidean sequential quadratic programming (SQP) to updates performed in the tangent spaces of a manifold. Classical Euclidean SQP generally enjoys strong theoretical guarantees and numerical efficiency. The foundational papers of Han \cite{H77} and Powell \cite{P77} proved local convergence under strict complementarity and the SOSC. 
    For more details of Euclidean SQP methods, we refer the reader to \cite{BT95}. In contrast, the theoretical development of RSQP is less mature. Schiela and Ortiz \cite{SO20} analyzed RSQP for equality constrained problems on manifolds and proved local convergence. More recently,  Obara, Okuno, and Takeda \cite{OOT22} extended the method to accommodate both equality and inequality constraints, demonstrating global convergence and, under manifold linear independence constraint qualification (M-LICQ), M-SOSC, and strict complementarity, attaining local quadratic convergence.
     

     In this section, we investigate the local convergence rate of the RSQP using the perturbation results obtained in Section \ref{sec:regular}. In particular, we  show that without requiring strict complementarity, RSQP attains superlinear local convergence under the M-SRCQ and M-SOSC conditions. We begin by outlining the RSQP framework for problem \eqref{eq:problem-mani} in Algorithm~\ref{alg:rsqp}.
\begin{algorithm}[htbp]
	\caption{Riemannian sequential quadratic programming (RSQP)  for solving  \eqref{eq:problem-mani}}
	\label{alg:rsqp}
    \normalsize 
    \begin{spacing}{1.15}
    \hangindent=0.5em \hangafter=0
{\bf Input:}  
	 $(x^{0},y^0)\in \mathcal{M}\times\mathbb{Y}$, $v\in (0,1]$, a nonnegative sequence $\{\delta^k\}$  converging to zero. Set $k=0$.
	\begin{algorithmic}[1]
		\State If $(x^{k},y^k)$ satisfies a suitable termination criterion: STOP.
		\State Choose a symmetric positive semidefinite vector field $M^k$. 
       \State Find $\xi^k\in T_{x^k}\mathcal{M}$ as an approximate solution to the subproblem
        		\begin{equation}\label{eq:rsqp}
\min \  Df(x^k)\xi +\frac{1}{2}\langle\xi , M^k\xi\rangle+\theta(g(x^k)+Dg(x^k)\xi) \quad { \rm s.t. } \quad \xi\in T_{x^k}\mathcal{M} ,
		\end{equation}
		 equivalently via finding $(\xi^k,\gamma^k)$ satisfying the following system
		\begin{equation}\label{eq:newton-man-kkt}
		\tilde{\delta}^k\in \begin{bmatrix}
				\operatorname{grad}_{x} L(x^k,y^k)\\
				-g(x^k)
			\end{bmatrix}+\begin{bmatrix}
				M^k& Dg(x^k)^*\\
				-Dg(x^k) &0
			\end{bmatrix}\begin{bmatrix}
				\xi\\
				\gamma
			\end{bmatrix}+ \begin{bmatrix}
				0\\
				\partial \theta^{*}(y^{k}+\gamma^k)
			\end{bmatrix},
		\end{equation}
		where $\tilde{\delta}^k:=\min\left\{\delta^k,\left\|\begin{bmatrix}
			\operatorname{grad}_{x} L(x^k,y^k)\\
			g(x^k)-\operatorname{prox}_{\theta}(g(x^k)+y^k)
		\end{bmatrix}\right\|^{1+v}\right\}$ controls the inexactness. 
		\State 	Compute $x^{k+1}=R_{x^k}(\xi^k)$ and $y^{k+1}=y^k+\gamma^k$.
        \State Set $k=k+1$ and go to step 1.
	\end{algorithmic}
    \end{spacing}
\end{algorithm}

Since $M^k$ is  symmetric positive semidefinite, the subproblem \eqref{eq:rsqp} is convex. And $\xi$ solves \eqref{eq:rsqp} if there exists a multiplier $\eta$ such that the pair $(\xi,\eta)$ satisfies the KKT conditions
 \begin{equation}\label{eq:rsqp-kkt}
 	\operatorname{grad}  f(x^k)+M^k\xi+Dg(x^k)^*\eta=0, \quad 
 	\eta\in \partial \theta(g(x^k)+Dg(x^k)\xi).
 \end{equation}	
We now interpret the RSQP  (Algorithm~\ref{alg:rsqp})  as a Newton-type method \cite{B94} for solving the KKT system \eqref{eq:man-kkt-origin} of problem \eqref{eq:problem-mani}. 
Since $\theta$ is  closed convex, by \cite[Theorem 23.5]{R70},
the KKT system \eqref{eq:man-kkt-origin} is equivalent to 
$$
0\in \begin{bmatrix}
	\operatorname{grad}_{x} L(x,y)\\
	-g(x)
\end{bmatrix}+\begin{bmatrix}
0\\
	\partial \theta^*(y)
\end{bmatrix}.
$$
Applying the Newton-type methods in \cite{B94} to this system yields the iteration \eqref{eq:newton-man-kkt} with $\tilde{\delta}^k=0$.
By direct computation, one verifies that $(\xi^k,\gamma^k)$ solves \eqref{eq:newton-man-kkt} with $\tilde{\delta}^k=0$ if and only if $(\xi^k,y^k+\gamma^k)$ solve \eqref{eq:rsqp-kkt}, i.e., $\xi^k$ solves \eqref{eq:rsqp}. Hence, the Newton iteration and the RSQP iteration are equivalent.
Moreover, if $\theta = 0$ and $M^k$ is chosen as the Hessian of $f$, Algorithm~\ref{alg:rsqp} reduces to a Riemannian Newton method.  If $M^k$ is taken as the generalized Jacobian of $\operatorname{grad}f$ and  certain semismoothness conditions are met, Algorithm~\ref{alg:rsqp} becomes a semismooth Newton algorithm.
This close connection between  RSQP  and Newton-type methods motivates the establishment of local superlinear convergence, as demonstrated in the following proposition.

\begin{proposition}\label{pro:newton-conv}
    Let $f$ and $g$ be twice continuously differentiable. Let $(\bar{x},\bar{y}) \in \mathcal{M}\times \mathbb{Y}$ be a KKT pair of problem~\eqref{eq:problem-mani} satisfying \eqref{eq:man-kkt-origin} and suppose that the KKT solution mapping $\mathcal{S}_{\mathrm{KKT}}$ in \eqref{eq:man-kkt-perturb2} is robust isolated calm (or strongly regular) at $(\bar{x},\bar{y}) $. Let  $\{(\xi^k,\gamma^k,x^k, y^k)\}$ be the sequence generated by Algorithm~\ref{alg:rsqp}, and define 
    $$
\calJbar^k := 
	\begin{bmatrix}
		P_{\bar{x}x^k} \operatorname{Hess}_xL(\bar{x}, \bar{y}) & Dg(\bar{x})^* \\
		-P_{\bar{x}x^k}Dg(\bar{x}) & 0
	\end{bmatrix} \quad \mbox{and}\quad 
	\mathcal{J}^k := 
	\begin{bmatrix}
		M^k & Dg(x^k)^* \\
		-Dg(x^k) & 0
	\end{bmatrix}.
	$$
    Assume  $(x^k, y^k) \to (\bar{x},\bar{y})$.
	If $$
    	(\calJbar^k - \mathcal{J}^k)
    [ \xi^k ; \gamma^k]
    = o(\|[ \xi^k ; \gamma^k] \|),$$ then $ (x^k, y^k)\to (\bar{x},\bar{y})$ converges superlinearly. Moreover, if 
    $$ (\calJbar^k - \mathcal{J}^k)
    [ \xi^k ; \gamma^k] 
    = O(\|[ \xi^k ; \gamma^k] \|^2),$$ and the parameter $v$ in Algorithm~\ref{alg:rsqp} is set to $v=1$, then $ (x^k, y^k)\to (\bar{x},\bar{y})$ converges quadratically.
\end{proposition}

\begin{proof}
    Let $\bar{z} = (\bar{x},\bar{y})$, $z^k = (x^k,y^k)$, $\zeta^k = [\xi^k;\gamma^k]$, $	\mu^k=	(\calJbar^k - \mathcal{J}^k)\zeta^k$, and
	$$
	r^k=\mu^k+\tilde{\delta}^k+ \begin{bmatrix}
		P_{x^{k+1}x^k}\operatorname{grad}_{x} L(x^{k+1},y^{k+1})\\
		-g(x^{k+1})
	\end{bmatrix} - \begin{bmatrix}
		\operatorname{grad}_{x} L(x^k,y^k)\\
		-g(x^k)
	\end{bmatrix} -	\calJbar^k \zeta^k.
	$$
	Then \eqref{eq:newton-man-kkt} 
    can be equivalently written as
	\begin{equation}\label{eq:rsqp-kkt-perturb}
			r^k\in \begin{bmatrix}
			P_{x^{k+1}x^k}\operatorname{grad}_{x} L(x^{k+1},y^{k+1})\\
			-g(x^{k+1})
		\end{bmatrix} + \begin{bmatrix}
			0\\
			\partial \theta^{*}(y^{k+1})
		\end{bmatrix}.
	\end{equation}
Since $L(x, y) = f(x) + \langle g(x), y \rangle$ is twice continuously differentiable in $x$, and $\tilde{\delta}^k = o(d(z^k, \bar{z}))$ due to the local Lipschitz property of both $\operatorname{grad}_x L(x, y)$ and $\operatorname{prox}_\theta(g(x)+y)$ with respect to $(x,y)$ at $\bar{z}$, it follows that, for sufficiently large $k$,
$
r^k = \mu^k + o(\|\zeta^k\|) + o(d(z^k, \bar{z})).
$
By the condition that $\mu^k=o(\|\zeta^k\|)$, we have
	$$
	\begin{aligned}
		&d(z^{k+1},\bar{z})
		=O(r^k)=o(\|\zeta^k\|)+o(d(z^k, \bar{z}))\\
		=&o(d(z^{k+1},z^{k}))+o(d(z^k, \bar{z}))
		=o (d(z^{k+1},\bar{z})+d(z^{k},\bar{z}) )+o(d(z^k, \bar{z})).
	\end{aligned}
	$$
The first equality follows from  the isolated calmness of $\bar{z}$ to the perturbed KKT system \eqref{eq:rsqp-kkt-perturb}, and the third equality uses the fact that $R_{x^k}$ is a local homomorphism. Hence $d(z^{k+1},\bar{z})=o(d(z^{k},\bar{z}))$, which establishes superlinear convergence of $\{z^k\}$.

Now suppose $ (\calJbar^k - \mathcal{J}^k)
    [ \xi^k ; \gamma^k] 
    = O(\|[ \xi^k ; \gamma^k] \|^2)$. 
First, $z^k\to \bar{z}$ converges superlinearly implies that $d(z^{k+1},z^k) /d(z^{k},\bar{z}) \rightarrow 1$. 
By the assumption that $f$ and $g$ are twice continuously differentiable, the mapping  
$
(x,y)\;\mapsto\; 
\begin{bmatrix}
\operatorname{Hess}_xL(x,y) & Dg(x)^* \\
-Dg(x) & 0
\end{bmatrix}
$
is locally Lipschitz at $\bar{z}$, with some Lipschitz constant $c>0$.
        For $k$ sufficiently large, we have
    $$
	\begin{aligned}
	&\left\|	\begin{bmatrix}
			P_{x^{k+1}x^k}\operatorname{grad}_{x} L(z^{k+1})\\
			-g(x^{k+1})
		\end{bmatrix} - \begin{bmatrix}
			\operatorname{grad}_{x} L(z^k)\\
			-g(x^k)
		\end{bmatrix} -\calJbar^k\zeta^k \right\|\\
		\leq \ &\left\|\begin{bmatrix}
			P_{x^{k+1}x^k}\operatorname{grad}_{x} L(z^{k+1})\\
			-g(x^{k+1})
		\end{bmatrix} - \begin{bmatrix}
			\operatorname{grad}_{x} L(z^k)\\
			-g(x^k)
		\end{bmatrix} -\begin{bmatrix}
		P_{x^{k+1}x^k}\operatorname{Hess}_xL(z^{k+1})& Dg(x^{k+1})^*\\
		-Dg(x^{k+1}) &0
		\end{bmatrix}\zeta^k \right\|\\
		&+ \left\|\begin{bmatrix}
			P_{x^{k+1}x^k}\operatorname{Hess}_xL(z^{k+1})& Dg(x^{k+1})^*\\
			-Dg(x^{k+1}) &0
		\end{bmatrix}-\calJbar^k\right\|   \left\|\zeta^k \right\|\\
			\leq \ & O(\|\zeta^k\|^2)+ c\|\zeta^k\|d(z^k,\bar{z}).
	\end{aligned}
	$$
	Again, applying the isolated calmness of $\bar{z}$ to the perturbed KKT system \eqref{eq:rsqp-kkt-perturb}, and noting that $v=1$ implies that $\tilde{\delta}^k=O(d(z^k,\bar{z})^2)$, we then obtain that
		$$
	d(z^{k+1},\bar{z})=O(r^k)=O(\|\zeta^k\|^2)+ L\|\zeta^k\|d(z^k,\bar{z})+O(d(z^k,\bar{z})^2).$$
Since $d(z^{k+1},z^k) /d(z^{k},\bar{z}) \to 1$ and $R_{x^k}$ is a local homomorphism, we conclude that
	$
	d(z^{k+1},\bar{z})=	O(d(z^{k+1},z^{k})^2)=O(d(z^{k},\bar{z})^2).
	$
	The proof is then completed. \qed
\end{proof}


To satisfy the hypotheses of Proposition~\ref{pro:newton-conv}, the following notion of hemistability was introduced by Bonnans \cite[Definition 2.2]{B94} under the Euclidean setting. We generalize this concept to Riemannian manifolds and, to accommodate inexact computations, formulate a perturbed hemistability condition that covers inexact subproblem solves and data perturbations in the following definition.
\begin{definition}
A point $\bar{z}=(\bar{x},\bar{y})$ is called a  hemistable solution of the KKT system~\eqref{eq:man-kkt-origin} if the following holds:  
for any $\alpha>0$, there exist constants $\bar{\delta}>0$ and $\varepsilon>0$ such that, for every $\hat{z}=(\hat{x},\hat{y})\in \mathcal{M}\times\mathbb{Y}$ and every symmetric positive semidefinite tangent vector $M\in T_{\hat{x}}\mathcal{M}$ satisfying
\[
d(\hat{z},\bar{z})+\left\|
\begin{bmatrix}
	P_{\bar{x}\hat{x}}\operatorname{Hess}_xL(\bar{x},\bar{y}) & Dg(\bar{x})^* \\
	-P_{\bar{x}\hat{x}}Dg(\bar{x}) & 0
\end{bmatrix}
-
\begin{bmatrix}
	M & Dg(\hat{x})^* \\
	-Dg(\hat{x}) & 0
\end{bmatrix}\right\| < \varepsilon,
\]
the following system
\[
\delta \in 
\begin{bmatrix}
	\operatorname{grad}_{x} L(\hat{x},\hat{y}) \\
	-g(\hat{x})
\end{bmatrix}
+
\begin{bmatrix}
	M & Dg(\hat{x})^* \\
	-Dg(\hat{x}) & 0
\end{bmatrix}
\begin{bmatrix}
	\xi \\ \gamma
\end{bmatrix}
+
\begin{bmatrix}
	0 \\
	\partial \theta^{*}(\hat{y}+\gamma)
\end{bmatrix},
\quad \text{for any $\delta\in[0,\bar{\delta}]$,}
\]
admits a solution $(\xi,\gamma)$ such that
$
z = \bigl(R_{\hat{x}}(\xi), \, \hat{y}+\gamma \bigr)
$ satisfies $ d(z,\bar{z}) \leq \alpha.
$
\end{definition}


Now we are able to present the following local convergence results of 
Algorithm~\ref{alg:rsqp}, when $M^k$ is chosen as the Riemannian Hessian of the Lagrangian function $L(x^k,y^k)$ at each step.

\begin{theorem} [Local convergence of RSQP]\label{thm:con-rsqp}
Let $f$ and $g$ be twice continuously differentiable. Let $\bar{z}=(\bar{x},\bar{y}) \in \mathcal{M}\times \mathbb{Y}$ be a KKT pair of problem~\eqref{eq:problem-mani} satisfying \eqref{eq:man-kkt-origin}.
Assume that $\bar{z}$ is a hemistable solution of \eqref{eq:man-kkt-origin}, the M-SOSC \eqref{eq:man-sosc-perturb} holds at $\bar{x}$, and the M-SRCQ \eqref{eq:man-srcq-perturb} holds at $\bar{x}$  with respect to $\bar{y}$. Choose $M^k = \operatorname{Hess}_xL(x^k, y^k)$ in Algorithm~\ref{alg:rsqp}. Let  $\{(x^k, y^k)\}$ be the sequence generated by Algorithm~\ref{alg:rsqp}, then $\{(x^k, y^k)\}$ converges superlinearly to $(\bar{x},\bar{y})$. Moreover, if the parameter $v$ in Algorithm~\ref{alg:rsqp} is set to $v=1$, then the convergence is quadratic.

\end{theorem}

\begin{proof}
	By Proposition \ref{pro:man-licq+ssosc-perturb}, M-SOSC and  M-SRCQ hold at $\bar{z}$ implies that $\mathcal{S}_{\mathrm{KKT}}$ is robust isolated calm at $\bar{z}$. Therefore, given $\delta=(\hat{a},b)$ in a neighborhood of $(0_{\bar{x}},0)$, there exist $\varepsilon_1,\kappa>0$ such that for any $z\in \mathcal{S}_{\mathrm{KKT}}(\delta)$ and $d(z,\bar{z})\leq \varepsilon_1$, we have	$d(z,\bar{z})\leq \kappa\|\delta\|$. Now choose $\alpha\leq \min (\varepsilon_1, 1 / 3 \kappa)$. Taking $M^k=\operatorname{Hess}_xL(x^k,y^k)$, the hemistability of $\bar{z}$ yields that for the given $\alpha$, there exist $\varepsilon_2\geq 0$ and $c>0$ such that for any $\tilde{\delta}^k\leq c$ and for
	 $z^k$ satisfying 
	$$d(z^k,\bar{z})+\left\|\calJbar^k-\mathcal{J}^k\right\|<\varepsilon_2,$$ 
	the system
	$$
	\tilde{\delta}^k\in \begin{bmatrix}
		\operatorname{grad}_{x} L(x^k,y^k)\\
		-g(x^k)
	\end{bmatrix}+\mathcal{J}^k\begin{bmatrix}
		\xi\\
		\gamma
	\end{bmatrix}+ \begin{bmatrix}
		0\\
		\partial \theta^{*}(y^k+\gamma^k)
	\end{bmatrix}
	$$
	has a solution $(\xi^k,\gamma^k)$ such that $z^{k+1}=(R_{x^k}(\xi^k),y^k+\gamma^k)$ and $d(z^{k+1},\bar{z})\leq \alpha$. Now let $\zeta^k=[\xi^k;\gamma^k]$, and denote
	$$
\eta^k:=\tilde{\delta}^k + \begin{bmatrix}
		P_{x^{k+1}x^k}\operatorname{grad}_{x} L(x^{k+1},y^{k+1})\\
		-g(x^{k+1})
	\end{bmatrix} - \begin{bmatrix}
		\operatorname{grad}_{x} L(x^k,y^k)\\
		-g(x^k)
	\end{bmatrix} -\mathcal{J}^k\zeta^k,
	$$
	then we have
	$$
	\eta^k\in \begin{bmatrix}
		P_{x^{k+1}x^k}\operatorname{grad}_{x} L(x^{k+1},y^{k+1})\\
		-g(x^{k+1})
	\end{bmatrix} + \begin{bmatrix}
		0\\
		\partial \theta^{*}(y^{k+1})
	\end{bmatrix}.
	$$
	By the homomorphism of $R_{x^k}$ and its local Lipschitz continuity, there exists $\mu>0$, such that $\|\xi^k\|\leq \mu d(R_{x^k}(\xi^k),x^k)$. Since $\tilde{\delta}^k=o(d(z^k,\bar{z}))$, there exists $\varepsilon_3>0$ such that $\|\tilde{\delta}^k\|\leq \frac{1}{6\kappa}d(z^k,\bar{z})$ for $d(z^k,\bar{z})\leq \varepsilon_3$. Shrinking $\alpha$,  $\varepsilon_2$ and $\varepsilon_3$ if necessary, the twice continuous differentiability of $L(x,y)$ ensures that 
	$
	\left\|\eta^k\right\| \leq \frac{1}{3 \kappa\mu}\left\|\zeta^k\right\|+\|\tilde{\delta}^k\| .
	$
	Therefore, by the isolated calmness, we have
	$$
	d(z^{k+1},\bar{z}) \leq \kappa	\left\|\eta^k\right\| \leq \frac{1}{3\mu}\left\|\zeta^k\right\| +\kappa\|\tilde{\delta}^k\|
		\leq \frac{1}{3}d(z^{k+1},z^k)+\frac{1}{6}d(z^k,\bar{z})
		\leq \frac{1}{3}d(z^{k+1},\bar{z})+\frac{1}{2}d(z^k,\bar{z}),
	$$
which implies $d(z^{k+1},\bar{z}) \leq \frac{3}{4}d(z^{k},\bar{z})$ and thus $\{z^{k}\}\to \bar{z}$.  Now the local convergence rate can be directly obtained by using Proposition \ref{pro:newton-conv}.  \qed
\end{proof}

Theorem~\ref{thm:con-rsqp} establishes  convergence results when $M^k = \operatorname{Hess}_xL(x^k, y^k)$. In fact, the result extends to broader choices of the approximate Hessians $M^k$. For example,  $M^k$ can be generated by quasi-Newton updates, provided the sequence remains uniformly positive definite and satisfies standard secant conditions. The Euclidean analysis is presented in \cite{B94}, and the same arguments extend to the Riemannian setting by working on tangent spaces via retraction. For brevity, we omit the routine details.




In \cite{B94}, it was shown that the robust isolated calmness of the KKT solution mapping implies hemistability under the polyhedral assumption of $\theta$ with $\delta=0$. This result was later extended in \cite{MMS20} to the case where $\theta$ is parabolically regular, a property introduced in \cite[Definition 13.65]{RW98} and shown in \cite{MMS21} to include the $\mathcal{C}^2$-cone reducible case. In the Riemannian setting, we establish that under the parabolic regularity assumption of $\theta$, the robust isolated calmness of the KKT solution mapping still implies hemistability. Furthermore, our result 
strengthens \cite[Theorem~5.1]{MMS20} by allowing for inexact cases.

\begin{proposition}\label{pro:poly-delta-hemi}
Let $f$ and $g$ be twice continuously differentiable and  $\theta$ be parabolically regular. Let $(\bar{x},\bar{y}) \in \mathcal{M}\times \mathbb{Y}$ be a KKT pair of problem~\eqref{eq:problem-mani} satisfying \eqref{eq:man-kkt-origin}.
Assume that the M-SOSC \eqref{eq:man-sosc-perturb} holds at $\bar{x}$ and the M-SRCQ \eqref{eq:man-srcq-perturb} holds at $\bar{x}$  with respect to $\bar{y}$. Then  $(\bar{x}, \bar{y})$  is a hemistable solution of \eqref{eq:man-kkt-origin}.
\end{proposition}
\begin{proof}
 It can be verified that $(0_{\bar{x}},\bar{y})$ is a solution of the KKT system \eqref{eq:rsqp-kkt} associated with the RSQP subproblem \eqref{eq:rsqp} at $\bar{x}$.
 Moreover, the SRCQ condition for \eqref{eq:rsqp-kkt} at this point takes the form
$$
Dg(\bar{x})T_{\bar{x}}\mathcal{M}+\mathcal{T}_{\operatorname{dom}(\theta)}(g(\bar{x}))\cap\bar{y}^{\perp}=\mathbb{Y},
$$
which coincides with the M-SRCQ condition for problem \eqref{eq:problem-mani} at $(\bar{x},\bar{y})$.
 Hence, the SRCQ condition holds under the M-SRCQ assumption. Likewise, the SOSC condition for \eqref{eq:rsqp-kkt} is also satisfied at $(0_{\bar{x}},\bar{y})$, as it has the same structure as the M-SOSC condition for \eqref{eq:problem-mani} at $(\bar{x},\bar{y})$. Given that $\theta$ is parabolically regular, we can apply \cite[Theorem 4.2]{MMS20} and obtain that for any neighborhood $\mathcal{U}_1$ of $\bar{z}$, there exists a neighborhood $\mathcal{U}_2$ of $\bar{z}$ such that for any $\hat{z} \in \mathcal{U}_2$, the perturbed KKT system \eqref{eq:rsqp-kkt-perturb} has a solution within $\mathcal{U}_1$ when $\delta$ is sufficiently small. This confirms the hemistability of the system at $\bar{z}$.
	\qed
\end{proof}

By combining Theorem~\ref{thm:con-rsqp} and Proposition~\ref{pro:poly-delta-hemi}, we are now ready to present the main result of this section.

\begin{proposition}\label{pro:poly-delta-hemi2}
Let $f$ and $g$ be twice continuously differentiable and  $\theta$ be parabolically regular. Let $(\bar{x},\bar{y}) \in \mathcal{M}\times \mathbb{Y}$ be a KKT pair of problem~\eqref{eq:problem-mani} satisfying \eqref{eq:man-kkt-origin}.
Assume that the M-SOSC \eqref{eq:man-sosc-perturb} holds at $\bar{x}$ and the M-SRCQ \eqref{eq:man-srcq-perturb} holds at $\bar{x}$  with respect to $\bar{y}$. Choose $M^k = \operatorname{Hess}_xL(x^k, y^k)$ in Algorithm~\ref{alg:rsqp}. Let  $\{(x^k, y^k)\}$ be the sequence generated by Algorithm~\ref{alg:rsqp}, then $(x^k,y^k)$ converges superlinearly (quadratically if $v=1$) to $(\bar{x}, \bar{y})$.
\end{proposition}

\begin{remark}
This result relaxes the assumptions in \cite{OOT22}, where quadratic convergence of RSQP was established under   M-LICQ  (equivalent to manifold constraint nondegeneracy in 
NLP on Riemannian manifolds) together with M-SOSC. In contrast, our analysis requires  the M-SRCQ, which is a weaker requirement compared with M-LICQ, together with M-SOSC.  To illustrate this, we provide a toy example in which M-SRCQ holds but manifold constraint nondegeneracy (and thus M-LICQ) fails. 
    Consider the optimization on the sphere $\mathcal{S}^1$:
	\begin{equation*}\label{exam:sphere}  
		\begin{array}{ll}
			\min & \|x\|_2^2\\
			\text { s.t. } 
			& x\in \mathcal{P}=\{x_1+x_2\geq 1,\ x_1-x_2\leq 1,\ x_2\geq 0\},\\
			& x_1^2+x_2^2=1.
		\end{array} 
	\end{equation*}
 It can be verified that  $x = (1,0)^{\top}$, $y = (0,0)^{\top}$ is a KKT pair. Now let $g(x):=x$. By definitions \eqref{eq:man-srcq-perturb} and \eqref{eq:man-cn-perturb}, the M-SRCQ is said to hold at $(x,y)$ if 
$$
Dg(x)T_x\mathcal{S}^1 + \mathcal{T}_{\mathcal{P}}(g(x))\cap y^{\perp}= \mathbb{R}^2,
$$
and the manifold constraint nondegeneracy holds at $x$ if
$$
Dg(x)T_x\mathcal{S}^1 + \operatorname{lin}(\mathcal{T}_{\mathcal{P}}(g(x))) = \mathbb{R}^2.
$$
Since $\mathcal{T}_{\mathcal{P}}(g(x))=\{\xi\mid\xi_2\geq|\xi_1|\}$, we have $\operatorname{lin}(\mathcal{T}_{\mathcal{P}}(g(x)))=0$. Combining that $Dg(x)T_x\mathcal{S}^1 =T_x\mathcal{S}^1 =\{\xi\mid x^{\top}\xi=0\}=\{\xi\mid\xi_1=0,\ \xi_2\in\mathbb{R}\}$, it is easy to see that the manifold constraint nondegeneracy does not hold at $x=(1,0)^{\top}$. However, since $y = (0,0)^{\top}$ and $y^{\perp}=\mathbb{R}^2$, this implies that $$
	Dg(x)T_x\mathcal{S}^1 +\mathcal{T}_{\mathcal{P}}(g(x))\cap y^{\perp}=\{\xi\mid\xi_1=0,\ \xi_2\in\mathbb{R}\}+ \{\xi\mid\xi_2\geq|\xi_1|\}=\mathbb{R}^2,
$$
which means that the M-SRCQ holds at this point.
\end{remark}

We now discuss the RSQP algorithm and compare it with the ManPG~\cite{CMMZ20} and AManPG~\cite{HW19} methods. Consider the case where $\mathcal{M}$ is an embedded submanifold and $g$ is the identity mapping, i.e., $g(x) = x$ for all $x \in \mathcal{M}$. In this case, the RSQP subproblem~\eqref{eq:rsqp} simplifies to
    	\begin{equation*}
\min \  Df(x^k)\xi +\frac{1}{2}\langle\xi , M^k\xi\rangle+\theta(x^k+\xi) \quad { \rm s.t. } \quad \xi\in T_{x^k}\mathcal{M} .
		\end{equation*}
This formulation closely resembles that in~\cite[Algorithm 1]{HW19}. However,  the ManPG and AManPG  guarantee convergence without convergence rates. In contrast, our RSQP (Algorithm~\ref{alg:rsqp}) achieves at least local superlinear convergence rate under mild assumptions (M-SRCQ and M-SOSC), as established in Theorem~\ref{thm:con-rsqp} and Proposition~\ref{pro:poly-delta-hemi2},  without requiring acceleration techniques.
Furthermore, the RSQP  bears structural resemblance to the ManPQN algorithm~\cite{WY23}, but with a critical distinction: our update matrix $M^k$ approximates the Hessian of the Lagrangian function, rather than merely $\operatorname{Hess} f(x)$. This adjustment enables RSQP to achieve superlinear convergence rate, outperforming the linear convergence rate of ManPQN.
Additionally, we note that the recently proposed Riemannian proximal Newton method (RPN)~\cite{SAHJV24} also attains superlinear convergence. However, RPN follows a two-stage procedure: it first computes a proximal gradient direction via ManPG, followed by a Riemannian Newton update. In contrast, RSQP achieves superlinear convergence in a single stage, offering a simpler and more direct approach.

\section{Application II: Riemannian augmented Lagrangian method}\label{sec:ralm}

Building on the stability results obtained in previous sections, especially the manifold (strong) variational sufficiency developed in Section \ref{sec: variational-sufficiency}, we will analyze  the local convergence  of the Riemannian Augmented Lagrangian Method (RALM) in this section as an application. 

The classical Euclidean ALM was initially proposed by Hestenes \cite{H69} and Powell \cite{P69} for equality constrained problems and later extended to NLP by Rockafellar \cite{R73}. Over the decades, the convergence analysis of ALM in the Euclidean setting has been extensively studied.  The RALM for nonsmooth manifold optimization was recently introduced in \cite{ZBDZ21} and becomes a powerful tool for solving constrained optimization problems on manifolds (see also \cite{DP19} and \cite{LB20}). In \cite{LB20}, the  convergence of the ALM to a KKT point is established under the linear independence constraint qualification. 
Similarly, in \cite{ZBDZ21},  the convergence of the RALM to a KKT point is established under suitable conditions, together with numerical experiments. 
Recent advancements, such as those in \cite{YS22l} and \cite{ACFH24}, have introduced approximate KKT conditions and reformulated constraint qualifications to enhance convergence guarantees. Regarding the local convergence analysis of the RALM, \cite{ZBD22} established the Q-linear convergence rate of RALM under the M-SOSC and M-SRCQ conditions, which implies the uniqueness of the multiplier. 

Prior analyses often relied on the uniqueness of the  multiplier, which simplifies the analysis but limits applicability when multiple multipliers exist. Inspired by recent Euclidean advances \cite{R22-ext}, where a local dual problem under variational sufficiency and a proximal point framework yield Q-linear rates for PPA and, in turn, R-linear rates for ALM, we may remove the   assumption of unique multiplier  in the manifold setting by employing the manifold strong variational sufficiency for local optimality proposed in Section \ref{sec: variational-sufficiency}.

\subsection{Convergence analysis under manifold variational sufficiency}
In this subsection, we will discuss the convergence of the RALM for solving problem (\ref{eq:problem-mani}) when the manifold variational sufficiency holds, see Definition~\ref{def:vari-suff-mani} and Proposition~\ref{pro:vari-suff-al-conv}.
We begin by presenting the RALM framework for  problem \eqref{eq:problem-mani} in Algorithm~\ref{alg:ralm}.
\begin{algorithm}[htbp]
	\caption{Riemannian augmented Lagrangian method (RALM) for solving \eqref{eq:problem-mani}}
	\label{alg:ralm}
    \normalsize 
    \begin{spacing}{1.15}
	\hspace*{0.02in} {\bf Input:}  
	$(x^{0},y^0)\in \mathcal{M}\times\mathbb{Y}$, $\rho_0 > {\rho}$. Set $k=0$.
	\begin{algorithmic}[1]
		\State If $(x^{k},y^k)$ satisfies a suitable termination criterion: STOP.
		\State Find an approximate solution $x^{k+1}$ 
			\begin{equation}\label{eq:alm-subproblem}
						x^{k+1}\approx \bar{x}^{k+1} = {\operatorname{argmin}}\; l^{\rho_k}(x, y^k) \quad { \rm s.t. } \quad x \in R_{\bar{x}}(\mathcal{W}).
				\end{equation}		
		\State 	 Update the multiplier  
	$
				y^{k+1}=y^k+\tilde{\rho}_{k} \nabla_y   l^{\rho_k}(x^{k+1}, y^k)$, with $ \tilde{\rho}_{k}=\rho_{k}-{\rho}.
$
			\State 	Update $\rho_{k+1} \geq \rho_{k} $. 
            \State Set $k=k+1$ and go to step 1.
	\end{algorithmic}
    \end{spacing}
\end{algorithm}

 When using the tangent-space formulation, the RALM can also be written as
\begin{equation}\label{eq:alm-tangent-vs}
	\left\{\begin{array}{l}
		\xi^{k+1} \approx \bar{\xi}^{k+1}:= {\operatorname{argmin}}\; l_{R_{\bar{x}}}^{\rho_k}(\xi, y^k) \quad {\rm s.t.} \quad \xi \in \mathcal{W},\\[4pt]
		x^{k+1}=R_{\bar{x}}(\xi^{k+1}),\\[4pt]
		y^{k+1}=y^k+\tilde{\rho}_{k} \nabla_y   l_{R_{\bar{x}}}^{\rho_k}(\xi^{k+1}, y^k).
	\end{array}\right.
\end{equation}
The first  iteration of \eqref{eq:alm-tangent-vs} for finding $\xi^{k+1}$ can be considered as a traditional inexact ALM step in the tangent space, and the second  iteration pulls $\xi^{k+1}$ back to the manifold using the retraction $R_{\bar{x}}$.
Now we begin to analyze the convergence of the RALM (Algorithm~\ref{alg:ralm}). 


For subproblem \eqref{eq:alm-subproblem}, three  conditions for measuring the inexactness  of $x^{k+1}$ are given in \cite[(1.15)]{R22-ext}
\begin{equation}\label{eq:alm-app-mani}
\left(2 \tilde{\rho}_{k}
	\Big[l^{\rho_{k}} (x^{k+1}, y^k )-\inf _{x\in R_{\bar{x}}(\mathcal{W})} l^{\rho_{k}} (x, y^k ) \Big] \right)^{\frac{1}{2}} \leq \begin{cases}(\text{a}) & \varepsilon_k, \\ (\text{b}) & \varepsilon_k \min  \{1,\|\tilde{\rho}_{k} \nabla_y l^{\rho_{k}} (x^{k+1}, y^k )\| \}, \\ (\text{c}) & \varepsilon_k \min \big\{1,\|\tilde{\rho}_{k}\nabla_y l^{\rho_{k}} (x^{k+1}, y^k )\|^2\big\}.\end{cases}
\end{equation}
Here, $\varepsilon_k \in (0,1)$ for all $k \geq 0$, and $\sum_{k=0}^\infty \varepsilon_k = \sigma < \infty$. As a counterpart, for (\ref{eq:alm-tangent-vs}), three  conditions for measuring the inexactness  of $\xi^{k+1}$ are given by
\begin{equation}\label{eq:alm-app-tangent}
		 \left(2  \tilde{\rho}_{k}
	\Big[  l_{R_{\bar{x}}}^{\rho_k} (\xi^{k+1}, y^k )-\inf _{\xi\in\mathcal{W}}    l_{R_{\bar{x}}}^{\rho_k} (\xi, y^k ) \Big] \right)^{\frac{1}{2}} \leq \begin{cases}(\text{a}^{\prime})  & \varepsilon_k, \\ (\text{b}^{\prime})  & \varepsilon_k \min  \{1,\|\tilde{\rho}_{k} \nabla_y    l_{R_{\bar{x}}}^{\rho_k} (\xi^{k+1}, y^k )\| \}, \\ (\text{c}^{\prime})  & \varepsilon_k \min \big\{1,\|\tilde{\rho}_{k}\nabla_y    l_{R_{\bar{x}}}^{\rho_k} (\xi^{k+1}, y^k )\|^2\big\}.\end{cases}
\end{equation}
Algorithm~\ref{alg:ralm} under the stopping criteria \eqref{eq:alm-app-mani} is equivalent to the iteration  \eqref{eq:alm-tangent-vs} under the stopping criteria \eqref{eq:alm-app-tangent}.

We now present the convergence analysis of Algorithm~\ref{alg:ralm} under the manifold variational sufficiency. {A point $\hat{x}\in \mathcal{M}$ is called a local optimal solution of \eqref{eq:problem-mani} over $\mathcal{U}\subseteq\mathcal{M}$ if $f(\hat{x})+\theta(g(\hat{x}))\leq f(x)+\theta(g(x))$ for any $x\in \mathcal{U}$.}
Following the approach of \cite{R22-ext}, the main idea is to first establish the local augmented duality for problem~\eqref{eq:problem-mani}, and then apply the PPA to the local dual problem to derive the convergence of RALM.

	\begin{proposition}\label{pro:converge-alm}
        Let $(\bar{x},\bar{y}) \in \mathcal{M}\times \mathbb{Y}$ be a KKT pair of problem~\eqref{eq:problem-mani} satisfying \eqref{eq:man-kkt-origin}. 
        Assume that  the manifold variational sufficient condition for problem~\eqref{eq:problem-mani} under the retraction $R_{\bar{x}}$ holds with respect to $(\bar{x}, \bar{y})$ at level $\rho > 0$. Let $\mathcal{W}$ and $\mathcal{Y}$ be the closed convex neighborhoods in Proposition~\ref{pro:vari-suff-al-conv}. Suppose that the set ${\operatorname{argmin}}_{\xi \in \mathcal{W}}\; l^{{\rho}}_{R_{\bar{x}}}(\xi, y)$ is nonempty and bounded for any $y \in \operatorname{int} \mathcal{Y}$. {Denote $Z := \arg\max_{y \in \mathcal{Y}} \inf_{x\in R_{\bar{x}}(\mathcal{W})} l^{\rho}(x, y)$. Assume there exists $p>\operatorname{dist}(y^0, Z) + \sigma$ such that $	\{y \mid \|y - y^0\| < 3p \} \subset \mathcal{Y}$.}
        Let  $\{(x^k, y^k)\}$ be the sequence generated by Algorithm~\ref{alg:ralm} under  \eqref{eq:alm-app-mani}.
		Then $\{x^k\} \subseteq R_{\bar{x}}(\mathcal{W})$ is bounded, and every accumulation point of $\{x^k\}$ is a local optimal solution of problem \eqref{eq:problem-mani} over $R_{\bar{x}}(\mathcal{W})$.
\end{proposition}

\begin{proof}

In fact, the update of $(\xi^k,y^k)$ in  \eqref{eq:alm-tangent-vs} corresponds to a standard Euclidean ALM iteration. 
This structure allows us to analyze its convergence using a dual approach.
	Define the function $H(y) := \inf_{x\in R_{\bar{x}}(\mathcal{W})} l^{\rho}(x, y)$ for $y \in \mathcal{Y}$ and $H(y) := -\infty$ otherwise. Then it holds that  $H_{R_{\bar{x}}}(y):=\inf_{\xi\in\mathcal{W}} l^{\rho}_{R_{\bar{x}}}(\xi, y)$  for $y \in \mathcal{Y}$.
The local primal and dual problems associated with \eqref{eq:problem-tangent} are given by $\min_{\xi\in\mathcal{W}}\; \sup_{y \in \mathcal{Y}} l^{\rho}_{R_{\bar{x}}}(\xi, y)$ and $	\max_{y \in \mathcal{Y}}\; H_{R_{\bar{x}}}(y)$, respectively. By applying the proof of \cite[Theorem 2.1]{R22-ext} to these local problems at a saddle point $(0_{\bar{x}},\bar{y})$ of  $l_{R_{\bar{x}}}^{\rho}$ on $ \mathcal{W}\times  \mathcal{Y}$, we conclude that both optimal values are attained at $(0_{\bar{x}}, \bar{y})$, and that $0_{\bar{x}}$ is the solution to \eqref{eq:problem-tangent} over $\mathcal{W}$.

We now apply the PPA 
to the local dual problem. The solution set of the dual problem is denoted by
$Z = \arg\max_{y \in \mathcal{Y}} H_{R_{\bar{x}}}(y) = \arg\max_{y \in \mathcal{Y}} H(y)$. Then for $c_k>0$, the PPA iteration is
\begin{equation*}
	 		y^{k+1} \approx 
\bar{y}^{k+1} 
=\underset{y}{\operatorname{argmax}}\big\{ H_{R_{\bar{x}}}^k(y):= H_{R_{\bar{x}}}(y)-\frac{1}{2 {c_k}}\left\|y-y^k\right\|^2\big\}.
	 	\end{equation*}
	 	The stopping criteria for PPA follow the standard approximate conditions as specified in \cite[(1.5), (1.6)]{R21}. By applying \cite[Theorem 2.3]{R22-ext} to the update of $(\xi^k,y^k)$ in  \eqref{eq:alm-tangent-vs}, 
        we conclude that $\{y^k\}$ corresponds to an inexact PPA sequence with $c_k = \tilde{\rho}_k$, and the stopping criteria for PPA are related to 
        \eqref{eq:alm-app-tangent}. Consequently, $\{y^k\}$ converges to a local maximizer of $H_{R_{\bar{x}}}(y)$ over $y\in \mathcal{Y}$. Moreover, the sequence $\{\xi^{k+1}\}$ is bounded, and its accumulation points form solutions of the local primal problem. Finally, by the relationships that $x^{k+1} = R_{\bar{x}}(\xi^{k+1})$ and $l_{R_{\bar{x}}}^{\rho_k}(\xi^{k+1}, y^k) = l^{\rho_k}(x^{k+1}, y^k)$, the desired result follows. \qed
\end{proof}

\begin{remark}

We should emphasize that the PPA used in the proof of Proposition \ref{pro:converge-alm} differs from the one developed in \cite{FO02}. In \cite{FO02}, the authors proposed a PPA on Hadamard manifolds and analyzed its convergence under geodesic convexity assumptions. In contrast, the PPA considered here is constructed for the local dual problem around the KKT point $(0_{\bar{x}}, \bar{y})$ and is solely used to analyze the convergence behavior of RALM. It is essentially a standard Euclidean algorithm and cannot be directly applied to solve the original Riemannian problem \eqref{eq:problem-mani}. Rather, it serves as a theoretical tool in our convergence analysis.

\end{remark}

\begin{corollary}\label{cor:alm-converge-strong-convex}
Let $(\bar{x},\bar{y}) \in \mathcal{M}\times \mathbb{Y}$ be a KKT pair of problem~\eqref{eq:problem-mani} satisfying \eqref{eq:man-kkt-origin}. 
Assume that  the manifold strong variational sufficient condition for problem~\eqref{eq:problem-mani} under the retraction $R_{\bar{x}}$ holds with respect to $(\bar{x}, \bar{y})$ at level $\rho > 0$.
Then the sequence $\{x^k\}$ generated by Algorithm~\ref{alg:ralm}  converges to  $\bar{x}$.
\end{corollary}
\begin{proof}
	The strong sufficiency refers to the isolated minimizing property of $\bar{x}$. By Proposition \ref{pro:converge-alm}, $x^k$ will converge to the unique solution $\bar{x}$. \qed
\end{proof}

Now we assume that the manifold strong variational sufficiency holds at $(\bar{x},\bar{y})$ in the remaining 
part of this section. The locally strong convexity of $   l_{R_{\bar{x}}}^{\rho_k}(\cdot,y^k)$ implies that each RALM subproblem \eqref{eq:alm-subproblem}  has a unique solution. Next we give 
the local convergence rate of $\{x^k\}$ generated by Algorithm~\ref{alg:ralm}. 
	
	\begin{theorem} [Local convergence  of RALM]\label{thm:conv-rate-ralm}
Let $(\bar{x},\bar{y}) \in \mathcal{M}\times \mathbb{Y}$ be a KKT pair of problem~\eqref{eq:problem-mani} satisfying \eqref{eq:man-kkt-origin}. 
Assume that  the manifold strong variational sufficient condition for problem~\eqref{eq:problem-mani} under the retraction $R_{\bar{x}}$ holds with respect to $(\bar{x}, \bar{y})$ at level $\rho > 0$. Let $\{(x^k,y^k)\}$ be the sequence generated by Algorithm~\ref{alg:ralm} under \eqref{eq:alm-app-mani}.  Under the assumptions in Proposition~\ref{pro:converge-alm}, $\{y^k\}$ converges to $\bar{y}$, and consequently  $x^k\to \bar{x}$. Moreover,  if $y^k \to \bar{y}$ converges Q-linearly, and   $\|\operatorname{grad}_x l^{\rho_k}(x^{k+1}, y^k)\| \leq c\|y^{k+1}-y^k\|$ for some $c>0$,
	then $x^k \to \bar{x}$ also converges R-linearly at the same rate.
    
	\end{theorem}
	\begin{proof}	

	Similar to the proof of Proposition \ref{pro:converge-alm}, by applying \cite[Theorem 3.1]{R22-ext} to  the update of $(\xi^k,y^k)$ in  \eqref{eq:alm-tangent-vs} 
    and using  $x^{k+1} = R_{\bar{x}}(\xi^{k+1})$, we obtain the convergence $x^k \to \bar{x}$. According to the proof of \cite[Theorem 2]{R22}, the strong convexity of the function $l_{R_{\bar{x}}}^{\rho}(\cdot, y)$ implies that the mapping
$\lambda(v,y):=\argmin_{\xi\in\mathcal{W}}\{  l_{R_{\bar{x}}}^{\rho}(\xi,y)- \langle v,\xi \rangle \} $ is Lipschitz continuous with modulus $s^{-1}$. Since $0_{\bar{x}} = \lambda(0_{\bar{x}}, \bar{y})$ and $\xi^{k+1} = \lambda ( \nabla_{\xi} l_{R_{\bar{x}}}^{\rho_k}(\xi^{k+1}, y^k), y^k  )$, the Lipschitz continuity of $\lambda$ yields
	$$
d(x^{k+1},\bar{x})=\|\xi^{k+1}-0_{\bar{x}}\|\leq (1/s) ( \|\nabla_{\xi}   l_{R_{\bar{x}}}^{\rho_k}(\xi^{k+1},y^k)\|+\|y^k-\bar{y}\| ).
$$
Note that $x^{k+1} = R_{\bar{x}}(\xi^{k+1})$ and the sequence $\{ \xi^k \}$ remains in the closed set $\mathcal{W}$. 
There exists  $L > 0$ such that
	\begin{equation}\label{eq:alm-stopping-tangent}
	 \|\nabla_{\xi}   l_{R_{\bar{x}}}^{\rho_k}(\xi^{k+1},y^k)  \| \leq L \|\operatorname{grad}_x l^{\rho_k} (x^{k+1}, y^k ) \| \leq Lc \|y^{k+1}-y^k \|.
\end{equation}
Hence, we further obtain
	$$
\begin{aligned}
	s^2d^2(x^{k+1},\bar{x})\leq L^2c^2 \|y^{k+1}-y^k \|^2+  \|y^{k}-\bar{y} \|^2\leq L^2c^2\Big( \dfrac{ \|y^{k+1}-\bar{y} \|}{ \|y^{k}-\bar{y} \|}+1\Big)  \|y^{k}-\bar{y} \|^2+ \|y^{k}-\bar{y} \|^2.
\end{aligned}
$$
The Q-linear convergence of $\{ y^k \}$ implies $
\limsup_{k}{ \|y^{k+1}-\bar{y} \|}/{ \|y^{k}-\bar{y} \|}=r$
for some $0 < r < 1$.
%
Therefore, $ x^k $ converges to $\bar{x}$ R-linearly at the same rate.
\qed
	\end{proof}
	
In Theorem \ref{thm:conv-rate-ralm}, the local R-linear convergence of  $\{x^k\}$ is established given 
that the multiplier sequence $\{y^k\}$ converges Q-linearly. As shown in the proof of Proposition \ref{pro:converge-alm}, the sequence $\{y^k\}$ can be interpreted as being generated by PPA in the Euclidean setting. Therefore, its convergence properties can be analyzed directly using \cite[Theorem 3.2]{R22-ext}. If the condition (\ref{eq:alm-app-mani}c) is satisfied in each iteration, then $\{y^k\}$ converges to $\bar{y}$  Q-linearly.

	\begin{remark}
    	As shown in \cite{ZBD22},  the primal dual convergence rate of RALM can be established under perturbation properties of the KKT solution mappings, which require unique multipliers. In contrast, our analysis under manifold variational sufficiency yields only local primal R-linear convergence rate but removes the uniqueness requirement, extending applicability to broader settings.

	\end{remark}
	
	\begin{remark}
	In Section~\ref{sec:rsqp}, we establish that the RSQP method achieves local linear convergence under robust isolated calmness and hemistability. A similar convergence behavior can also be shown for  RALM  through the perturbation properties of the KKT solution mapping, using the error bound condition in Proposition~\ref{pro:errorbound}. 
    As shown in \cite{ZBD22}, the local $Q$-linear convergence of RALM can be established under the robust isolated calmness of the KKT solution mapping. 
    This  calmness can be verified in practice via  the M-SOSC and M-SRCQ conditions, by Proposition~\ref{pro:man-licq+ssosc-perturb}.
    Again by Proposition~\ref{pro:man-licq+ssosc-perturb},  M-SSOSC combined with constraint nondegeneracy is equivalent to the strong regularity of the KKT solution mapping when $\operatorname{epi}(\theta)$ is a polyhedral convex set, a second-order cone, or a semidefinite cone, which also guarantees linear convergence of RALM.
    Actually, the local linear convergence of ALM under the SSOSC and constraint nondegeneracy is classical in the Euclidean setting, as shown in \cite{SSZ08} for NLSDP and in \cite{LZ08} for NLSOC. 
	\end{remark}

	Now we  discuss the convergence rate of  the KKT residual. It follows from \cite[Theorem 3.1]{R22-ext} that the conditions \eqref{eq:alm-app-tangent} can be achieved by
        \begin{equation}\label{eq:alm-app-tangent-grad}
            \sqrt{{\tilde{\rho}_k}} \|\nabla_{\xi}  l_{R_{\bar{x}}}^{\rho_k}  (\xi^{k+1}, y^k ) \| \leq \begin{cases}\text { (a)}  & \varepsilon^{\prime}_k, \\ \text { (b})  & \varepsilon^{\prime}_k \min  \{1, \|\tilde{\rho}_{k} \nabla_y    l_{R_{\bar{x}}}^{\rho_k} (\xi^{k+1}, y^k ) \| \}, \\ \text { (c})  & \varepsilon^{\prime}_k \min \big\{1, \|\tilde{\rho}_{k}\nabla_y    l_{R_{\bar{x}}}^{\rho_k} (\xi^{k+1}, y^k ) \|^2\big\},\end{cases}
        \end{equation}	
	where $\varepsilon^{\prime}_k=\varepsilon_k\sqrt{s}$ as $s$ is the modulus of the strong convexity of $ l_{R_{\bar{x}}}^{\rho_k} (\cdot, y^k )$.  
    Moreover, {combining \eqref{eq:alm-app-tangent-grad} with (\ref{eq:alm-stopping-tangent}), the conditions \eqref{eq:alm-app-mani} can be achieved by}
		\begin{equation}\label{eq:alm-app-mani-grad}
			\sqrt{{\tilde{\rho}_k}} \|\operatorname{grad}_x l^{\rho_k} (x^{k+1}, y^k )  \|\leq \begin{cases}\text { (a) } & \varepsilon^{\prime\prime}_k, \\ \text { (b) } & \varepsilon^{\prime\prime}_k \min  \{1, \|\tilde{\rho}_{k} \nabla_y l^{\rho_{k}} (x^{k+1}, y^k ) \| \}, \\ \text { (c) } & \varepsilon^{\prime\prime}_k \min \big\{1, \|\tilde{\rho}_{k}\nabla_y l^{\rho_{k}} (x^{k+1}, y^k ) \|^2\big\}.\end{cases}
		\end{equation}
	 where $\varepsilon^{\prime\prime}_k=\varepsilon_k\sqrt{s}/L$ as $L$ is the  modulus given in  (\ref{eq:alm-stopping-tangent}). 
    The following proposition establishes the R-linear convergence of the KKT residual $\{E (x^{k+1}, y^{k+1} ) \}$ defined in \eqref{kktresidual}. Its proof follows the argument of \cite[Proposition 4.7]{WDZZ22}, with the projection step replaced by a proximal mapping in our setting. For brevity, we omit the details.

	\begin{proposition}\label{pro:kkt-r-linear}
Assume that  the manifold strong variational sufficient condition for problem~\eqref{eq:problem-mani} under the retraction $R_{\bar{x}}$ holds with respect to $(\bar{x}, \bar{y})$ at level $\rho > 0$. Let $\{(x^k,y^k)\}$ be the sequence generated by Algorithm~\ref{alg:ralm} under (\ref{eq:alm-app-mani-grad}b).
 If $\sqrt{{\tilde{\rho}_k}} \varepsilon_k<1$  for sufficiently large $k$, then there exists $L_g > 0$ such that
		$$
	E (x^{k+1}, y^{k+1} ) \leq w^k \operatorname{dist} (Y^k, Z ),
		\mbox{ where } w^k= (\varepsilon_k^{\prime\prime} \sqrt{\tilde{\rho}_k}+(1+{\rho}+L_g {\rho}) (\tilde{\rho}_k )^{-1} ) (1-\sqrt{\tilde{\rho}_k} \varepsilon_k )^{-1}.
        $$
	\end{proposition}

\subsection{Semismooth Newton method for RALM subproblem}\label{sec:semismooth}
For solving the RALM subproblem \eqref{eq:alm-subproblem} efficiently, we adopt the globalized semismooth Newton method on Riemannian manifolds \cite{ZBDZ21}. We begin by recalling the definition of the generalized covariant derivative for vector fields on manifolds from \cite{OF20}.

\begin{definition}
	 Let $X$ be a locally Lipschitz vector field on $\mathcal{M}$. The B-derivative of $X$  is a set-valued map $\partial_B X: \mathcal{M} \rightrightarrows \mathcal{L}(T \mathcal{M})$:
     	$$
	\partial_B X(x)=\{H \in \mathcal{L} (T_x \mathcal{M} )\mid\exists \{x^k \} \subseteq \mathcal{D}_X, \   x^k \to x,  \   \nabla X (x^k ) \to  H \}, \ x\in \mathcal{M},
	$$
	where the convergence $\nabla X (x^k ) \to  H$ means that $ \|\nabla X (x^k ) [P_{x x^k} v ]-P_{x x^k} H v \| \rightarrow 0$ for all $v \in T_x \mathcal{M}$. The Clarke generalized covariant derivative of $X$ is a set-valued map $\partial X: \mathcal{M} \rightrightarrows \mathcal{L}(T \mathcal{M})$ such that $\partial X(x)$ is the convex hull of $\partial_B X(x)$.
\end{definition}

If a retraction $R_x$ additionally satisfies $ \frac{\mathrm{D}^2}{\mathrm{~d} t^2} R(t \xi)|_{t=0}=0$ for any $\xi \in T_x \mathcal{M}$, where $\frac{\mathrm{D}^2}{\mathrm{~d} t^2} \gamma$ denotes the acceleration of a curve $\gamma$, then $R_x$ is called a second-order retraction. Such retractions can be interpreted as  approximations of exponential mappings and include, for example,  the polar retraction on the Stiefel manifold and the projective retraction on fixed rank manifolds; see \cite[Example 4.12]{AM12} for more details. In this section, we assume the retraction is second-order.

The globalized semismooth Newton method for solving the RALM subproblem (\ref{eq:alm-subproblem}),  given in \cite[Algorithm 4.1]{ZBDZ21},  aims to solve the equation  $\operatorname{grad}_x l^{\rho_k}(x,y^k)=0$ for $x$  in the $k$-th iteration. Let $R_{\bar{x}}$ be a second-order retraction.  Suppose that $\theta$ is either a polyhedral convex function or is the indicator function of a second-order cone or a positive semidefinite cone. Then, by Theorem \ref{thm:strong-vari-suff-ssosc}, the M-SSOSC is equivalent to the positive definiteness of the elements in the Hessian bundle of $l^{\rho}_{R_{\bar{x}}}(0_{\bar{x}},\bar{y})$ for sufficiently large $\rho$. Moreover, we shall show that this condition also suffices to guarantee the superlinear convergence of the semismooth Newton method.

\begin{lemma}\label{lem:equi-general-jacobian}
Let $(\bar{x},\bar{y}) \in \mathcal{M}\times \mathbb{Y}$ be a KKT pair of problem~\eqref{eq:problem-mani} satisfying \eqref{eq:man-kkt-origin}. 
Assume that the  manifold $\mathcal{M}$ is an embedded submanifold, and  the retraction $R_{\bar{x}}$ is second-order. For any $\rho>0$, we have $\partial \nabla_{\xi}l_{R_{\bar{x}}}^{\rho}(0_{\bar{x}},\bar{y})=\partial \operatorname{grad}_xl^{\rho}\left(\bar{x},\bar{y}\right)$.
\end{lemma}
\begin{proof}
   For any $G\in \partial \operatorname{grad}_xl^{\rho}\left(\bar{x},\bar{y}\right)$, we claim that $G\in \partial \nabla_{\xi}l_{R_{\bar{x}}}^{\rho}(0_{\bar{x}},\bar{y})$, or equivalently, 
   $$
   \lim_{\xi\to 0_{\bar{x}}}  \|\nabla_{\xi\xi}^2l_{R_{\bar{x}}}^{\rho} (\xi,\bar{y} ) v -Gv \|= 0,\quad \forall\ v\in T_{\bar{x}}\mathcal{M}. 
   $$ 
   Denote $l^{\rho}(\cdot,\bar{y})$ by $h(\cdot)$.
   Since $h_{R_{\bar{x}}}(\xi)=h\circ R_{\bar{x}}(\xi)$, for any $v\in T_{\bar{x}}\mathcal{M}$ we have
   \begin{equation*}
 	\begin{aligned}
   		&\langle \nabla^2h_{R_{\bar{x}}}({\xi})v,v \rangle =\dfrac{d^2}{dt^2}(h\circ R_{\bar{x}})(\xi+tv) \Big|_{t=0}\\
   		=& \dfrac{d}{dt}\Big[\dfrac{d}{dt}h(R_{\bar{x}}(\xi+tv))\Big] \Big|_{t=0}
   		=\dfrac{d}{dt}Dh(R_{\bar{x}}(\xi+tv))\Big[\dfrac{d}{dt}R_{\bar{x}}(\xi+tv)\Big] \Big|_{t=0}\\
   		=&\langle\dfrac{D}{dt}\operatorname{grad}h(R_{\bar{x}}(\xi+tv)),DR_{\bar{x}}(\xi)v\rangle+\langle\operatorname{grad}h(R_{\bar{x}}(\xi+tv)),\dfrac{D^2}{dt^2}R_{\bar{x}}(\xi+tv)\rangle \Big|_{t=0}\\
   		=& \underset{(*)}{\underbrace{\langle\nabla\operatorname{grad}h(R_{\bar{x}}(\xi))[DR_{\bar{x}}(\xi)v],DR_{\bar{x}}(\xi)v\rangle}}
        +\underset{(**)}{\underbrace{\langle\operatorname{grad}h(R_{\bar{x}}(\xi+tv)),\dfrac{D^2}{dt^2}R_{\bar{x}}(\xi+tv)\rangle \Big|_{t=0}}}.
   	\end{aligned}
   \end{equation*}
   By the definition of retraction and $\mathcal{M}$ is an embedded submanifold, for any $v\in T_{\bar{x}}\mathcal{M}$, we have
   $DR_{\bar{x}}(\xi)v=D\operatorname{exp}_{\bar{x}}(\xi)v+O(\|\xi\|)$ when $\xi\to 0_{\bar{x}}$.
   Therefore, the first term of the last equation is 
   \begin{align*}
     (*)   = & \langle\nabla\operatorname{grad}
     h(R_{\bar{x}}(\xi))[D\operatorname{exp}_{\bar{x}}(\xi)v],D\operatorname{exp}_{\bar{x}}(\xi)v\rangle+O(\|\xi\|)\\
        = & \langle\nabla\operatorname{grad} h(R_{\bar{x}}(\xi))[P_{\bar{x}R_{\bar{x}}(\xi)}v],P_{\bar{x}R_{\bar{x}}(\xi)}v\rangle+O(\|\xi\|).
   \end{align*}
   And the second term is zero $(**)=0$
   if {$\xi\to 0_{\bar{x}}$} and $t\to 0$ as $R_{\bar{x}}$ is  a second-order retraction. Thus, note that the parallel transport is isometry, for any $v\in T_{\bar{x}}\mathcal{M}$,
   $$
   \begin{aligned}
   	\big\langle\nabla^2h_{R_{\bar{x}}}\big(\xi\big)v-Gv,v\big\rangle
   	= &\;\big\langle P_{\bar{x}R_{\bar{x}}(\xi)}\nabla^2h_{R_{\bar{x}}}\big(\xi\big)v-\operatorname{Hess}h\big(R_{\bar{x}}(\xi)\big)\big[P_{\bar{x}R_{\bar{x}}(\xi)} v\big],P_{\bar{x}R_{\bar{x}}(\xi)}v\big\rangle\\
   	&+\big\langle\operatorname{Hess}h\big(R_{\bar{x}}(\xi)\big)\big[ P_{\bar{x}R_{\bar{x}}(\xi)}v\big]-P_{\bar{x}R_{\bar{x}}(\xi)}Gv,P_{\bar{x}R_{\bar{x}}(\xi)}v\big\rangle,\\
   \end{aligned}
   $$
   which converges to $0$ when $\xi\to 0_{\bar{x}}$. 
   {The arbitrary taken $v$ implies that}
   $
   \lim_{\xi\to 0_{\bar{x}}} \big\|\nabla^2h_{R_{\bar{x}}}\big(\xi\big)v-Gv\big\|= 0.
   $
   Similarly, $\partial \nabla_{\xi}l_{R_{\bar{x}}}^{\rho}(0_{\bar{x}},\bar{y})\subseteq \partial \operatorname{grad}_xl^{\rho} (\bar{x},\bar{y} )$. The equality is proved and the proof is completed.
   \qed
\end{proof}


Now, combining Lemma~\ref{lem:equi-general-jacobian} with \cite[Theorem 4.3]{ZBDZ21}, we obtain the following local convergence result for the globalized semismooth Newton method.

\begin{theorem}\label{thm:ssn-convergence}
Assume that  the manifold strong variational sufficient condition for problem~\eqref{eq:problem-mani} under the retraction $R_{\bar{x}}$ holds with respect to $(\bar{x}, \bar{y})$ at level $\rho > 0$.
Let  $R_{\bar{x}}$ be a second-order retraction and {$\{x^j\}$} be the sequence generated by the semismooth Newton method  \cite[Algorithm 4.1]{ZBDZ21}.
 Suppose there exists $\delta > 0$ such that the level set $\Omega := \{x \in \mathcal{M} \mid l^{\rho_k}(x, y^k) \leq l^{\rho_k}(x^0, y^k) + \delta\}$ is compact. {Denote $\hat{x}$ to be any accumulation point of $\{x^j\}$.} If $ \operatorname{grad}_xl^{\rho_{k}} (\cdot,y^k )$ is $\nu$-th order semismooth at {$\hat{x}$} 
 with respect to $ \partial \operatorname{grad}_xl^{\rho_{k}}\left(\cdot,y^k\right)$, then we have the convergence $x^k \rightarrow \hat{x}$, 
 and $\hat{x}$ is optimal to subproblem (\ref{eq:alm-subproblem}). Moreover, for sufficiently large $j$, it holds
$$
d (x^{j+1}, \hat{x} ) \leq O\big(d (x^j, \hat{x} )^{1+\min \{\nu, \bar{\nu}\}}\big),
$$
where $\bar{\nu} \in(0,1]$ is the parameter defined in {\cite[Algorithm 4.1]{ZBDZ21}.}
\end{theorem}
\begin{proof}
	By Theorem~\ref{thm:strong-vari-suff-ssosc}, the manifold strong sufficiency ensures the positive definiteness of all elements in $\partial \nabla_{\xi} l_{R_{\bar{x}}}^{\rho_j}(0_{\bar{x}}, \bar{y})$.  Then by Lemma~\ref{lem:equi-general-jacobian}, all elements in $\partial \operatorname{grad}_x l^{\rho_j}(\bar{x}, \bar{y})$ are also positive definite. Thus the conditions in \cite[Theorem 4.3]{ZBDZ21} are satisfied,  implying the local superlinear convergence of the globalized semismooth Newton method. The proof is  completed. \qed
\end{proof}

\begin{remark}\label{rm:ssosc-ssn-positive}
Based on Theorem \ref{thm:strong-vari-suff-ssosc} and Lemma \ref{lem:equi-general-jacobian},  the positive definiteness of the generalized Hessian of the augmented Lagrangian function is equivalent to the manifold strong variational sufficient condition of problem (\ref{eq:problem-mani}) at the KKT 
point. This equivalence underscores the pivotal role of manifold strong variational sufficiency in ensuring the efficiency of the semismooth Newton method for solving the RALM subproblem.
\end{remark}

\section{Numerical experiments}\label{sec:numerical}
In this section, we present numerical experiments on the fixed-rank manifold under the M-SOSC and M-SRCQ conditions, and on the Stiefel manifold under the M-SSOSC condition. These experiments illustrate the local convergence rates established in Theorems~\ref{thm:con-rsqp} and \ref{thm:conv-rate-ralm}.
The codes were implemented in Matlab (R2021b), and all experiments were performed on a 64-bit MacOS system with an Intel Cores i5 2.4GHz CPU and 16GB of memory.

\subsection{Solving robust matrix completion by  RSQP}
We apply RSQP (Algorithm~\ref{alg:rsqp}) to the robust matrix completion (RMC) problem proposed in \cite{CA16}.  Given $A\in \mathbb{R}^{m\times n}$, let $g(X)=P_{\Omega}(X-A)$ and $\theta(\cdot)=\mu\|\cdot\|_{1}$, where $\Omega$ is an index set and $P_{\Omega}$ is the projection operator defined by $\left(P_{\Omega}(X)\right)_{i j}=X_{i j}$ if $(i, j) \in \Omega$ and $0$ otherwise. We let $\mathcal{M}=Fr(m,n,r):=\{X\in \mathbb{R}^{m\times n}\mid\operatorname{rank}(X)=r\}$, and the RMC problem can be formulated as
\begin{equation}\label{eq:rmc-full}
\min \  \left\|P_{\Omega}(X-A)\right\|_{1} \quad
\text { s.t. } \quad  X \in Fr(m,n,r).
\end{equation}
As shown in \cite{V13}, the tangent and normal spaces of $\mathcal{M}$ at a point $X$ with the singular value decomposition $X=USV^{\top}$ are, respectively, given by
\begin{align*}
T_{X} \mathcal{M}=
\left\{\left[\begin{array}{ll}
U & U_{\perp}
\end{array}\right]\left[\begin{array}{cc}
\mathbb{R}^{r \times r} & \mathbb{R}^{r \times(n-r)} \\
\mathbb{R}^{(m-r) \times r} & 0^{(m-r) \times(n-r)}
\end{array}\right]\left[\begin{array}{ll}
V & V_{\perp}
\end{array}\right]^{\top}\right\},  \\
N_{X} \mathcal{M}=
\left\{\left[\begin{array}{ll}
U & U_{\perp}
\end{array}\right]\left[\begin{array}{cc}
0^{r \times r} & 0^{r \times(n-r)} \\
0^{(m-r) \times r} & \mathbb{R}^{(m-r) \times(n-r)}
\end{array}\right]\left[\begin{array}{ll}
V & V_{\perp}
\end{array}\right]^{\top}\right\},
\end{align*}
and the projection of a matrix $Y\in\mathbb{R}^{m\times n}$ onto the tangent space $T_{X} \mathcal{M}$ is given by $
\Pi_{X}(Y)=P_{U} Y P_{V}+P_{U}^{\perp} Y P_{V}+P_{U} Y P_{V}^{\perp}$. The Lagrangian of (\ref{eq:rmc-full}) is given by $L(X,y) = \left\langle P_{\Omega}(X-A), y\right\rangle$. And its KKT conditions are
\begin{equation}\label{eq:rmc-kkt}
		\Pi_{X} P_{\Omega} (y) =0, \quad
		y \in \partial\|	P_{\Omega}(X-A)\|_{1}.
\end{equation}
By \cite[Section 3]{HLWY20}, the Hessian of a function $f$ on $\mathcal{M}$ at $X=U\Sigma V^{\top}$ is given by
$$
\operatorname{Hess}f(X)\xi=U \hat{M} V^{\top}+\hat{U}_p V^{\top}+U \hat{V}_p^{\top},\quad   \xi \in T_X \mathcal{M},
$$
where 
$$
	\begin{aligned}
		\hat{M}&=M\left(\nabla^2 f(X)[\xi]; X\right),  M(Z ; X)=U^{\top} Z X,\\
		\hat{U}_p&=U_p\left(\nabla^2 f(X)[\xi] ; X\right)+P_U^{\perp} \nabla f(X) V_p(\xi ; X)/\Sigma, & U_p(Z ; X)=P_U^{\perp} Z V,\\
		\hat{V}_p&=V_p\left(\nabla^2 f(X)[\xi] ; X\right)+P_V^{\perp} \nabla f(X) U_p(\xi; X) / \Sigma, & V_p(Z ; X)=P_V^{\perp} Z^{\top} U.
	\end{aligned}
$$
Let $(\betterbar{X},\bar{y})$ be a KKT pair satisfying \eqref{eq:rmc-kkt}. Since $\nabla_X L(\betterbar{X},\bar{y})=P_{\Omega}(\bar{y})$ and $\nabla_{XX}^2 L(\betterbar{X},\bar{y})=0$, it holds that
$$
\begin{aligned}
	\langle \xi, \operatorname{Hess}_{X} L(\betterbar{X},\bar{y})\xi \rangle &= \langle \xi,P_U^{\perp}P_{\Omega}(\bar{y})P_V^{\perp}\xi^{\top}U\Sigma^{-1}V^{\top}+U\Sigma^{-1}V^{\top}\xi^{\top}P_U^{\perp}P_{\Omega}(\bar{y})P_V^{\perp} \rangle\\
	&=2\operatorname{tr} ( \xi^{\top}P_U^{\perp}P_{\Omega}(\bar{y})P_V^{\perp}V^{\perp}\xi^{\top}U\Sigma^{-1}V^{\top} )=2\operatorname{tr} ( \xi^{\top}P_{\Omega}(\bar{y})\xi^{\top}U\Sigma^{-1}V^{\top} ),
\end{aligned}
$$
where the last equality is obtained by the KKT condition (\ref{eq:rmc-kkt}) that $\Pi_{\betterbar{X}}P_{\Omega}(\bar{y})=0$. Moreover, by the definition of $G$, 
$Dg(X)\xi=g^{\prime}(X)\xi=P_{\Omega}(\xi)$, for any $\xi\in T_{X}\mathcal{M}$. We can further obtain that $\mathcal{C}_{\theta, g}(X, y)=\left\{ d\in\mathbb{R}^{m\times n}\mid \theta^{\downarrow}(P_{\Omega}(X-A) ; d) = \left\langle d,y \right\rangle \right\} $, in which
$$
{\theta^{\downarrow}(P_{\Omega}(X-A) ; d)}
= \sum_{P_{\Omega}(X-A)_{ij}=0}|d_{ij}|+\sum_{P_{\Omega}(X-A)_{ij}>0}d_{ij}-\sum_{P_{\Omega}(X-A)_{ij}<0}d_{ij}.
$$
Therefore, {the M-SOSC} for  problem (\ref{eq:rmc-full}) holds at the KKT pair $(\betterbar{X}, \bar{y})$  if 
$$ 
\operatorname{tr}\left( \xi^{\top}P_{\Omega}(\bar{y})\xi^{\top}U\Sigma^{-1}V^{\top}\right)>0, \mbox{ for any } \xi\in T_{\betterbar{X}}\mathcal{M} \mbox{ satisfying } P_{\Omega}(\xi)\in \mathcal{C}_{\theta, g}(\betterbar{X}, \bar{y})\backslash\{0\}.
$$


Next we consider a toy example such that  the M-SOSC and M-SRCQ conditions can be verified. We take $m=n=5$, $r=3$, and  $\Omega$ as the full index set in problem~\eqref{eq:rmc-full}. Let
\[
U= \begin{bmatrix}
	1& 0 &0 &0&0\\
	0 &-{\sqrt{2}}/{2}&{\sqrt{2}}/{2}&0&0 \\
	0  &{\sqrt{2}}/{2}&{\sqrt{2}}/{2}&0&0
\end{bmatrix}^{\top}, \quad 
V= \begin{bmatrix}
	1& 0 &0 &0&0\\
	0 &0.6&-0.8&0&0 \\
	0  &0.8&0.6&0&0
\end{bmatrix}^{\top},  \quad \mbox{and} \quad
S= \begin{bmatrix}
	1&&\\&2&\\&&3
\end{bmatrix}.
\]
Set $A_{\text{ex}} = USV^{\top}$. We then add noise by defining  $E_{\text{out}}$ as a matrix with Gaussian entries at positions $(i,j)\in\{4,5\} \times \{4,5\}$ and $0$  elsewhere, and set $A = A_{\text{ex}} + E_{\text{out}}$. 

We take $\betterbar{X} = A_{\text{ex}}$ and $\bar{y}_{ij} = -\operatorname{sgn}((E_{\text{out}})_{ij})$. It can be  verified that $(\betterbar{X}, \bar{y})$ satisfies the KKT conditions~\eqref{eq:rmc-kkt}. Next we show that  the M-SOSC and  M-SRCQ conditions hold at $(\betterbar{X}, \bar{y})$.
First, we can compute the critical cone 
$
\mathcal{C}_{\theta, g}(\betterbar{X}, \bar{y}) = \{ d\in \mathbb{R}^{m\times n} \mid d_{ij}\in \mathbb{R} \;\text{if}\; (E_{\text{out}})_{ij}\neq 0, \; d_{ij}=0\;\text{if}\; (E_{\text{out}})_{ij}=0\}.
$
The sparsity pattern of $E_{\text{out}}$ implies that only the zero element $\xi=0_{\betterbar{X}}$ satisfies $P_{\Omega}(\xi) \in \mathcal{C}_{\theta, g}(\betterbar{X}, y)$. Therefore, the M-SOSC holds at $(\betterbar{X}, \bar{y})$.
Moreover, observe that $Dg(\betterbar{X})T_{\betterbar{X}}\mathcal{M} = T_{\betterbar{X}}\mathcal{M}$. Thus, the M-SRCQ condition holds if there exists a multiplier $\bar{y}$ 
such that
$
N_{\betterbar{X}}\mathcal{M} \subseteq \mathcal{C}_{\theta, g}(\betterbar{X}, \bar{y}).
$
For any multiplier $y$ satisfying $Dg(\betterbar{X})^*y=0$, we have $y \in N_{\betterbar{X}}\mathcal{M}$ and thus ${y} = P_U^{\perp} Y P_V^{\perp}$ for some $Y\in \mathbb{R}^{m\times n}$. For any $d = P_U^{\perp} \widetilde{Y} P_V^{\perp} \in N_{\betterbar{X}}\mathcal{M}$, it is easy to see that  there exists a  multiplier ${y} = P_U^{\perp} Y P_V^{\perp}$ satisfying
\begin{align*}
	\sum_{i,j\in\{4,5\}}-\operatorname{sgn}((E_{\text{out}})_{ij})\widetilde{Y}_{ij} &= \sum_{i,j\in\{4,5\}}-\operatorname{sgn}((E_{\text{out}})_{ij})d_{ij} 
	= \theta^{\downarrow}(g(\betterbar{X}); d) \\
	&=  \langle d, y \rangle 
	=  \langle P_U^{\perp} \widetilde{Y} P_V^{\perp}, P_U^{\perp} Y P_V^{\perp}  \rangle 
	= \sum_{i,j\in\{4,5\}} \widetilde{Y}_{ij} Y_{ij}.
\end{align*}
Therefore, $d \in \mathcal{C}_{\theta, g}(\betterbar{X}, \bar{y})$, and the M-SRCQ condition holds at $\betterbar{X} = A_{\text{ex}}$ with respect to $y$.

\begin{figure}[htbp]
	\centering
    \subfloat[$m=n=5$, $r=3$]{\includegraphics[width=0.24\linewidth, trim=10 0 15 20, clip]{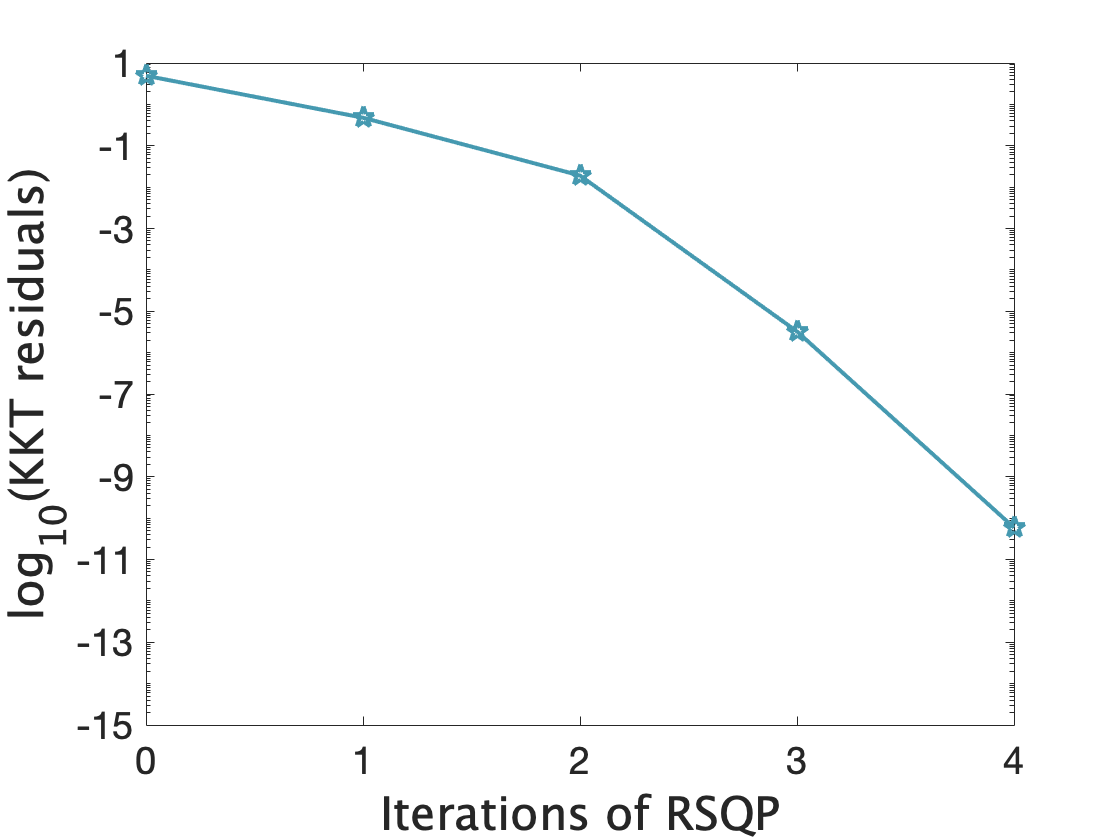}}
	\subfloat[$m=n=50$, $r=6$]{\includegraphics[width=0.24\linewidth, trim=10 0 15 20, clip]{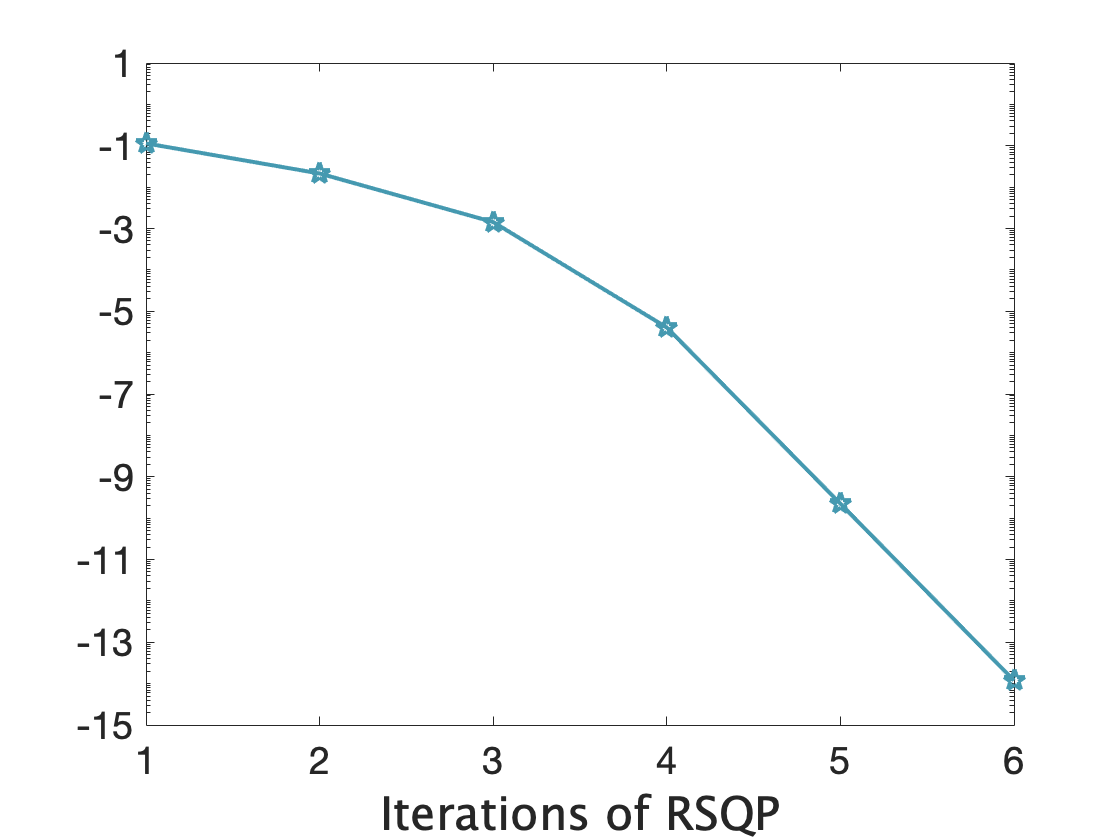}} 
	\subfloat[$m=n=100$, $r=8$]{\includegraphics[width=0.24\linewidth, trim=10 0 15 20, clip]{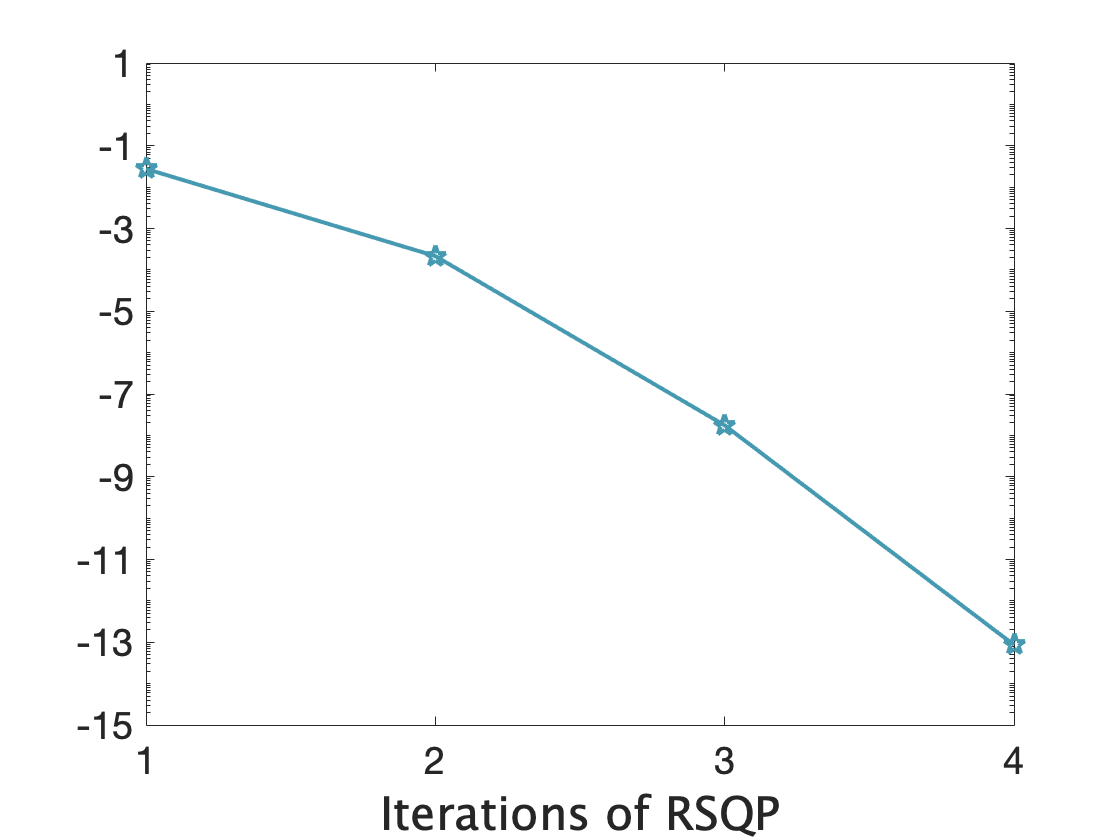}}
	\subfloat[$m=n=200$, $r=10$]{\includegraphics[width=0.24\linewidth, trim=10 0 15 20, clip]{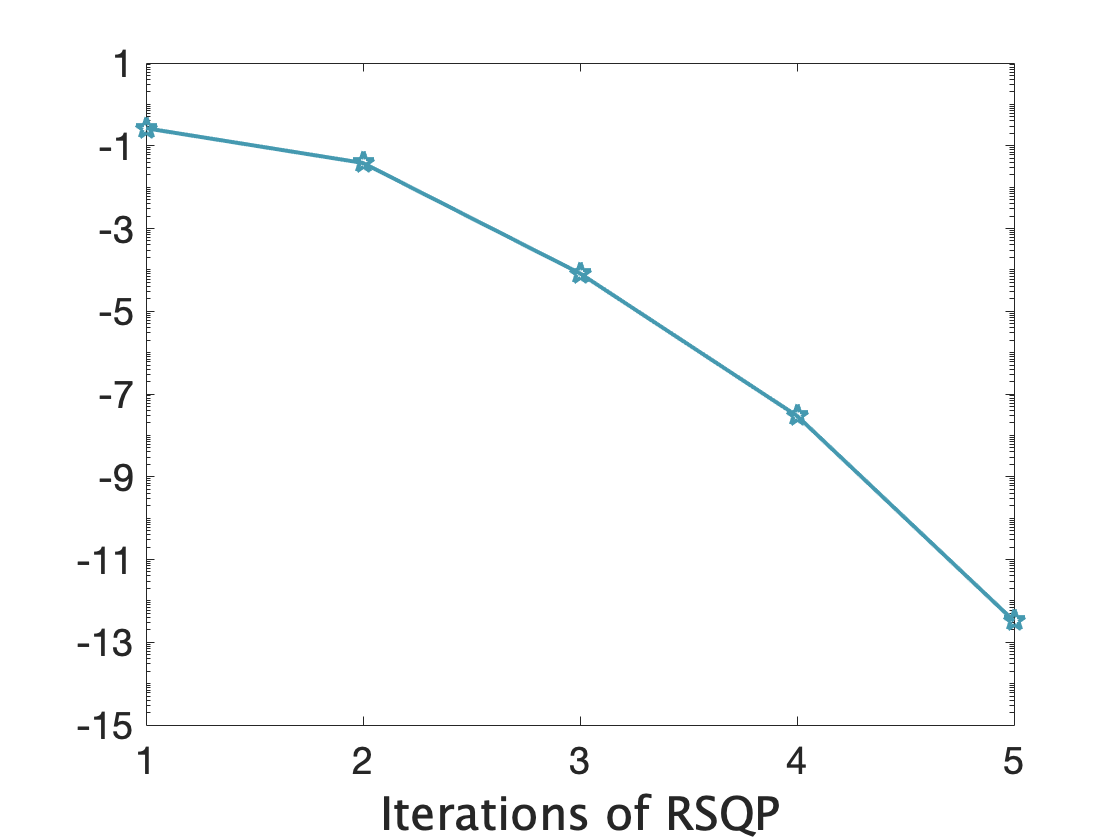}}
	\caption{Superlinear decay of the KKT residuals 
    (see \eqref{kktresidual}) across RSQP iterations (Algorithm~\ref{alg:rsqp}) for the RMC problem \eqref{eq:rmc-full}: (a) toy instance; (b-d) random instances of varying dimensions.}
	\label{fig:rmcall}
\end{figure}
\begin{table}[htbp]
	\centering
	\caption{KKT residuals 
    (see \eqref{kktresidual}) and prediction errors $\|X^k-A\|$ across RSQP iterations (Algorithm~\ref{alg:rsqp}) for random instances of the RMC problem \eqref{eq:rmc-full}.}
	\label{table:rsqp-vertical}
	\begin{tabular}{c c S[table-format=1.2e-1] S[table-format=1.2e-1]}
		\toprule
		Problem size & Iteration $k$ & {KKT residuals} & {Errors $\|X^k-A\|$} \\
		\midrule
		\multirow{6}{*}{$m=n=50,\,r=6$}
		& 1 & 1.14e-1 & 1.39e-1 \\
		& 2 & 2.14e-2 & 1.49e-2 \\
		& 3 & 1.42e-3 & 6.57e-4 \\
		& 4 & 4.04e-6 & 1.78e-6 \\
		& 5 & 2.23e-10 & 5.91e-10 \\
		& 6 & 1.19e-14 & 6.18e-14 \\
		\midrule
		\multirow{4}{*}{$m=n=100,\,r=8$}
		& 1 & 2.76e-2 & 7.79e-2 \\
		& 2 & 2.15e-4 & 2.54e-4 \\
		& 3 & 1.75e-8 & 1.24e-8 \\
		& 4 & 8.62e-14 & 5.52e-14 \\
		\midrule
		\multirow{5}{*}{$m=n=200,\,r=10$}
		& 1 & 2.65e-1 & 1.69e-1 \\
		& 2 & 3.84e-2 & 5.01e-2 \\
		& 3 & 8.23e-5 & 3.95e-5 \\
		& 4 & 2.97e-8 & 1.06e-8 \\
		& 5 & 3.26e-13 & 1.69e-12 \\
		\bottomrule
	\end{tabular}
\end{table}

We apply RSQP (Algorithm~\ref{alg:rsqp}) for solving this toy instance of  the RMC problem \eqref{eq:rmc-full}. Theoretically,  by  Proposition \ref{pro:poly-delta-hemi2}, the  M-SOSC and M-SRCQ conditions ensure the superlinear convergence rate of RSQP (Algorithm~\ref{alg:rsqp}). This is confirmed numerically in Figure~\ref{fig:rmcall}(a), which displays the superlinear decay of the KKT residuals 
defined in \eqref{kktresidual}.

We further test larger problem sizes with dimensions $(m,n,r) = (50,50,6)$, $(100,100,8)$, and $(200,200,10)$.   The random matrix $A\in \mathbb{R}^{m\times n}$ is generated as $A=LR^{\top}$, where $L\in \mathbb{R}^{m\times r}$ and $R\in \mathbb{R}^{n\times r}$ have entries independently sampled from the uniform distribution on $[0,1]$. The observation set $\Omega$ is formed by including each index independently with probability $p=0.3$.
Algorithm~\ref{alg:rsqp} is implemented  based on a slightly modified version of the method in \cite{OOT22}, where each subproblem is solved under an orthonormal basis of the tangent space.  Since Algorithm~\ref{alg:rsqp} does not include a globalization strategy, we use a single iteration of RALM to produce the initial point before applying RSQP. We report the averaged performance over 10 independent replications of each experiment.

Figure~\ref{fig:rmcall}(b-d) demonstrates the superlinear decay of the KKT residuals 
defined in \eqref{kktresidual}. Correspondingly,
Table~\ref{table:rsqp-vertical} reports both the KKT residuals and the prediction errors of the iterates from the true matrix $\|X^k-A\|$.
Although it is difficult to verify the M-SOSC and M-SRCQ conditions in these instances, the results in Figure~\ref{fig:rmcall} and Table~\ref{table:rsqp-vertical} provide clear empirical evidence of superlinear convergence. Moreover, Table~\ref{table:rsqp-vertical} shows that the iterates $X^k$  recover the truth  $A$ almost  exactly, with errors between $10^{-14}$ and $10^{-12}$. This high accuracy is likely due to the fact that no noise is added to the low-rank true matrix $A$ in these experiments.

\subsection{Solving compressed modes by  RALM}

We apply RALM (Algorithm~\ref{alg:ralm}) to the  compressed modes (CM) problem proposed in \cite{OLCO13}.
Let $f(X)=\operatorname{tr}(X^{\top}HX)$, $\theta(\cdot)=\mu\|\cdot\|_1$, $g(X)=X$, and $\mathcal{M}= \operatorname{St}(n,r) :=\{X\in \mathbb{R}^{n \times r}\mid X^{\top}X=I_r\}$,
Let $H$ be a discretization of the Hamilton operator. Then the CM problem can be formulated as
\begin{equation}\label{eq:cm}
\min \  \operatorname{tr}(X^{\top}HX)+\mu\left\|X\right\|_{1} \quad \text { s.t. } \quad  X \in \operatorname{St}(n,r).
\end{equation}
The tangent and normal spaces of $\mathcal{ M }$ at a point $X$ are, respectively, given by
\begin{equation}\label{eq:stiefel-tangent}
\begin{aligned}
    T_X \mathcal{ M }&=\left\{\xi \in \mathbb{R}^{n \times r}\mid X^{\top}\xi+\xi^{\top}X=0\right\},\\
    N_X \mathcal{M}&=\left\{ XA \mid A=A^{\top} \right\},
\end{aligned}
\end{equation}
and the the projection of a matrix $Y\in\mathbb{R}^{n\times r}$ onto the tangent space $T_X \mathcal{ M }$ is given by $\Pi_X (Y)=Y-X\operatorname{sym}(X^{\top}Y)$, where 
$
\operatorname{sym}(Z) := (Z + Z^{\top} )/2
$ for any $Z$.
The Lagrangian of (\ref{eq:cm}) is given by $L(X,y) =\operatorname{tr}(X^{\top}HX) + \langle X, y \rangle$. And its KKT conditions are
\begin{equation}\label{eq:cm-kkt}
\Pi_{X}( 2HX+y )=0, \quad
y \in \mu\partial\|X\|_{1}.
\end{equation}
The Euclidean gradient of $L(X,y)$ is $\nabla_{X} L(X,y)= 2HX+y$. And the Euclidean Hessian  of  $L(X,y)$ is $\nabla_{XX}^2 L(X,y)\xi= 2H\xi$ for any $\xi\in T_X\mathcal{ M }$. Therefore, by \cite{HLWY20}, the Riemannian gradient and Hessian of $L(X,y)$ can be computed as
$$
\operatorname{grad}_{X}L(X,y) = \Pi_{X}( 2HX+y ),\quad
\operatorname{Hess}_{X} L(X,y)\xi =  \Pi_{X}(2H\xi-\xi \operatorname{sym}(X^{\top}\nabla_{X}L(X,y) ).
$$
By the KKT condition (\ref{eq:cm-kkt}), we have $2HX+y\in N_X\mathcal{ M}$, which implies that  $2HX+y=XS$ for some symmetric matrix $S$. Therefore, for any $\xi\in T_{X}\mathcal{ M }$, 
\begin{equation*}
	\begin{aligned}
	 \langle \xi, \operatorname{Hess}_{X}L(X,y)\xi\rangle
	&= \langle \xi,2H\xi-\xi \operatorname{sym}(X^{\top}\nabla_{X}L(X,y)  \rangle \\
	&= \langle \xi,2H\xi  \rangle - \langle \xi,2\xi X^{\top}HX \rangle - \langle \xi,\xi \operatorname{sym}(X^{\top}y)  \rangle\\
    &=  \langle \xi,2H\xi  \rangle - \langle \xi,2\xi X^{\top}HX \rangle - \langle \xi,\xi \operatorname{sym} ( X^{\top}(XS-2HX) )  \rangle \\
&=\operatorname{tr}(\xi^{\top}H\xi)-\operatorname{tr}(\xi^{\top}\xi S).
	\end{aligned}
\end{equation*}
We further obtain that $\mathcal{ C}_{\theta,g}(X,y)=\{d\in \mathbb{R}^{n\times r}\mid\theta^{\downarrow}(X;d)= \langle d,y  \rangle\}$, in which
$$
{\theta^{\downarrow}(X ; d)}= \sum_{X_{ij}=0}|d_{ij}|+\sum_{X_{ij}>0}d_{ij}-\sum_{X_{ij}<0}d_{ij}.
$$
Its affine hull is 
\begin{equation}\label{aff_c_mc}
\operatorname{aff}\mathcal{ C}_{\theta,g}(X,y)= \{ d\in \mathbb{R}^{n\times r}\mid d_{ij}=0\;\text{if}\; y_{ij}\neq \pm \mu ,\;d_{ij}\in \mathbb{R} \;\text{if}\; y_{ij}=\pm \mu \}.
\end{equation}
Therefore, the M-SSOSC for  problem \eqref{eq:cm} holds at the KKT pair $(\betterbar{X}, \bar{y})=(\betterbar{X}, -2H\betterbar{X}+\betterbar{X}\bar{S})$ if
$$
\operatorname{tr}(\xi^{\top}H\xi)-\operatorname{tr}(\xi^{\top}\xi \bar{S})>0, \mbox{ for any }  \xi\in T_{\betterbar{X}}\mathcal{M} \mbox{ satisfying } \xi\in \operatorname{aff}\mathcal{C}_{\theta, g}(\betterbar{X}, \bar{y})\backslash\{0\}.
$$

In \cite{ZBDZ21}, the authors formulate the CM problem in the context of solving the Schrödinger equation for a one-dimensional free-electron model with periodic boundary conditions:  
\begin{equation*}
	-\tfrac{1}{2} \Delta \phi(x) = \lambda \phi(x), \quad x \in [0,50].
\end{equation*}
Their numerical results indicate that the smallest eigenvalue of $H_k \in \partial \operatorname{grad}_{X} L^{\rho_{k}}(X^k,y^k)$ remains strictly positive, suggesting that the M-SSOSC condition may hold in this setting. To further illustrate this conjecture, we present a simple example below.
Consider the Schrödinger equation on the domain $[0,2]$ with periodic boundary conditions. We discretize $[0,2]$ into $n=4$ grid points and denote $H$ by the discrete approximation of $-\tfrac{1}{2}\Delta$. In this case, $H$ takes the form  
$
H = -\begin{bmatrix}
-4 & 2 & 0 & 2 \\
2 & -4 & 2 & 0 \\
0 & 2 & -4 & 2 \\
2 & 0 & 2 & -4
\end{bmatrix}.
$
We take $\mu<5\sqrt{2}$, $r=2$, 
$
\betterbar{X}=\begin{bmatrix}
0&0&\sqrt{2}/2&\sqrt{2}/2\\
\sqrt{2}/2&\sqrt{2}/2&0&0
\end{bmatrix}^{\top}, \mbox{ and }
\bar{y}=\mu\begin{bmatrix}
0&0&1&1\\
1&1&0&0
\end{bmatrix}^{\top}.
$
It can be verified that $(\betterbar{X},\bar{y})$ satisfies the KKT conditions \eqref{eq:cm-kkt}.
And for $\bar{S}=\begin{bmatrix}
-4+\sqrt{2}\mu&4\\
4&-4+\sqrt{2}\mu
\end{bmatrix}$, it holds that $2H\betterbar{X}+\bar{y}=\betterbar{X}\bar{S}$.
Next we show that the M-SSOSC condition holds at  $(\betterbar{X},\bar{y})$. By \eqref{aff_c_mc}, if follows that 
$$
\operatorname{aff}\mathcal{ C}_{\theta,g}(\betterbar{X},\bar{y})=\bigg\{ \begin{bmatrix}
0&0&c_1&c_2\\
c_3&c_4&0&0
\end{bmatrix}^{\top}\Big|\;c_i\in \mathbb{R},i=1,2,3,4 \bigg\}.
$$
For any $\xi\in T_{\betterbar{X}}\mathcal{ M}$  satisfying   $\xi\in \operatorname{aff}\mathcal{ C}_{\theta,g}(\betterbar{X},\bar{y})$, (\ref{eq:stiefel-tangent}) implies that $\xi$ takes the form 
$$
\xi=\begin{bmatrix}
0&0&c_1&-c_1\\
c_2&-c_2&0&0
\end{bmatrix}^{\top},\quad c_1,c_2 \in \mathbb{R},
$$ and thus for $\mu<5\sqrt{2}$,
$
	\operatorname{tr}(\xi^{\top}H\xi)-\operatorname{tr}(\xi^{\top}\xi \bar{S})=12(c_1^2+c_2^2)-2(-4+\sqrt{2}\mu)(c_1^2+c_2^2)=(20-2\sqrt{2}\mu)(c_1^2+c_2^2)>0.
$
Therefore the M-SSOSC holds at $(\betterbar{X},\bar{y})$ given $\mu<5\sqrt{2}$. 

We set $\mu=0.8$ and apply RALM (Algorithm~\ref{alg:ralm}) for solving this toy instance of the CM problem \eqref{eq:cm}. Figure~\ref{fig:cmfigure}(a) illustrates the linear convergence rate of the KKT residuals defined in \eqref{kktresidual}, which confirms the theoretical results in Theorem~\ref{thm:conv-rate-ralm}.

We further evaluate larger problem sizes by discretizing the domain $[0,50]$ into $n$ nodes with $n=200,\,500,\,1000$, setting $r=20$ and $\mu=1$, and taking $H$ as the discretized version of $-\tfrac{1}{2}\Delta$. Algorithm~\ref{alg:ralm} is implemented following  \cite{ZBDZ21}. Figure~\ref{fig:cmfigure}(b-d) demonstrates the linear decay of the KKT residuals based on $10$ independent replications. In parallel, Table~\ref{table:cm} reports the smallest eigenvalues of all generalized Hessians arising in the algorithm, from $5$ independent replications of each experiment. While direct verification of the M-SSOSC condition remains challenging in these instances and thus theoretical guarantees on  eigenvalue positivity are absent, we still observe in Table~\ref{table:cm} that the generalized Hessians arising in the semismooth Newton method  consistently exhibit positive eigenvalues. This might suggest that the M-SSOSC condition may indeed hold in these instances.
\begin{figure}[htbp]
	\centering
    \subfloat[$n=4,r=2$]{\includegraphics[width=0.24\linewidth]{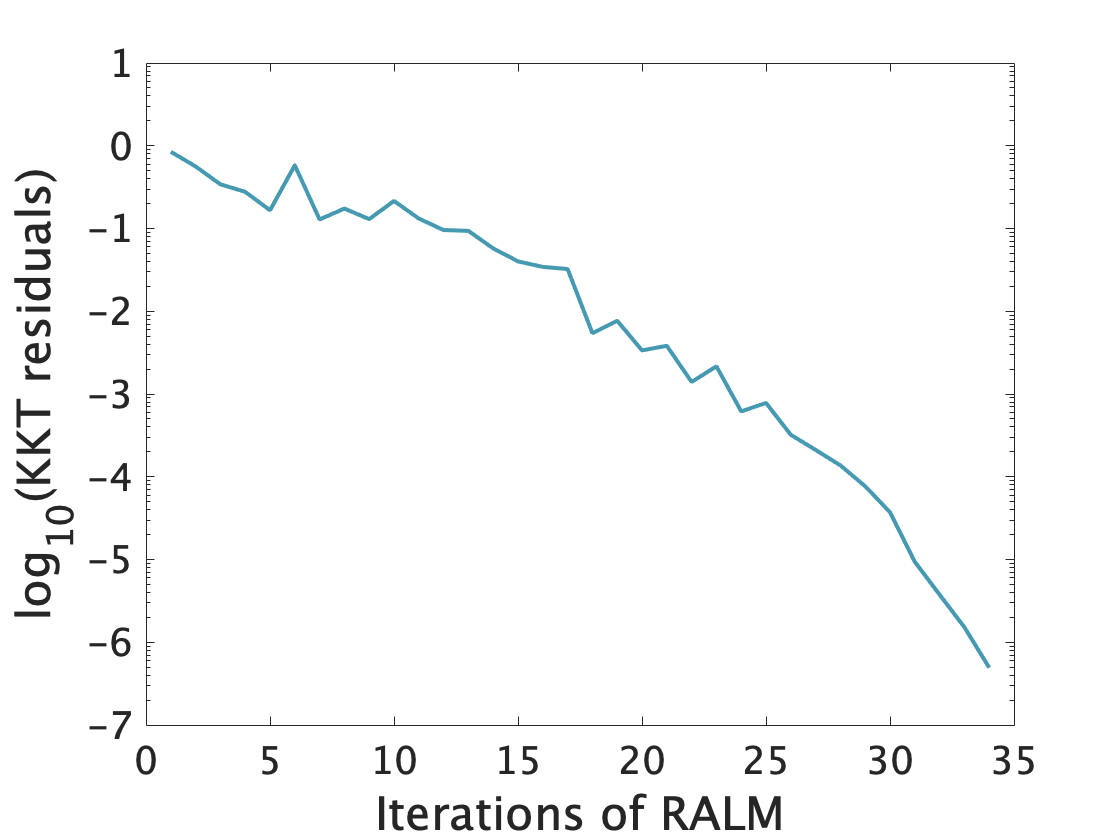}}
	\subfloat[$n=200,r=20$]{\includegraphics[width=0.24\linewidth]{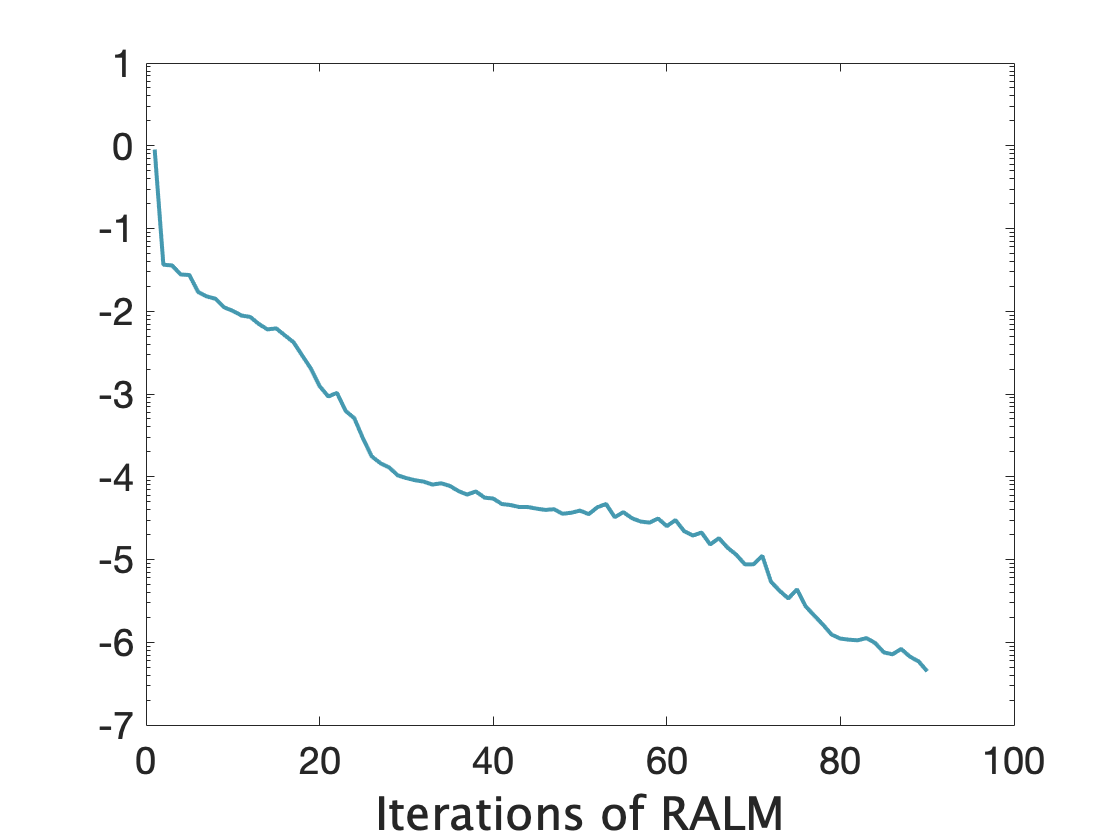}}
	\subfloat[$n=500, r=20$]{\includegraphics[width=0.24\linewidth]{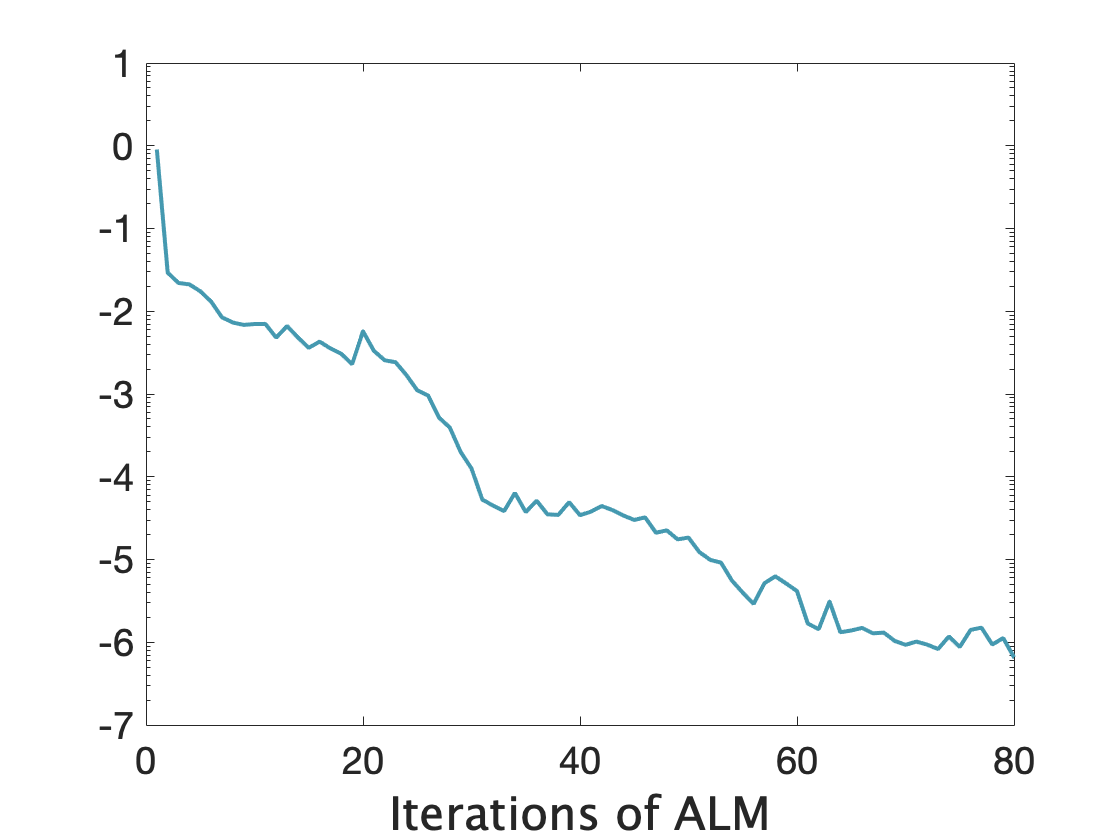}}
		\subfloat[$n=1000, r=20$]{\includegraphics[width=0.24\linewidth]{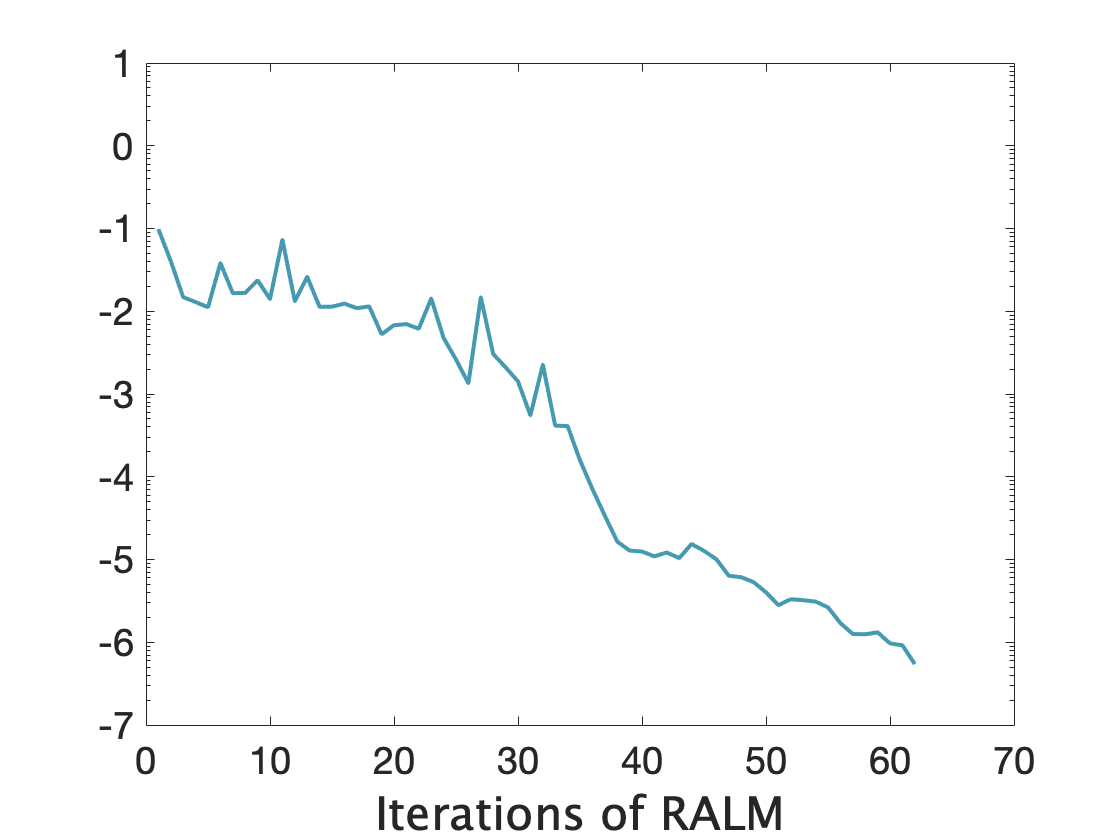}}
	\caption{Linear decay of the KKT residuals 
    (see \eqref{kktresidual}) across RALM iterations (Algorithm~\ref{alg:ralm}) for the CM problem \eqref{eq:cm}: (a) toy instance; (b-d) random instances of varying dimensions.}
	\label{fig:cmfigure}
\end{figure}
\begin{table}[ht]
	\centering
	\caption{Smallest eigenvalues of all generalized Hessians arising in the semismooth Newton based RALM (Algorithm~\ref{alg:ralm}) for random instances of the CM problem \eqref{eq:cm}, based on $5$ independent replications. }
	\label{table:cm}
	\begin{tabular}{c c|c c c c c}
		\toprule
		$n$ &  $r$  & \multicolumn{5}{c}{smallest eigenvalues}\\
		\midrule
		$200$  & $20$ &$0.30$ &$0.32$& $0.19$ &$0.28$ &$0.14$\\
		\midrule
		$500$ & $20$ &$0.28$& $0.11$ & $0.12$ & $0.09$ &$0.11$ \\
		\midrule
		$1000$ & $20$ & $0.03$ & $0.17$& $0.01$ & $0.17$ & $0.21$ \\
		\bottomrule
	\end{tabular}
\end{table}

\section{Conclusion}\label{sec:conclu}
In this paper, we have developed a retraction-based perturbation framework for nonsmooth optimization on connected Riemannian manifolds. By reducing the analysis to locally equivalent tangent space models, we extend fundamental regularity concepts from Euclidean variational analysis to the manifold setting and provide explicit characterizations. 
In particular, we show that the robust isolated calmness of the KKT solution mapping is equivalent to the M-SOSC and M-SRCQ conditions. {Moreover, for $\operatorname{epi} \theta$ being a convex polyhedral set,} a second-order cone, or a semidefinite cone, we establish that the strong regularity of the KKT solution mapping is equivalent to the M-SSOSC and manifold constraint nondegeneracy, which is also equivalent to the nonsingularity of every element in Clarke's generalized Jacobian of the natural mapping. 
We further introduce the notion of the manifold (strong) variational sufficiency, and clarify their role in perturbation analysis. In particular, we show for such $\theta$ the manifold strong variational sufficiency is equivalent to the M-SSOSC. These theoretical developments have facilitated the convergence analysis of nonsmooth optimization algorithms on manifolds. Specifically, under {robust} isolated calmness, the RSQP  achieves local superlinear (quadratic) convergence. Under manifold strong variational sufficiency, the RALM attains R-linear convergence of its outer iterations. Moreover, for the associated inner subproblems, all generalized Hessians arising in the semismooth Newton method are positive definite, thereby guaranteeing superlinear local convergence of the inner iterations. Numerical experiments on the robust matrix completion and compressed modes problem  confirm the predicted convergence rates.

Several challenges remain open. In the robust matrix completion problem, computing the Riemannian Hessian  is computationally expensive in our current implementation.  This motivates the development of Hessian-free or coordinate-free subproblem strategies that can more effectively exploit the underlying manifold structure. In addition, while the PPA for the primal problem and the  ALM for the dual problem are known to be equivalent in the Euclidean setting, their precise relationship under convexity assumptions on Riemannian manifolds remains unclear and warrants further investigation.

\bibliographystyle{plain}
\bibliography{main}
\end{document}